\documentclass[11pt]{article}
\usepackage{amsmath}
\usepackage{amsthm}
\usepackage{amssymb}
\usepackage{graphics}
\usepackage{bm}
\usepackage{setspace}
\usepackage[numbers]{natbib}
\usepackage{authblk}

\usepackage{tikz}
\usepackage{amsfonts}
\usepackage[cal=bickham,calscaled=1.05]{mathalfa}
\usepackage{eucal}
\usepackage{latexsym}
\usepackage{mathdots}
\usepackage{mathrsfs}
\usepackage{mathtools}
\usepackage{hyphenat}
\usepackage{amscd}
\usepackage{enumerate}
\usepackage[pagebackref,hypertexnames=false, colorlinks, citecolor=blue, linkcolor=red, pdfstartview=FitB]{hyperref}
\usepackage[latin1]{inputenc}
\usepackage{tikz}
\usepackage{tikz-cd}
\usetikzlibrary{positioning}

\usepackage{fullpage}
\usepackage{color}

\newcommand{\be}{\begin{equation}}
\newcommand{\ee}{\end{equation}}
\newcommand{\ba}{\begin{eqnarray*}}
\newcommand{\ea}{\end{eqnarray*}}
\newcommand{\bal}{\begin{align}}
\newcommand{\eal}{\end{align}}
\newcommand{\baln}{\begin{align*}}
\newcommand{\ealn}{\end{align*}}
\newcommand{\bi}{\begin{itemize}}
\newcommand{\ei}{\end{itemize}}
\newcommand{\bn}{\begin{enumerate}}
\newcommand{\en}{\end{enumerate}}
\newcommand{\bbm}{\begin{bmatrix}}
\newcommand{\ebm}{\end{bmatrix}}
\newcommand{\bpm}{\begin{pmatrix}}
\newcommand{\epm}{\end{pmatrix}}
\newcommand{\bsm}{\left ( \begin{smallmatrix}}
\newcommand{\esm}{\end{smallmatrix} \right) }
\newcommand{\bp}{\begin{proof}}
\newcommand{\ep}{\end{proof}}

\newcommand{\mr}{\ensuremath{\mathrm}}
\newcommand{\scr}{\ensuremath{\mathscr}}

\newcommand{\ga}{\ensuremath{\gamma}}

\newcommand{\la}{\ensuremath{\lambda }}
\newcommand{\om}{\ensuremath{\omega}}

\def\nbdom{\mathrm{Dom} \, }
\def\nbker{\mathrm{Ker} \, }

\newcommand{\Addresses}{{
  \bigskip

Fouad~Naderi, \textsc{Department of Mathematics, University of Manitoba}\par\nopagebreak
  \textit{E-mail address:} \texttt{naderif@myumanitoba.ca}

\vspace{1cm}

}}

\newcommand{\ip}[2]{\ensuremath{\left\langle {#1} , {#2} \right\rangle}}

\newtheorem{thm}{Theorem}

\newtheorem{lemma}{Lemma}
\newtheorem{prop}{Proposition}
\newtheorem{cor}{Corollary}
\newtheorem*{thm*}{Theorem}

\theoremstyle{definition}
\newtheorem{defn}{Definition}
\newtheorem{remark}{Remark}
\newtheorem{eg}{Example}

\title{Non-commutative Lebesgue decomposition of  non-commutative measures}

\author[]{Fouad Naderi}
\affil[]{\footnotesize University of Manitoba}

\date{\vspace{-1.2cm}}

\begin{document}
\maketitle
\begin{abstract}
   A positive non-commutative (NC) measure is a positive linear functional  on the free disk operator system which is generated by a $d$-tuple of non-commuting isometries. By introducing the hybrid forms, their Cauchy transforms, and techniques from NC reproducing kernel Hilbert spaces (RKHS), we construct a natural Lebesgue decomposition for any positive NC measure against any  other such measure. Our work extends the Jury-Martin decomposition, which originally decomposes positive NC measures  against the standard NC Lebesgue measure. In fact, we give a more generalized definition of absolute continuity and singularity, which reduces to their definition when the splitting measure is the standard NC Lebesgue measure. This generalized definition makes it possible to extend Jury-Martin theory for any splitting NC measure, and it recovers their decomposition when the splitting NC measure is the Lebesgue one. Our work implies a Lebesgue decomposition for representations of the Cuntz-Toeplitz  C*-algebra. Furthermore, our RKHS method gives a new proof of the classical Lebesgue decomposition when applied to the classical one dimensional setting, i.e., $d=1$. 
\end{abstract}
\section{Introduction}
 Classically,  any positive, finite and regular Borel measure on the complex unit circle can be viewed as a positive linear functional on the space of continuous functions on the unit circle by the Riesz-Markov theorem. By the Weierstrass approximation theorem, the continuous functions on the unit circle can be obtained as the norm-closure of the disk algebra and its conjugate. Here, the disk algebra is the unital Banach algebra of analytic functions in the disk which extend continuously to the boundary. Disk algebra can be identified with the unital norm closed operator algebra generated by the forward shift, $S=M_{z}$, the isometry of multiplication by $z$ on the Hardy space $H^{2}$, the Hilbert space of square-summable Taylor series in the complex unit disk. 
 
 In the classical measure theory, Lebesgue theorem tells us if we have two positive finite Borel measures $\mu$ and $\la$, then one can  decompose $\mu$  into absolutely continuous and singular parts with respect to $\la$. The aim of the current work is to pursue the Lebesgue decomposition in several non-commuting variables, i.e., when one has several different copies of the shift operator which do not commute with each other. By using the techniques of positive quadratic forms and reproducing kernel Hilbert spaces (RKHS), we will first translate the Lebesgue decomposition into the language of operator theory and operator algebras,   and then we use the language of non-commutative function theory to further generalize our Lebesgue decomposition to the case when several non-commutative versions of the shift operator are present. Interestingly, when one translates back our results to the classical single variable setting, they find a new proof of classical Lebesgue decomposition based on RKHS theory.

In recent years several authors have shown some versions of  Lebesgue decomposition of one NC measure with respect to a standard NC Lebesgue measure, see \cite{ArvI}, \cite{DLP-ncld}, \cite{GK-ncld}, and \cite{JM-ncld}. Some of these studies have been implemented on NC measures on C*-algebra, see  \cite{ArvI} and \cite{GK-ncld}. However, when the domain of these measures is  the bonafide generalization of continuous functions, i.e., the free disk operator system, the Lebesgue decomposition is more challenging.  For measures living on the free disk operator system, Jury and Martin in \cite{JM-ncld} put forth a theory for Lebesgue decomposition of an NC measure with respect to the canonical NC Lebesgue measure. Nevertheless, they have not explored how to generalize their construction to a pair of arbitrary NC measures, i.e., when one intends to split an arbitrary NC measure $\mu$ with respect to another arbitrary splitting NC measure $\la$.  Their method, however, can be generalized to the case when the splitting measure $\la$ is similar in structure to the Lebesgue one, i.e., when $\la$ is a non-Cuntz measure. Here, a Cuntz measure is the one whose  GNS representation is surjective.  When the splitting measure is Cuntz, we have to generalize Jury-Martin theory so that our generalization recovers their case as well. It turns out their definitions somehow uses the classical properties of the measures, so we have to generalize their definitions to construct our Lebesgue decomposition. We give a counter example that does not fall within the scope of their theory for splitting Cuntz measures, and based on it, we give our suggested generalization. 

There are four approaches to the Lebesgue decomposition:
\begin{enumerate}
    \item The classical approach, which has been generalized in \cite{GK-ncld} by Gheondea-Kavruk  in its utmost generality for measures living on C*-algebras. We call this approach as GKLD standing for  Gheondea-Kavruk-Lebesgue decomposition.  
    \item The RKHS approach for measures living on a special and important operator system. This theory was first discussed by Jury-Martin in the NC setting in the special case when the splitting  measure is the standard Lebesgue measure. We will refine and generalize their theory here  to accommodate any splitting measure. 
    \item Form approach which is based on von Neumann method using inner product spaces, as generalized by Simon \cite{Simon1}. This was also generalized by Jury-Martin to the NC setting in the special case when the splitting  measure is the standard NC Lebesgue measure. We will generalize it here for any splitting measure; however, there is an inclusion condition on the isotripic left ideals of zero length elements, which at the moment we are not able to remove.  
    \item The weak*-approach, which was pursued in \cite{DLP-ncld} and  \cite{DY-freesemi}, and generalized by Jury-Martin  
 in \cite{JM-ncld} to the NC setting in the special case when the splitting  measure is the standard Lebesgue measure. The generalized version of this theory is under investigation by the current authors. We mainly interested in this approach because it helps us to prove some cone-like and hereditary property of absolutely continuous and singular measures, while proving such properties in RKHS and form approaches is difficult.
\end{enumerate}
The following simple diagram shows the equivalence among these theories for the classical measures, and such an equivalence is a desired one in our non-commutative setting. 

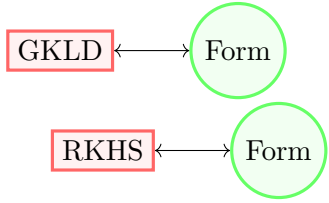
\begin{figure}[ht]
    \centering
\begin{tikzpicture}[
        roundnode/.style={circle, draw=green!60, fill=green!5, very thick, minimum size=7mm},
        squarednode/.style={rectangle, draw=red!60, fill=red!5, very thick, minimum size=5mm},
    ]
    \node[roundnode] (Form) {Form};
    \node[squarednode] (GKLD) [left=of Form] {GKLD};

    \draw[<->] (GKLD.east) -- (Form.west);
    \end{tikzpicture}
    \hspace{1cm} 
    
    \begin{tikzpicture}[
        roundnode/.style={circle, draw=green!60, fill=green!5, very thick, minimum size=7mm},
        squarednode/.style={rectangle, draw=red!60, fill=red!5, very thick, minimum size=5mm},
    ]
    \node[roundnode] (Form) {Form};
    \node[squarednode] (RKHS) [left=of Form] {RKHS};

    \draw[<->] (GKLD.east) -- (Form.west);
    \end{tikzpicture}
    
    \caption{relationship among various theories}
\end{figure}

That $(\text{GKLD} \Longleftrightarrow \text{Form})$ was proved in \cite{GK-ncld} for measures living on C*-algebras. In this paper, we show $(\text{Form} \Longleftrightarrow \text{RKHS})$  for any splitting NC measure, not just the standard NC Lebesgue measure. The common conjunction, forms here, might be misleading since the domain of them for each equivalence might differ, so we cannot say for sure that $(\text{GKLD} \Longleftrightarrow \text{RKHS})$, see Remark \ref{GK-MN-relation} for further explanations. We should mention that our starting definitions of absolute continuity and singularity is different from Gheondea-Kavruk's definitions, and it extends Jury-Martin definitions \cite{JM-ncld}.  Furthermore, the Gheondea-Kavruk theory does not apply for splitting non-Cuntz measures on the disk operator system, whereas our theory remains effective and also implies the Jury-Martin theory.  

Like Jury-Martin in \cite{JM-ncld} and \cite{JM-ncFatou}, we carry out  the Lebesgue decomposition in two different ways, i.e.,  the decomposition of RKHS and the decomposition of forms. We will introduce the hybrid forms and their Cauchy transforms to construct the Lebesgue decomposition for splitting Cuntz NC measures.  In addition, we also show some cone-like and hereditary property for the set of positive measures which are absolutely continuous (AC) or singular (sing) with respect to a given fixed NC measure. Here, a sub-cone $\mathcal{S}$  of a positive cone $\mathcal{C}$ is hereditary if whenever $s \in \mathcal{S}$ and $c \in \mathcal{C}$ with $c \leq s$, then $c \in \mathcal{S}$. Our work also develops a Lebesgue decomposition for representations of the Cuntz-Toeplitz C*-algebra. This is because every NC measure comes equipped with a cyclic representation of the Cuntz-Toeplitz algebra via a Gelfand-Naimark-Segal (GNS) construction, and if we allow for operator valued NC measures, then we can obtain every representation of Cuntz-Toeplitz C*-algebra up to unitary equivalence \cite{JM-freeCE}.

Outline of the paper: In section 2, we introduce some preliminaries. In sections 3  we generalize Jury-Martin-Lebesgue decomposition when the splitting measure is non-Cuntz.  In section 4 we introduce hybrid forms, and we show that for the case of non-Cuntz measures, the decomposition of hybrid forms and RKHS are equivalent. Theorems and proofs in sections 3 and 4 are generalizations and modifications of Jury-Martin proofs in \cite{JM-ncFatou} and \cite{JM-ncld}, and sometimes we need to do some adjustments in the proofs. However, for the sake of clarity and future records, we will give the complete proofs of the theorems to show how they work in their ultimate generality. On the other hand, in section 5 we show that Jury-Martin theory \cite{JM-ncld} needs some extension to allow for the splitting Cuntz measures since their theory uses direct sum decomposition of RKHS' and reducing sub-spaces, which is not available in the general form. The direct sum decomposition and reducing subspaces work well for non-Cuntz measures, but for Cuntz measures we should relax them.  So in our final section, we give our refined definitions of absolute continuity and singularity based on Toeplitz property of Radon-Nikodym derivatives and 
combine them with Cauchy transforms of hybrid forms to obtain the maximal Lebesgue decomposition, regardless of the fact that  the splitting measure is Cuntz or not. We abandon Jury-Martin's direct sum decomposition, and instead we use some maximal subspace on which the Radon-Nikodym derivative has Toeplitz property. We  will also explain why our theory implies the Jury-Martin theory for non-Cuntz measures. 
\section{Preliminaries}

Fix a natural number $d \in \mathbb{N}$, which we use it as the dimension of our non-commutative world. By the NC universe, we mean
\begin{equation*}
\mathbb{C}^{d}_{\mathbb{N}}=\bigsqcup_{n=1}^{\infty} \mathbb{C}^{d}_{n} ; \quad  \mathbb{C}^{d}_{n}=\mathbb{C}^{n \times n} \otimes \mathbb{C}^{1 \times d}.
\end{equation*}
In the above, any element of $\mathbb{C}^{d}_{n}=\mathbb{C}^{n \times n} \otimes \mathbb{C}^{1 \times d}$ is a $d$-tuple $Z=(Z_{1}, ..., Z_{d})$ of $n \times n$ matrices with complex entries. When $d=1$, we do not write the super script. There are several norms one can impose on $\mathbb{C}^{d}_{\mathbb{N}}$, but the most important one to us here is the row-norm defined by
\begin{equation*}
\parallel Z \parallel_{row}={\left\lVert \sum_{1}^{d} Z_{j}Z_{j}^{*}   \right\rVert}^{\frac{1}{2}}.
\end{equation*}

\noindent We mainly work with functions on the open unit row ball which is defined by

\begin{equation*}
\mathbb{B}^{d}_{\mathbb{N}}=\bigsqcup_{n=1}^{\infty} \mathbb{B}^{d}_{n} ; \quad  \mathbb{B}^{d}_{n}=\{ Z=(Z_{1}, ..., Z_{d}) \in \mathbb{C}^{d}_{n}~: ~  \parallel Z \parallel_{row} < 1  \}
\end{equation*}

The NC Hardy space, ${\mathbb{H}}^{2}_{d}$, is defined as the set of all formal power series in several non-commuting variables, $Z := ( Z _1, \cdots, Z _d ) \in \mathbb{B}^{d}_{\mathbb{N}}$, with square summable complex coefficients. Namely, any $h \in {\mathbb{H}}^{2}_{d}$ is a power series of the form:
\begin{eqnarray*}
   h(Z ) = \sum _{\omega \in {\mathbb{F}}^{d}} \hat{h}_{\omega} Z ^\omega, 
\end{eqnarray*}
 where $\hat{h}_{\omega}  \in \mathbb{C}$ and 
 \begin{eqnarray*}
     \sum _{\omega \in {\mathbb{F}}^{d}}  | \hat{h} _{\omega} | ^2 < + \infty.
 \end{eqnarray*}
 Here, ${\mathbb{F}}^{d}$ denotes the \emph{free monoid}, the set of all words in the $d$ letters $\{ 1, \cdots , d \}$, with the action defined by the concatenation of words, and the unit empty word denoted by $\emptyset$. The \emph{free monomials}, $Z ^\omega$ are defined as $Z ^\omega = Z _{i_1} \cdots Z_{i_n}$ if $\omega = i_1 \cdots i_n$, $i_j \in \{ 1, \cdots , d \}$ and $Z ^\emptyset := 1$.  Let $f:= \sum _{\omega \in {\mathbb{F}}^{d} } \hat{f_{\omega}} Z^{\omega}$ and $g:= \sum _{\omega \in {\mathbb{F}}^{d} } \hat{g_{\omega}} Z^{\omega}$ be in ${\mathbb{H}}^{2}_{d}$, the inner product of them is given by
 \begin{eqnarray*}
     \langle f , g \rangle = \sum _{\omega \in {\mathbb{F}}^{d} } \overline{ \hat{f_{\omega}}} \hat{g_{\omega}}.
 \end{eqnarray*}
 We should mention that in the sequel, we follow the agreement that  all inner products are conjugate linear in the first argument while linear in the second one. 

Left free shifts are isometries of left multiplication by the independent variables on ${\mathbb{H}}^{2}_{d}$, i.e., $L_k = M_{Z_{k}}^{L}$,  $k=1, \cdots, d$. Since we have $d$ non-commuting variables $Z _1, \cdots, Z _d$, we see that the ranges of $L_k$'s are pairwise orthogonal
\begin{equation} \label{row-isometry-property}
L^{*}_{k}L_{j}=\delta_{k,j} I_{{\mathbb{H}}^{2}_{d}}
\end{equation}
so that the $d$-tuple 
\begin{equation} \label{NC shift}
L:=(L_{1}, ..., L_{d}): {\mathbb{H}}^{2}_{d} \otimes {\mathbb{C}}^{d} \longrightarrow {\mathbb{H}}^{2}_{d}
\end{equation} 
defines a row isometry. One can similarly define right shifts and a row isometry of them.  Popescu's extension of the Wold decomposition \cite{Pop-dil} says $L$ is the universal pure row isometry. Popescu \cite{Pop-dil} shows that any row isometry can be decomposed into two parts, the $L$-part which is a direct sum of several copies of $L$, and the Cuntz part which is a surjective row isometry, i.e., unitary.  By the Popescu-von Neumann inequality for free polynomials \cite{Pop-vN-inq}, one can show that for any $f \in {\mathbb{H}}^{2}_{d}$ and any $Z \in \mathbb{B}^{d}_{\mathbb{N}}$ we have
\begin{equation} \label{pointwise bounded}
\parallel f(Z)\parallel^{2} \leq \frac{\parallel f \parallel_{{\mathbb{H}}^{2}_{d}}^{2} } {1-\parallel Z \parallel_{\mathrm{row}} ^2}.
\end{equation}

For any $n \in \mathbb{N}$, we can put the trace inner product on ${\mathbb{C}}^{n \times n }$ by declaring $\langle A, B \rangle =\mathrm{tr}(A^{*}B)$. Since ${\mathbb{C}}^{n \times n }$ is finite dimensional, the induced trace norm and the matrix operator norm are equivalent on it. For any $Z\in \mathbb{B}^{d}_{\mathbb{N}}$, let $\ell_{Z}: {\mathbb{H}}_{d}^{2} \longrightarrow  {\mathbb{C}}^{n \times n }$ be the point evaluation map defined by $\ell_{Z}(f)=f(Z)$. By inequality \ref{pointwise bounded} and the equivalence of the trace and operator norms on ${\mathbb{C}}^{n \times n }$, we see that $\ell_{Z}$ is a bounded operator, so we can define its Hilbert space adjoint operator by
\begin{eqnarray*}
K_{Z}={\ell_{Z}}^{*}: {\mathbb{C}}^{n \times n } &\longrightarrow& {\mathbb{H}}_{d}^{2} \\
 yv^{*} &\longrightarrow& K_{Z}(yv^{*})=K\{Z, y, v\}, \quad y , v \in \mathbb{C}^{n}  \\
 \langle K\{Z, y, v\} , f \rangle &=& \langle y, f(Z)v \rangle, \quad y , v \in \mathbb{C}^{n}
\end{eqnarray*}
\noindent Now, we can define the NC kernel function \cite{BMV}, for any $n , m \in \mathbb{N}$ by
\begin{eqnarray}
K: \mathbb{B}^{d}_{n} \times \mathbb{B}^{d}_{m} & \longrightarrow &   \mathscr{B} ({\mathbb{C}}^{n \times m }) \nonumber \\
(Z,W) & \longmapsto &  K(Z,W) \nonumber \\
\langle  y, K(Z,W)(vu^{*})x \rangle & = & \langle K\{Z,y,v\} , K\{W,x,u\} \rangle, \label{rkhs_inner_product}
\end{eqnarray}

\noindent where $Z \in {\mathbb{C}}_{n}^{d}, ~ y, v \in {\mathbb{C}}^{n}$  and $W \in {\mathbb{C}}_{m}^{d }, ~ x, u \in {\mathbb{C}}^{m}$. For any $Z$ and $W$, the mapping $K(Z,W)$ is completely bounded, and when $Z=W$ it is completely positive. One can then define the NC reproducing kernel Hilbert space (RKHS)  corresponding to the NC kernel $K$, by
\begin{equation*}
\mathcal{H}_{nc} (K)=\bigvee_{Z, y, v} K\{Z,y,v\}, 
\end{equation*}
\noindent with the inner product given by \ref{rkhs_inner_product}. For example, for the NC Szego kernel 
$$ K(Z,W)[\cdot]=\sum_{\alpha \in \mathbb{F}^{d}} Z^{\alpha} [\cdot] W^{\alpha *},$$ 
one has
$\mathcal{H}_{nc} (K)={\mathbb{H}}_{d}^{2}$, see \cite{BMV} and \cite{JM-ncld}. The construction of NC RKHS's parallels the classical ones, see \cite[Chapters 1 \& 2]{Paulsen-rkhs}.

Let $\mathbb{D}$ be the open unit ball of the complex numbers and $\mathbb{T}$ be its boundary. Recall that $A(\mathbb{D})$ is the disk algebra, the Banach space of all analytic functions on the open unit disk $\mathbb{D}$ which have continuous extension to the closure of $\mathbb{D}$, and $H^{\infty}(\mathbb{D})$ denotes the Banach space of all bounded analytic functions on the unit disc $\mathbb{D}$. In addition, the Toeplitz C*-algebra $\mathcal{E}$ is the universal C*-algebra generated by a non-surjective isometry and the identity operator.  By using the left shift \ref{NC shift}, one can construct the non-commutative multi-variable counterparts of these classical spaces. In the following table,  the notations \emph{op} and \emph{wk*} mean that the closure is being taking in operator norm topology and weak* operator topology on $\mathscr{B}(\mathbb{H}_{d}^{2})$ respectively. We also used $L$ for the family of left shift operators $\{L_1 , \cdots, L_d \}$.

\renewcommand{\arraystretch}{1.5}

\vspace{2mm}
\begin{tabular}{c| c| c}
 Non-commutative & nomenclature &  counterpart of classical \\
\hline 
$\mathbb{A}_{d}={\overline{Alg[I, L]}}^{op}$ & free disk algebra & $A(\mathbb{D})$ \\
$\mathbb{L}_{d}^{\infty} = {\overline{\mathbb{A}_{d}}}^{wk*}$ & free Hardy algebra &   $H^{\infty}(\mathbb{D})$\\
$\mathcal{A}_{d}={\overline{\mathbb{A}_{d} + \mathbb{A}^{*}_{d} }}^{op}$ & free disk operator system & $C(\mathbb{T})$ \\ 
$\mathcal{L}_{d}^{\infty}={\overline{\mathcal{A}_{d}}}^{wk*}$ & free Toeplitz system & $L^{\infty}(\mathbb{T})$ \\

 $ \mathcal{E}_{d}=C^{*}\{I, L\}$ & Cuntz-Toeplitz C*-algebra & $\mathcal{E}$
\end{tabular}
\vspace{2mm}

The  Cuntz-Toeplitz C*-algebra, $ \mathcal{E}_{d}$,  was first introduced and studied by Cuntz \cite{Cuntz}, and it has been a main object in the study and the development of the non-commutative Hardy spaces and dilation theory of row contractions. A very important property of the NC disk algebra is the \emph{semi-Dirichlet property} expressed by the following inclusion,
\begin{equation} \label{semi-Dirichlet}
 {\mathbb{A}_{d}}^{*} \mathbb{A}_{d} \subseteq   {\overline{\mathbb{A}^{*}_{d} +\mathbb{A}_{d}  }}^{op}=\mathcal{A}_{d},
\end{equation}
which is an application of the equation \ref{row-isometry-property}.

Let $(\mathcal{A}_{d})_{+}^{\dagger}$ denote the set of all positive linear functionals on the operator system $\mathcal{A}_{d}$. We can think of $(\mathcal{A}_{d})_{+}^{\dagger}$ as the set of all positive NC measures as the following explanation shows. Let $M(\mathbb{T})$ denote the classical Banach space of all bounded complex regular Borel measures under the total variation norm on the unit circle. Since in the one variable setting $\mathcal{A}_{1} \cong C(\mathbb{T})$,   by the analogy with the classical Riesz-Markov Theorem $M(\mathbb{T}) \cong C(\mathbb{T})^{\dagger}$, we can consider  $(\mathcal{A}_{d})_{+}^{\dagger}$ as the set of all positive NC measures. However, we should keep in mind that the NC measures we consider live on an operator system when $d \geq 2$, so working with them are tricky and sometimes we need other properties like semi-Dirichlet property \ref{semi-Dirichlet} to deal with them. 

We usually specify an NC measure by its moments, and in doing so we  need the semi-Dirichlet property \ref{semi-Dirichlet}.  One important NC measure on $\mathcal{A}_{d}$ is the standard NC Lebesgue measure whose moments are given by 
\begin{equation*}
    m(L^{\alpha})= \langle 1, L^{\alpha} 1 \rangle =\delta_{\alpha, \emptyset}.
\end{equation*}
\noindent This definition is plausible since in the classical case the moments of the Lebesgue measure are similar to the ones above. Also, in the classical case the Herglotz transformation takes the Lebesgue measure to the constant unit function, and we have a similar situation by applying the NC Herglotz transform, defined below, against the NC Lebesgue measure. What is more, it even complies with the classical fact that it is the Clark measure  of the zero function. 

We now define two important transformations which are analogous to  classical ones and are important in the theory of Hardy spaces. Assume that $id_{n}$ is the identity map on ${\mathbb{C}}^{n \times n}$ and $I_{n}$ is the identity matrix in it. Let $I_{{\mathbb{H}}_{d}^{2}}$ be the identity map on ${\mathbb{H}}_{d}^{2}$ as well, then the NC Herglotz transformation is defined by
\begin{eqnarray*}
  H: (\mathcal{A}_{d})_{+}^{\dagger} &\longrightarrow & \mathscr{L}_{d}^{+} \\
 \mu &\longmapsto & H_{\mu}, \\
 H_{\mu}(Z) &=& id_{n} \otimes \mu [(I_{n} \otimes I_{{\mathbb{H}}_{d}^{2}} -Z \otimes L^{*})^{-1} (I_{n} \otimes I_{{\mathbb{H}}_{d}^{2}}+Z \otimes L^{*})],
\end{eqnarray*}
\noindent where $Z\in \mathbb{B}^{d}_{n}$ and  $Z \otimes L^{*} =\sum_{k=1}^{d} Z_{k}\otimes L_{k}^{*}$. Here, $\mathscr{L}_{d}^{+}$ denotes the positive cone of locally bounded NC holomorphic functions on the NC unit ball with positive semi definite real part. Corresponding to the NC holomorphic function $H_{\mu}$ and the NC Szego kernel $K$, there is an NC kernel $\prescript{\mu}{}{K}$ defined by
\begin{equation} \label{mu-kernel}
\prescript{\mu}{}{K}(Z,W)[\cdot]=\frac{1}{2}[H_{\mu}(Z)K(Z,W)[\cdot]+K(Z,W)[\cdot]H_{\mu}(W)^{*}],
\end{equation}
The NC RKHS generated by $\prescript{\mu}{}{K}$ is $\mathscr{H}^{+}(H_{\mu})$, i.e., $\mathcal{H}_{nc}(\prescript{\mu}{}{K})=\mathscr{H}^{+}(H_{\mu})$, see \cite{JM-ncFatou} and \cite{BMV}. For the Lebesgue NC measure $m$, we have $\mathscr{H}^{+}(H_{m})=\mathbb{H}_{d}^{2}$.

For any NC positive measure $\mu \in (\mathcal{A}_{d})_{+}^{\dagger} $, we can mimic the GNS construction for the pair $({\mathbb{A}}_{d} , \mu)$, 
[see \cite{JM-ncFatou}], to construct a Hilbert space denoted by $\mathbb{H}_{d}^{2}(\mu)$, which takes the role of the classical Hardy space $H^{2}(\mu) \subseteq L^{2}(\mu)$ in the NC setting. Roughly speaking, the classical $H^{2}(\mu)$ is the closed linear span of all analytic polynomials in $L^{2}(\mu , \mathbb{T})$. To start with our GNS construction, we can use the semi-Dirichlet property \ref{semi-Dirichlet} to define a  pre-inner product on $\mathbb{A}_{d}$  by

\begin{equation*}
    \langle a , b \rangle_{\mu}=\mu(a^{*}b), \quad a, b \in \mathbb{A}_{d}.
\end{equation*}

The GNS space $\mathbb{H}_{d}^{2}(\mu)$ is then the Hilbert space completion of $\mathbb{A}_{d}$  modulo zero length vectors with respect to the above inner product. The equivalence class of $a\in \mathbb{A}_{d}$ is denoted by $a+N_{\mu}$, where $N_{\mu} =\{ a\in \mathbb{A}_{d} : \mu(a^{*}a)=0 \} $ is the left ideal of all elements of zero length. Moreover, we have the left regular representation,
\begin{eqnarray}
\pi_{\mu}: \mathbb{A}_{d} &\longrightarrow & \mathscr{B}(\mathbb{H}_{d}^{2}(\mu)) \nonumber \\
a &\longrightarrow& \pi_{\mu} (a); \quad  \pi_{\mu} (a)(b+N_{\mu})  =  ab+N_{\mu}. \nonumber 
\end{eqnarray}

\noindent  This representation extends naturally to a $*$-representation of the Cuntz-Toeplitz C*-algebra $\mathcal{E}_d$, and in this way one can obtain, up to  unitary equivalence, all the cyclic $*$-representations of  $\mathcal{E}_{d}$, \cite[Lemma 2.2]{JMT-NCFM}. In fact, we can obtain all $*$-representations if we use operator valued positive NC measures \cite{JM-freeCE}. Now,  corresponding to the $d$ left multiplications $\pi_{\mu} (L_{i})$, $i=1,2,...,d$,  we can define the row isometry $\Pi_{\mu}= \pi_{\mu}(L)$ by letting $\Pi_{\mu , i} = \pi_{\mu} (L_{i})$ , and note that
\begin{eqnarray*}
  &\Pi_{\mu}&   : \mathbb{H}_{d}^{2}(\mu) \otimes {\mathbb{C}}^{d} \longrightarrow \mathbb{H}_{d}^{2}(\mu)\\
  &\Pi_{\mu}&=(\Pi_{\mu , i})_{i=1}^{d}=(\pi_{\mu} (L_{i}))_{i=1}^{d}.
\end{eqnarray*}
It is important to note that $\Pi_{\mu , i}$'s act basically by left multiplication on $\frac{\mathbb{A}_d}{N_{\mu}}$. Notice that $\pi_\mu$ is a representation and $\Pi_\mu$ is the corresponding row isometry. If $\Pi_{\mu}$ is surjective; hence, a unitary, then  $\mu$ is called a Cuntz measure. However, if $\Pi_\mu$ is a non-surjective isometry, $\mu $ is called a non-Cuntz measure. For a non-Cuntz measur $\mu$, we will later show that $N_\mu=0$. So, for a classical non-Cuntz measure $\mu$ on the unit circle, applying the GNS procedure retrieves $H^{2}(\mu)$. Thus, we see for a classical non-Cuntz measure $\mu$, the Hardy space $H^{2}(\mu)$ is the closed linear span of analytic polynomials. Interestingly, when $\mu$ is a classical Cuntz measure, we will see that $H^2(\mu)=L^2(\mu)$. Furthermore, for the standard NC Lebesgue measure $m$, one has $N_{m}=0$, $\mathbb{H}_{d}^{2}(\mu)=\mathbb{H}_{d}^{2}$ and $\Pi_m = L$.

Define the NC Cauchy transform generated  by $\mu$ by
\begin{eqnarray*}
 \mathscr{C}_{\mu}:\mathbb{H}_{d}^{2}(\mu) &\longrightarrow & \mathscr{H}^{+}(H_{\mu}) \\
  p & \longmapsto & \mathscr{C}_{\mu}(p),\\
   \mathscr{C}_{\mu}(p)(Z) & =& id_{n} \otimes \mu [(I_{n} \otimes I_{{\mathbb{H}}_{d}^{2}} -Z \otimes L^{*})^{-1} (I_{n} \otimes p(L))],
\end{eqnarray*}
\noindent  where $Z\in \mathbb{B}^{d}_{n}$ and  $Z \otimes L^{*} =\sum_{k=1}^{d} Z_{k}\otimes L_{k}^{*}$, and other notations are like the NC Herglotz transformation. Also note that $H_{\mu}(Z) = 2~ \mathscr{C}_{\mu}(1+N_{\mu})( Z) -\mu(I_{{\mathbb{H}}^{2}}) I_{n}$, like the classical theory. Besides, one can show
\begin{equation*}
   \left( \mathscr{C}_{\mu}p \right) (Z) =\sum_{\alpha} Z^{\alpha} \mu (L^{\alpha *} p(L) ).
\end{equation*}

\noindent Recall that $\mathcal{H}_{nc}(\prescript{\mu}{}{K})=\mathscr{H}^{+}(H_{\mu})$, so we define the following row isometry between the RKHS  which is unitary equivalent to the row isometry $\Pi_{\mu}$ between GNS spaces:

$$V_{\mu}=(V_{\mu , k})_{k=1}^{d}:\mathscr{H}^{+}(H_{\mu}) \otimes {\mathbb{C}}^{d} \longrightarrow \mathscr{H}^{+}(H_{\mu})$$ 
$$V_{\mu , k}= \mathscr{C}_{\mu} \Pi_{\mu, k} {\mathscr{C}_{\mu}}^{*}$$
This formula can be translated into the language of intettwiners by knowing  that the Cauchy transform $\mathscr{C}_{\mu}$ is a unitary operator hence an intertwiner for $\Pi_{\mu}$ and $V_{\mu}$, i.e., $\mathscr{C}_{\mu} \Pi_{\mu,k}= V_{\mu, k} \mathscr{C}_{\mu} $ with the following commutative diagram

\begin{center}
\begin{tikzcd}
\mathbb{H}_{d}^{2}(\mu) \arrow[r, "\mathscr{C}_{\mu}"]
& \mathscr{H}^{+}(H_{\mu}) \\
\mathbb{H}_{d}^{2}(\mu) \arrow[r, "\mathscr{C}_{\mu}"] \arrow[u, "\Pi_{\mu, k}"]
& \mathscr{H}^{+}(H_{\mu}) \arrow[u, "V_{\mu , k} "]
\end{tikzcd}
\end{center}

\noindent For the NC Lebesgue measure $m$,  $V_{m}$ is just the left free shift $L$ or $M_{Z}^{L}$, see \cite{JM-ncFatou}.

Now, let $\lambda$ and $\tau$ be two positive NC measures such that $\lambda \leq \tau$, then by NC Aronszjan theorems \cite[Theorems 4.1 and 4.2]{JM-ncld} and \cite{BMV} we have the inclusion among RKHS's, i.e., $\mathscr{H}^{+}(H_{\lambda}) \subseteq \mathscr{H}^{+}(H_{\tau})$. So, we can define the contractive embedding (see \cite[Theorem3]{JM-ncFatou} )
\begin{eqnarray*}
 e_{\lambda}=e_{\tau, \lambda}:\mathscr{H}^{+}(H_{\lambda}) &\longrightarrow & \mathscr{H}^{+}(H_{\tau}) \\
  f & \longmapsto & f.
  \end{eqnarray*}
\noindent In this situation, $N_{\tau} \subseteq N_{\lambda}$, and we can define the contractive co-embedding among GNS spaces ( see \cite[Lemma 3]{JM-ncFatou})
\begin{eqnarray*}
 E_{\lambda}=E_{\lambda, \tau}:\mathbb{H}_{d}^{2}(\tau) &\longrightarrow & \mathbb{H}_{d}^{2}(\lambda) \\
  f+N_{\tau} & \longmapsto & f+N_{\lambda},
  \end{eqnarray*}
\noindent so that $E_{\lambda} = \mathscr{C}_{\la}^{*} e_{\lambda}^{*} \mathscr{C}_{\tau}$ according to the following diagram 

\begin{center}
\begin{tikzcd}
\mathbb{H}_{d}^{2}(\lambda) \arrow[r, "\mathscr{C}_{\la}"]
& \mathscr{H}^{+}(H_{\lambda}) \arrow[d, "e_{\lambda}"] \\
\mathbb{H}_{d}^{2}(\tau) \arrow[r, "\mathscr{C}_{\tau}"] \arrow[u, "E_{\lambda}"]
& \mathscr{H}^{+}(H_{\tau}) 
\end{tikzcd}
\end{center}
Note that by \cite[Lemma 3]{JM-ncFatou}, $||E_\la ||=||e_\la|| \leq 1$. The contractive co-embedding $E_{\lambda}$, by its very definition, is an intertwiner between $\Pi_{\lambda}$ and $\Pi_{\tau}$, i.e., $E_{\lambda} \Pi_{\tau,k}=\Pi_{\la, k} E_{\lambda} $ with the following commutative diagram
\begin{center}
\begin{tikzcd}
\mathbb{H}_{d}^{2}(\tau) \arrow[r, "E_{\lambda}"]
& \mathbb{H}_{d}^{2}(\lambda) \\
\mathbb{H}_{d}^{2}(\tau) \arrow[r, "E_{\lambda}"] \arrow[u, "\Pi_{\tau, k}"]
& \mathbb{H}_{d}^{2}(\lambda) \arrow[u, "\Pi_{\lambda , k} "]
\end{tikzcd}
\end{center}

When the two measures   $\lambda$ and $\tau$ are not comparable but $N_\la \subseteq N_\tau$, we can   define $e_{\lambda, \tau}$ and $E_{\tau, \lambda}$ as  densely defined and closed unbounded operators on suitable domains (see sections 4 and 5), and again $E_{ \tau, \la}$ is an intertwiner.

\section{The Lebesgue reproducing kernel Hilbert space decomposition with respect to a non-Cuntz measure}

In the classical Hardy space theory, absolute continuity of measures on the unit circle $\mathbb{T}$ can be recast in terms of containment of reproducing kernel Hilbert spaces. Namely, given two finite positive regular Borel measures $\mu$ and $\lambda$ on $\mathbb{T}$ recall that $\mu$ is absolutely continuous (AC) with respect to $\lambda$, in notations $\mu \ll \la$, if there is an increasing sequence of finite positive regular Borel measures $\mu_{n}$, which are dominated by $\lambda$, and increasing monotonically to $\mu$:
\begin{eqnarray*}
0 \leq \mu_{n} \leq \mu, \quad \mu_{n} \uparrow \mu, \\
\mu_{n} \leq t_{n}^2 \lambda, \quad t_{n} >0 .
\end{eqnarray*}
\noindent By a classical formula similar to \ref{mu-kernel}, we see that $\prescript{\mu_{n}}{}{K} \leq \prescript{\mu}{}{K}$; hence, by Aronszajn's theorem \cite[Theorem 5.1]{Paulsen-rkhs},  $\mathcal{H}(\prescript{\mu_{n}}{}{K}) \subseteq \mathcal{H}(\prescript{\mu}{}{K})$, and also
\begin{equation*} 
    \bigvee \mathscr{H}^{+}(H_{\mu_{n}}) =\mathscr{H}^{+}(H_{\mu}).
\end{equation*}

\noindent Similarly, for $\mu_{n} \leq t_{n}^2 \lambda$ we see that  $\mathcal{H} (\prescript{\mu_{n}}{}{K}) \subseteq \mathcal{H} (\prescript{\lambda}{}{K})$. Thus, the intersection space 
$$ \mathrm{int}(\mu , \lambda)={\mathscr{H}}^{+}(H_{\mu}) \cap {\mathscr{H}}^{+}(H_{\lambda}) $$ 
\noindent is dense in ${\mathscr{H}}^{+}(H_{\mu})$. When $\mu$ is singular with respect to $\lambda$, their joint minimum is zero, see \cite[Theorem 37.5]{Aliprantis}. So, there is no  absolutely continuous part of $\mu$ with respect to $\la$. Hence, we will assume that the intersection space is just the zero vector. This assumption when $\la$ is non-Cuntz is valid; otherwise, it might be problematic, for example see the Example \ref{classical-sing-not-RKHS-Sing}.  For now, the NC analogue of these observations is summarized in the following definition.

\begin{defn} \label{RK-AC-Sing}
Given $\mu, \la \in \left( \scr{A} _d   \right) ^\dag _{+}$, we say that:
\begin{itemize}
    \item $\mu$ is \emph{absolutely continuous with respect to $\la$ in the reproducing kernel sense}, in notations $\mu \ll_{R} \lambda$, if 
$$ \mathrm{int} (\mu , \la ) := \scr{H} ^+ (H_{\mu} ) \cap \scr{H} ^+ (H_{\la} ), $$ is norm--dense in $\scr{H} ^+ (H_{\mu})$. The set of all positive NC measures that are absolutely continuous with respect to $\la$ in the repreducing kernel Hilbert space sense is denoted by $AC_{R}[\la]$.
\item $\mu$ is \emph{singular with respect to $\la$ in the reproducing kernel sense} , in notations $\mu \perp_{R} \lambda$, if $\mr{int} (\mu , \la ) = \{ 0 \}$.  The set of all positive NC measures that are singular with respect to $\la$ in the repreducing kernel Hilbert space sense is denoted by $SG_{R}[\la]$.
\item the decomposition $\mu=\mu_{1}+\mu_{2}$ with respect to $\la$ is called a Lebesgue RKHS  decomposition if $\mu_{1} \ll_{R} \la$ and $\mu_{2} \perp_{R} \la$. We may write this decomposition by $\mu=\mu^{R}_{1}+\mu^{R}_{2}$, where the superscript $R$ pints to the RKHS decomposition sense.
\end{itemize}
We might also say  \emph{ $\mu$ is RK-AC with respect to $\la$} instead of $\mu \ll_{R} \lambda$ and \emph{ $\mu$ is RK-Sing with respect to $\la$} in place of of $\mu \perp_{R} \lambda$. We further define $\mathrm{Int} (\mu , \la) := \mathrm{int} (\mu , \la ) ^{-\| \cdot \| _{H_{\mu}}}$, the closure of $\mathrm{int} (\mu , \la )$ in $\scr{H} ^+ (H _{\mu} )$. Note that in $\mathrm{Int} (\mu , \la)$  the order of $\mu$ and $\la$ tells with respect to which space we are taking the closure.
\end{defn}
\begin{remark} \label{a_classical_non_tranlable_fact}
    In section 5, we will see that this definition is a special case of a more general definition that contains the role of Toeplitz operators as well. For splitting non-Cuntz measures the above definition certainly works, but for Cuntz splitting measures, or generally any splitting measure, we need a replacement.  In fact, the problem is with singularity definition which is does not cover the splitting Cuntz measures, see Example \ref{non-example} and Example \ref{classical-sing-not-RKHS-Sing}. This condition has something to do with the fact that the derivative of singular measures with respect to the Lebesgue measure is zero \cite[Proposition 3.30]{Folland}. Besides, singularity should somehow be the antithesis of absolute continuity, but the above definition does not cover the other alternatives. Even if the intersection space is not dense, there might be some portions of $\mu$ that provides a dense intersection.  Also, prior to the Definition \ref{RK-AC-Sing}, we showed that absolutely continuous in the classical  sense implies that the intersection space is dense. Let's point out here that  the converse  is also true. That is if the intersection space is dense, then we have classical-AC; see  Corollary \ref{RK-AC-equiv-AC} and also \cite[Theorem 6]{BMN}. The proof of this equivalence needs a lot of machinery that will be built later. $\blacksquare$
\end{remark}

\begin{defn} \label{RKHS-order-of-Nc-measures}
  We put a RKHS order on positive NC measures by declaring that  
  $\mu \leq_{R} \nu$ if ${\mathscr{H}}^{+}(H_{\mu}) \subseteq {\mathscr{H}}^{+}(H_{\nu}) $.   
\end{defn}

\begin{remark} \label{maximal-Lebesgue-decomp}
    Suppose that we have two positive NC measures $\mu$ and $\la$. Assume that we have a notion of absolute continuity and singularity with respect to $\la$, which we depict them by $\ll$ and $\perp$. Furthermore, suppose that we have an order among positive NC measures, denoted by $\leq$, which in some ways is related to the way how  $\ll$ and $\perp$ are defined. Then, when we say that a    
    decomposition $\mu=\mu_1 +\mu_2$ with respect to $\la$ is a \emph{Lebesgue decomposition}, we mean that $\mu_1 \ll \la$ and $\mu_2 \perp \la$, and $\mu_i$'s are positive NC measures bounded above by $\mu$. Also, when we say that the decomposition $\mu=\mu_1 +\mu_2$ is \emph{ the maximal Lebesgue decomposition} of $\mu$ with respect to $\la$, not only do we require it to be a Lebesgue decomposition, but also we want $\mu_1$ to be the maximal among those measures which are absolutely continuous  with respect to $\la$ and at the same time are less than or equal to $\mu$. Obviously, maximal Lebesgue decomposition is unique whenever it exists.
\end{remark}

\begin{defn}
    A positive NC measure $\lambda$ is called Cuntz if its GNS row isometry $\Pi_{\lambda}$ is surjective; i.e., $\Pi_{\lambda}$ is unitary. By \cite[Theorem 6.4]{JM-freeCE}, such measures are also called column extreme (CE).
\end{defn}
At this point we would like to record some general facts about non-Cuntz measure which we will need later.
\begin{remark} \label{Popescu-decomposition-of-row-isometry}
    In \cite[Theorem 6.4]{JM-freeCE}, a characterization of Cuntz measures is given. For example, $\la$ is Cuntz if and only if $\mathscr{H}^{+}(H_\la)$ does not contain constant functions, and this is equivalent to saying that the constant monomial $I+N_\la$ can be approximated by the closed linear span free monomials $L^{\alpha}+N_\la$. Hence, one can see that   
    the standard NC Lebesgue measure $m$ is an example of non-Cuntz measures, while any singular classic measure is an example of a Cuntz measure.  Furthermore, by \cite[Theorem 1.3]{Pop-dil}, for any positive NC measure $\la$, one can split its GNS row isometry $\Pi_\la$ into two pieces, the $L$-part and the Cuntz part. That is $\Pi_\la=\Pi_{\la; L} \oplus \Pi_{\la; C}$, such that $\Pi_{\la; L} =L \otimes I_{\mathcal{K}}$ where $I_{\mathcal{K}}$ is the identity operator on some Hilbert space $\mathcal{K}$, and $\Pi_{\la; C}$ is Cuntz. Here, the $L$-part $\Pi_{\la; L}$ is non-Cuntz, while the Cuntz part $\Pi_{\la; C}$ is Cuntz. We should mention that $\la$ is Cuntz if and only if $\Pi_{\la; L}=0$, i.e., it has no $L$-part. Thus, while a Cuntz measure is solely composed  of Cuntz part,  a non-Cuntz measure always has a non-zero $L$-part plus probably a Cuntz part. 
\end{remark}
Recall that for a positive NC  measure $\tau$ the left ideal of zero length elements, $N_{\tau}$, is defined by $N_\tau = \left\{ a \in \mathbb{A}_{d} : \tau(a^* a)=0 \right\}.$
\begin{thm} \label{left-ideal-non-Cuntz=0}
    For any non-Cuntz measure $\tau$, the left ideal of zero length elements, $N_{\tau}$, is zero.
\end{thm}
\begin{proof}
    By the discussion in Remark \ref{Popescu-decomposition-of-row-isometry}, $\tau$ decomposes as $\tau=\tau_{L}+\tau_{C}$, where both summands are positive NC measures corresponding to the $L$-part and Cuntz part of  GNS row isometry $\Pi_{\tau}$. By the discussion after  \cite[Definition 6.12]{JM-ncld}, there is $h=P_{L}(I+N_{\tau}) \in \mathbb{H}^{2}_{d} $  such that $\tau_{L}(p^*q)=\left \langle p(L)h, q(L)h \right \rangle_{\mathbb{H}^{2}_{d}}$. That is the $L$-part of $\tau$ is a weighted Lebesgue measure with weight $h$. By \cite[Theorem 1.7]{DP-inv}, any element of  $\mathbb{A}_{d} \subseteq \mathbb{L}^{\infty}_{d}$ is injective; hence, $\tau_{L}(p^*p)=\parallel p(L)h \parallel^{2}_{\mathbb{H}^{2}_{d}}=0$ if and only if $p=0$. Now, suppose that $p \in N_\tau$, so $0=\tau(p^*p)=\tau_L (p^*p)+\tau_C (p^*p)$. However, since both summands are positive NC measure, we see that $\tau_L (p^*p)$ and $\tau_C (p^*p)$ are non-negative, so they both must be zero. In particular, $\tau_{L}(p^*p)=0$, which forces $p=0$ by the previous discussion. Thus, $N_\tau =0$.
\end{proof}

\begin{cor} \label{left-ideal-Cuntz+non-Cuntz=0}
    Let $\la$ be a Cuntz measure, and $\nu$ be a non-Cuntz measure, then $\la+\nu$ is non-Cuntz and $N_{\la+\nu}=0$
\end{cor}
\begin{proof}
    Note that $\la+\nu$ has a non-Cuntz part by the discussion in Remark \ref{Popescu-decomposition-of-row-isometry}; hence, the result follows form  Theorem \ref{left-ideal-non-Cuntz=0}.
\end{proof}
\begin{cor} \label{H2(la)-Cuntz-nonCuntz}
    Let  $\la$ is a classical measure on the unit circle $\mathbb{T}$. 
    \begin{enumerate}
        \item  If $\la$ is non-Cuntz,   $H^2(\la)$ is the completion of the polynomials in the inner product $\langle p,q \rangle_{\la} = \la(\overline{p}q)$.
        \item If $\la$ is Cuntz, ${H}^{2}(\la)=L^{2}(\la)$.
        \end{enumerate}
\end{cor}
\bp
1. By Theorem \ref{left-ideal-non-Cuntz=0}, $N_\la=0$, so $H^2(\la)$ is the completion of the disc algebra in $\langle \cdot, \cdot \rangle_\la$. However, polynomials are dense in the disc algebra. 

2. Since $\la$ is Cuntz, by 
\cite[Theorem 6.4]{JM-freeCE}, constant functions can be approximated by monomials $z^n,~ n \in \mathbb{N}$, then by Szego-Kolomogorov-Krein theorem \cite[Chapter 4]{Hoff} we see that ${H}^{2}(\la)=L^{2}(\la)$.
\ep

\begin{remark} \label{GNS-of-non-Cuntz-vs-m}
Let $\tau$ be a non-Cuntz measure, so by Theorem \ref{left-ideal-non-Cuntz=0} we know that $N_\tau=0$. Thus, if we want to construct the GNS space of $\tau$, we see that $\frac{\mathbb{A}_d}{N_\tau} \cong \mathbb{A}_d $, which is similar to the case of NC Lebesgue measure $m$ with $N_m=0$. However, the difference is in completion, i.e., if we complete $\mathbb{A}_d $ with respect to the inner product induced by $\tau$, we get $\mathbb{H}^{2}_{d}(\tau)$. While completing $\mathbb{A}_d$ with respect to $m$ gives the Fock space $\mathbb{H}^{2}_{d}$. Note though in the decomposition $\tau=\tau_L +\tau_C$ the part $\tau_L$ is given by the weighted Lebesgue measure, we cannot generally compare the two norms induced by $\tau$ and $m$. Hence, we cannot say anything more about the comaprison between the GNS spaces of $\tau$ and $m$. Furthermore, as far as we work with $\mathbb{A}_d $ the GNS isometry $\Pi_\tau$ is   essentially $L$; however, when we work with the extended isometry they are different as they live on different GNS spaces.
\end{remark}
\begin{prop}\label{reducing-intersection}
    Let $\mu$ and $\lambda$ be two arbitrary positive NC measures on the operator system ${\mathcal{A} }_{d}$, and $\lambda$ be non-Cuntz. Then, the intersection space ${\mathscr{H}}^{+}(H_{\mu}) \cap {\mathscr{H}}^{+}(H_{\lambda}) $  is $V_\mu$ reducing.
\end{prop}
\bp
By \cite[Theorem 4.5]{JM-ncld}, the intersection space is $V_\mu$ co-invariant. On the other hand, since $\lambda$ is non-Cuntz, ${\mathscr{H}}^{+}(H_{\lambda})$ contains constant functions by \cite[Theorem 6.4]{JM-freeCE}. Hence, by \cite[Proposition 4.8]{JM-ncld} the intersection space is $V_\mu$ invariant.
\ep

\begin{thm} \label{rkhs-non-CE}
Let $\mu$ and $\lambda$ be two arbitrary positive NC measures on the operator system ${\mathcal{A} }_{d}$, and $\lambda$ be non-Cuntz. Then, $\mu$ has the  maximal Lebesgue RKHS decomposition $\mu =\mu_{ac} + \mu_{s} $ with respect to $\lambda$, i.e.,  $\mu_{ac}$ is the maximal positive NC measure bounded above by $\mu$, which is also RK-AC with respect to $\la$. Also, $\mu_s$  is RK-Sing with respect to $\la$ and $0 \leq \mu_{s} \leq \mu $. The measures $\mu_{ac}$ and $\mu_s$ are defined by
$${\mathscr{H}}^{+}(H_{\mu_{ac}}):= {\overline{{\mathscr{H}}^{+}(H_{\mu}) \cap {\mathscr{H}}^{+}(H_{\lambda}) }}^{H^{+}(H_{\mu})}=\mathrm{Int}(\mu, \la)$$
and 
$${\mathscr{H}}^{+}(H_{\mu_{s}}):={\mathscr{H}}^{+}(H_{\mu}) \ominus {\mathscr{H}}^{+}(H_{\mu_{ac}}).$$
Both $ {\mathscr{H}}^{+}(H_{\mu_{ac}})$ and ${\mathscr{H}}^{+}(H_{\mu_{s}})$ are reducing for $V_{\mu}$ and 
$${\mathscr{H}}^{+}(H_{\mu})={\mathscr{H}}^{+}(H_{\mu_{ac}}) \oplus {\mathscr{H}}^{+}(H_{\mu_{s}}).$$
\end{thm}
\begin{proof}
     When $\lambda$ is not Cuntz, by Proposition \ref{reducing-intersection} the intersection space ${\mathscr{H}}^{+}(H_{\mu}) \cap {\mathscr{H}}^{+}(H_{\lambda})$ is $V_{\mu}={\mathscr{H}}^{+}(H_{\mu}) \cap {\mathscr{H}}^{+}(H_{\lambda})$ reducing. Hence,  \cite[Theorem 4.7 ]{JM-ncld} applies to $\mathcal{M}=\overline{{\mathscr{H}}^{+}(H_{\mu}) \cap {\mathscr{H}}^{+}(H_{\lambda})}^{\mu}$ , and like  \cite[Theorem 4.9]{JM-ncld} we get our result. Note that the decomposition is maximal since by the way $\mu_{ac}$ is defined, it just depends on the intersection space $\mathrm{int}(\mu, \la)$ and its closure with respect to $\mathscr{H}^{+}(H_{\mu})$. Thus, for example if $\mu_{ac} \leq_{R} \mu_1$, then $\mathrm{int}(\mu_{ac}, \la) \subseteq \mathrm{int}(\mu_1, \la) \subseteq \mathrm{int}(\mu, \la)$. Upon taking closure with respect to $\mathscr{H}^{+}(H_{\mu})$, we see that $\mathrm{Int}(\mu_{ac}, \la) = \mathrm{Int}(\mu_1, \la) = \mathrm{Int}(\mu, \la)$, which shows that $\mathscr{H}^{+}(H_{\mu_{ac}})=\mathscr{H}^{+}(H_{\mu_1})$, i.e., $\mu_{ac}=\mu_1$.
     \end{proof}
\begin{thm} \label{positive-cone-AC-{R}(la)}
    Let $\la$ be any positive NC measure. Then, the set $AC_{R}[\la]=\{ \mu \in \left(\mathcal{A}_d \right)^{\dagger}_{+}: \mu \ll_{R} \la \}$ is a positive cone.
\end{thm}
\bp
$AC_{R}[\la]$ is not empty since it contains $\la$, so let $\mu$ and $\nu$ be in it. We must show that  $\gamma=\mu+\nu \in AC_{R}[\la]$, i.e.,
$$\mathscr{H}^{+}(H_{\gamma}) = \overline{\mathscr{H}^{+}(H_{\gamma}) \cap \mathscr{H}^{+}(H_{\la})}^{\gamma}.$$
So, suppose $h \in \mathscr{H}^{+}(H_{\gamma}) $. We have to find a sequence $h_n \in \mathrm{int}(\gamma, \la)$ which is $|| \cdot ||_{H_{\gamma}}$-convergent to $h$. However, note that $h=f+g$ where $f \in \mathscr{H}^{+}(H_{\mu}) $ and $\mathscr{H}^{+}(H_{\nu})$. Also, recall that $\mu \ll_{R} \la$ and  $\nu \ll_{R} \la$. Therefore, 
\begin{eqnarray*}
    \exists ~ (f_n) \subseteq \mathscr{H}^{+}(H_{\mu}) \cap \mathscr{H}^{+}(H_{\la}) \quad \mathrm{ so ~that} \quad f_n \longrightarrow f ~ \mathrm{in} ~ \mathscr{H}^{+}(H_{\mu})\\
    \exists ~ (g_n) \subseteq \mathscr{H}^{+}(H_{\nu}) \cap \mathscr{H}^{+}(H_{\la}) \quad \mathrm{ so ~that} \quad g_n \longrightarrow g ~ \mathrm{in} ~ \mathscr{H}^{+}(H_{\nu})
\end{eqnarray*}
Since $\mu , \nu \leq \gamma$, by \cite[Lemma 4.2]{JM-ncld} we have the  contractive  embeddings $e_\mu$ and $e_\nu$ of $\mathscr{H}^{+}(H_{\mu})$ and $\mathscr{H}^{+}(H_{\nu})$ into $\mathscr{H}^{+}(H_{\gamma})$. For the sequence
$$ h_n := e_\mu f_n + e_\nu g_n \in \mathscr{H}^{+}(H_{\gamma}) \cap \mathscr{H}^{+}(H_{\la}),$$
 we have
 \begin{eqnarray*}
     || h_n - h_m ||_{H_{\gamma}} &\leq& || e_\mu( f_n - f_m )||_{H_{\gamma}} + || e_\nu (g_n - g_m )||_{H_{\gamma}}\\
     &\leq& ||  f_n - f_m ||_{H_{\mu}} + || g_n - g_m ||_{H_{\nu}},
 \end{eqnarray*}
  which shows $h_n$ is Cauchy in $\mathscr{H}^{+}(H_{\gamma})$, and hence it must be convergent. Nevertheless, for any $Z \in \mathbb{B}_{\mathbb{N}}^{d}$ we see
  $$h_n (Z) := e_\mu f_n (Z) + e_\nu g_n (Z)=f_n (Z) +  g_n (Z)  \longrightarrow f(Z)+g(Z), $$
  so $h$ is the limit of $h_n$. This shows that $\mathscr{H}^{+}(H_{\gamma}) \cap \mathscr{H}^{+}(H_{\la})$ is dense in $\mathscr{H}^{+}(H_{\gamma}) $, i.e., $\gamma \ll_{R} \la$.
\ep
\begin{prop}
Let $\la$ be a non-Cuntz positive NC measure. Then, the set $SG_{R}[\la]=\{ \mu \in \left(\mathcal{A}_d \right)^{\dagger}_{+}: \mu \perp_{R} \la \}$ is hereditary, i.e., if $\mu \in SG_{R}[\la]$ and $\nu$ is any positive NC measure such that $\nu \leq \mu$, then $\nu \in SG_{R}[\la]$.
\end{prop}
\bp
If $\nu$ happen not to be singular, then it  has an absolutely continuous part. However, $0 \lvertneqq \nu_{ac} \leq \nu \leq \mu$. Thus, 
$$\{0\} \subsetneqq \mathscr{H}^{+}(H_{\nu_{ac}}) \cap  \mathscr{H}^{+}(H_{\la})  \subseteq \mathscr{H}^{+}(H_{\mu}), $$
which means $ \mathscr{H}^{+}(H_{\mu})$ has a non-trivial intersection with $ \mathscr{H}^{+}(H_{\la})$, contradicting the assumption. Hence, $\nu$ must be singular.
\ep

\section{The Lebesgue form decomposition with respect to a non-Cuntz measure}
The quadratic form method for the Lebesgue decomposition of measures has its roots in the von Neumann approach  to the Lebesgue decomposition of measures \cite{vN3}, and in fact it recovers von Neuman's proof \cite{Simon1}.

Consider a positive semi-definite quadratic form $q: \mathrm{Dom(q)} \times \mathrm{Dom(q)} \longrightarrow \mathbb{C} $, where $\mathrm{Dom(q)}$ is a dense linear subspace of some Hilbert space $\mathcal{H}$. Define a new inner product on $\mathrm{Dom(q)}$ by
$$\langle x , y \rangle_{q+1} := q(x,y) + \langle x , y \rangle_{\mathcal{H}}, $$
\noindent and let $\mathcal{H}(q+1)$ be the Hilbert space completion of $\mathrm{Dom(q)}$ with respect to the inner product  $\langle \cdot , \cdot \rangle_{q+1}$. Now, the form $q$ is \emph{closed} if $\mathrm{Dom(q)}$ is complete in the norm $|| \cdot ||_{q+1}$, i.e., $\mathrm{Dom(q)}=\mathcal{H}(q+1)$.  Equivalently, $q$ is closed if and only if $x_{n} \in \mathrm{Dom(q)}$ converges in $|| \cdot ||_{\mathcal{H}}$ to $x \in \mathcal{H}$ and $q(x_n - x_m , x_n - x_m) \longrightarrow 0$, then $x \in Dom(q)$ and $q(x_n-x , x_n-x) \longrightarrow 0$.  The importance of closed forms is that they obey a Riesz-like lemma. That is, $q$ is closed if and only if there is a unique closed, positive semi-definite operator $A$, with dense domain in $\mathcal{H}$ so that $\mathrm{Dom}(q)=\mathrm{Dom} (\sqrt{A})$ and 
$$q(x,y)=q_{A}(x,y)=\langle \sqrt{A} x, \sqrt{A}y \rangle_{\mathcal{H}}; \quad x,y \in \mathrm{Dom}(q),$$
\cite[Chapter VI, Theorem 2.1, Theorem 2.23]{Kato}. 

An extension of a form $q$ is another form $p$, symbolically $q \subseteq p$, if $\mathrm{Dom(q)} \subseteq \mathrm{Dom(p)}$ and $q(u,v)=p(u,v)$ over $\mathrm{Dom(q)}$. A form $q$ is closable if it has a closed extension; equivalently, whenever $x_{n} \in \mathrm{Dom(q)}$ converges in $|| \cdot ||_{\mathcal{H}}$ to $0$ and $q(x_n - x_m , x_n - x_m) \longrightarrow 0$, then $q(x_n , x_n) \longrightarrow 0$.   A Closable form $q$ has a minimal closed extension, $\overline{q}$, with $\mathrm{Dom(\overline{q})} \subseteq \mathcal{H}$ equal to the set of all $h\in \mathcal{H}$ so that there exists a sequence $(h_n)$ in $\mathrm{Dom(q)}$ converging to $h$  such that $(h_n)$ is Cauchy in the norm of $\mathcal{H}(q+1)$. In this situation, we have
$$\overline{q}(h,h)= \mathrm{lim}_{n \to \infty} q(h_n , h_n).$$

A linear subspace $\mathcal{D} \subseteq \mathrm{Dom(q)}$ is called a form core for a closed form $q$ if $\mathcal{D}$ is dense in $\mathcal{H}(q+1)$. For a closable form $q$, $\mathrm{Dom(q)}$  is a form core for  $\overline{q}$, see \cite[Chapter VI, Theorem 1.21]{Kato}. If $q=q_{A}$ is a closed, positive semi-definite quadratic form, then $\mathcal{D}$ is a core for $q$ if and only if $\mathcal{D}$ is a core for $\sqrt{A}$.  Recall that $\mathcal{D}$ is a core for closed operator $T$ if $\{ (x, Tx): x \in \mathcal{D} \} $ is dense in the graph of $T$.

Let $q$ and $\sigma$ be two densely defined forms. We say that $q \leq \sigma$ if $\mathrm{Dom}(q) \supseteq \mathrm{Dom}(\sigma)$ and $q(h,h) \leq \sigma(h,h)$ for all $h \in \mathrm{Dom}(\sigma)$. In \cite[Section 2]{Simon1}, Simon proved that any densely defined, positive semi-definite form $q$ has a unique maximal Lebesgue decomposition 
$$q=q_{ac}+q_s;$$
where $q_{ac}$ is the maximal closable positive form bounded above by $q$, i.e., $0 \leq q_{ac} \leq q$ and $0 \leq q_s=q-q_{ac} \leq q$. The part $q_{ac}$  is called the absolutely continuous part of $q$, and $q_s$ is the singular part of $q$. For forms on GNS Hilbert spaces we give the following definition.
\begin{defn}\label{Simon's definition}
Let $\mu$ and $\la$ be two positive NC measures and $N_{\mu} \supseteq N_{\la}$. We say:
\begin{itemize}
    \item $\mu$ is absolutely continuous with respect to $\la$ in the forms sense, in notation $\mu \ll_{F} \la$, if the hybrid form
\begin{eqnarray*}
    q_{\mu}^{\la}: \left(\mathbb{A}_d +N_{\la}\right) \times \left(\mathbb{A}_d +N_{\la}\right)  &\longrightarrow &\mathbb{C}\\
    q_{\mu}^{\la}(a+N_{\la} , b+N_{\la}) &=&\left \langle a+N_\mu, b+N_\mu \right \rangle_{\mu}= \mu(a^{*}b)
\end{eqnarray*}
    is closeable. The set of all positive NC measures that are absolutely continuous with respect to $\la$ in the form sense is denoted by $AC_{F}[\la]$.

    \item $\mu$ is singular with respect to $\la$ in the forms sense, in notation $\mu \perp_{F} \la $, if the hybrid form $q_{\mu}^{\la}$ on $\mathbb{H}_{d}^{2}(\la)$ does not majorize any closeable form except the identically zero form. The set of all positive NC measures that are singular with respect to $\la$ in the form sense is denoted by $SG_{F}[\la]$.
    \item the decomposition $\mu=\mu^{F}_{1}+\mu^{F}_{2}$ with respect to $\la$ is called a Lebesgue form decomposition if $\mu^{F}_{1} \ll_{F} \la$ and $\mu^{F}_{2} \perp_{F} \la$.

    We might use the short-hand $\mu$ is form-AC and form-Sing instead of $\mu \ll_{F} \la$ and $\mu \perp_{F} \la $ respectively.
\end{itemize}  
\end{defn} 
\begin{remark}
    The assumption $N_\la \subseteq N_\mu$  in Definition \ref{Simon's definition} is necessary to make sure the hybrid form $q_{\mu}^{\la}$ is well-defined. If we think of the simplest case where $\mu \leq \la$ as positive linear functionals, we notice that $N_\la \subseteq N_\mu$, and in this case $\mu \ll_{F} \la$ as we expect. From now on, when we discuss hybrid forms, or write notations $\mu \ll_{F} \la$ and $\mu \perp_{F} \la $, we implicitly assume that $N_\la \subseteq N_\mu$. So, we write this assumption only when there is a danger of confusion. Also, note that the domain of the hybrid form $q_{\mu}^{\la}$ is the dense subspace $\mathbb{A}_d +N_\la$ of  $\mathbb{H}_{d}^{2}(\la)$. Furthermore, by Remark \ref{GNS-of-non-Cuntz-vs-m}, when we deal with a non-Cuntz measure $\la$, the left ideal $N_\la$ is zero. In this case, we usually write $a$ in place of $a+N_\la$ for $ a \in \mathbb{A}_{d}$ unless there is a danger of confusion. 
\end{remark}
\begin{defn} \label{form-order-of-Nc-measures}
  We put a form order on positive NC measures by declaring that  
  $\mu \leq_{F} \nu$ if there exists a measure $\la$ such that the hybrid forms $q_{\mu}^{\la}$ and $ q_{\nu}^{\la} $ are well-defined and $q_{\mu}^{\la} \leq  q_{\nu}^{\la} $.   That is,  $q_{\mu}^{\la}(h,h) \leq q_{\nu}^{\la}(h,h)$ for all $h \in \mathbb{A}_d +N_{\la}$ . Such a $\la$ always exists, e.g., $\la=m$ the standard NC Lebesgue measure.
\end{defn}
\begin{remark}
 By  \cite[Theorem 4.1]{JM-ncld}, 
 and the discussion after it, the orders $\leq_{R}$ and $\leq_{F}$, whenever defined, are equivalent to the ordinary $\leq$ for positive NC measures. So, it is not important if we work with any of them since all of them coincide.
 \end{remark}
 \begin{defn}
     Let $\la$ be any positive NC measure. A positive semi-definite sesquilinear form $q$ with dense  domain $\mathrm{Dom}(q) \subseteq \mathbb{H}^{2}_{d}(\la)$ is called $\Pi_{\la}$-Toeplitz if: 
\begin{enumerate}
    \item $\mathrm{Dom}(q) $ is $\Pi_\la$-invariant,
    \item $q$ enjoys Toeplitz property; i.e., 
     $$q(\Pi_{\la, i}g ,\Pi_{\la, j} h )=\delta_{i,j}q(g , h); \quad \forall g,h \in \mathrm{Dom}(q). $$
\end{enumerate}
 \end{defn}

\begin{defn}
    Let $\la$ be any positive NC measure. A closed, positive, semi-definite operator $T$ with domain $\mathrm{Dom}(T) \subseteq \mathbb{H}^{2}_{d}(\la) $  will be called $\Pi_{\la}$-Toeplitz if the associated quadratic form 
        $$q_{T}^{\la}(g, h):=\langle \sqrt{T}g, \sqrt{T}h \rangle_{\la} $$ 
        is $\Pi_{\la}$-Toeplitz. That is, if    
    \begin{enumerate}[(i)]
        \item $\mathrm{Dom}(\sqrt{T})$ is $\Pi_\la$-invariant,
        \item  $q_{T}^{\la}(\Pi_{\la, i}g ,\Pi_{\la, j} h )=\delta_{i,j}q_{T}^{\la}(g , h), \quad \forall g,h \in \mathrm{Dom}(\sqrt{T}).$
    \end{enumerate}
    When $T$ is a bounded operator on $\mathbb{H}^{2}_{d}(\la)$, both conditions above reduce to one single condition; i.e., $T$ is $\Pi_{\la}$-Toeplitz if and only if
    $\Pi_{\la, i}^{*} T \Pi_{\la, j}=\delta_{i,j}T$
\end{defn}

\begin{lemma}\label{hybrid-form-Toeplitz}
     Suppose that $\mu$ and $\la$ are two positive NC measures such that $N_\la \subseteq N_\mu$. Then, the hybrid form $q_{\mu}^{\la}$ defined by
     $$q_{\mu}^{\la}(a+N_{\la} , b+N_{\la}) =\left \langle a+N_\mu, b+N_\mu \right \rangle_{\mu}= \mu(a^{*}b)$$
     is $\Pi_\la$- Toeplitz, i.e.,
$$q_{\mu}^{\la}(\Pi_{\la, i}( a+N_\la) ,\Pi_{\la, j} (b+N_\la) )=\delta_{i,j}q_{\mu}^{\la}(a+N_\la , b+N_\la).$$
\end{lemma}
\bp
We have 
\begin{eqnarray*}
    q_{\mu}^{\la}(\Pi_{\la, j}(a+N_\la), \Pi_{\la, k}(b+N_\la))&=&q_{\mu}^{\la}(Z_ja+N_\la, Z_kb+N_\la)\\
    &=& \mu(a^*L_{j}^{*}L_kb)\\
    &=& \delta_{j,k}\mu(a^*b)\\
    &=& \delta_{j,k} q_{\mu}^{\la}(a+N_\la, b+N_\la).
\end{eqnarray*}
\ep

The following proposition is a generalization of \cite[Proposition 1]{JM-ncFatou} with essentially the same proof. However, for the sake of completeness and future record we would like to give a detailed proof here. 
\begin{prop} \label{reducing-projections}
Let $\mu$ and $\la$ be two positive NC measures on $\mathcal{A}_d$, and let $\sigma=\mu+\la$. Then, the contractive co-embedding 
\begin{eqnarray*}
 E_{\lambda}=E_{\lambda, \sigma}:\mathbb{H}_{d}^{2}(\sigma) &\longrightarrow & \mathbb{H}_{d}^{2}(\lambda) \\
  f+N_{\sigma} & \longmapsto & f+N_{\lambda},
  \end{eqnarray*}
  is well-defined.  Suppose that $Q_s$ and $Q_{ac}$ are  orthogonal projections onto $\mathrm{Ker}(E_\la)$ and $\mathrm{Ker}(E_\la)^{\perp}$ respectively. If  $\la$ is  non-Cuntz, then $Q_s$ and $Q_{ac}$ are reducing for $\Pi_{\sigma}$, i.e., $Q_s , ~Q_{ac} \in vN(\Pi_\sigma)^{'}$, where $vN(\Pi_\sigma)$ is the von Neumann algebra generated by $\Pi_\sigma$, and $vN(\Pi_\sigma)^{'}$ is its commutant.
\end{prop}
\bp
Since $\la \leq \sigma$, we see that $N_\sigma \subseteq N_\la$; hence, $E_\la$ is well-defined.   Note that $Q_{ac}+Q_s=I_{\mathbb{H}_{d}^{2}(\sigma)}$, so it is enough to show that $\mathrm{Ran}(Q_s)$ and $\mathrm{Ran}(Q_{ac})$  are both $\Pi_\sigma$-invariant. Recall that $E_\la$ is an intertwiner, i.e., $E_{\lambda} \Pi_{\sigma,k}=\Pi_{\la, k} E_{\lambda} $ with the following commutative diagram
\begin{center}
\begin{tikzcd}
\mathbb{H}_{d}^{2}(\sigma) \arrow[r, "E_{\lambda}"]
& \mathbb{H}_{d}^{2}(\lambda) \\
\mathbb{H}_{d}^{2}(\sigma) \arrow[r, "E_{\lambda}"] \arrow[u, "\Pi_{\sigma, k}"]
& \mathbb{H}_{d}^{2}(\lambda) \arrow[u, "\Pi_{\lambda , k} "]
\end{tikzcd}
\end{center}
However, this means that $\mathrm{Ran}(Q_s)=\mathrm{Ker}(E_\la)=\mathrm{Ker}(E_{\la}^{*}E_{\la})$ is $\Pi_\sigma$-invariant. In fact, if $x\in \mathrm{Ker}(E_\la) $, then for any $1 \le k \leq d$
$$E_\la (\Pi_{\sigma, k} x)=\Pi_{\la,k} (E_\la x)=\Pi_{\la,k}(0)=0,$$
which shows $\Pi_{\sigma, k} x \in \mathrm{Ker}(E_\la) $ again. Thus, we see that $Q_s$ is $\Pi_{\sigma}$-invariant. This also shows that $Q_{ac}$ is $\Pi_{\sigma}^{*}$-invariant since $Q_{ac}=Q_{s}^{\perp}$. 

We now need to show that $\mathrm{Ran}(Q_{ac})$ is $\Pi_{\sigma}$-invariant. To do so, we transfer to the RKHS model. Recall that $ V_{\sigma,k}=\mathscr{C}_{\sigma} \Pi_{\sigma, k} \mathscr{C}_{\sigma}^{*} $ with the following commutative diagram
\begin{center}
\begin{tikzcd}
\mathbb{H}_{d}^{2}(\sigma) \arrow[r, "\mathscr{C}_{\sigma}"]
& \mathscr{H}^{+}(H_{\sigma}) \\
\mathbb{H}_{d}^{2}(\sigma) \arrow[r, "\mathscr{C}_{\sigma}"] \arrow[u, "\Pi_{\sigma, k}"]
& \mathscr{H}^{+}(H_{\sigma}) \arrow[u, "V_{\sigma , k} "]
\end{tikzcd}
\end{center}
Thus, we need to show that $\check{Q}_{ac}= \mathscr{C}_{\sigma} Q_{ac} \mathscr{C}_{\sigma}^{*}$ is $V_{\sigma}$-invariant. Notice that $\mathrm{Ran}(Q_{ac})=\overline {\mathrm{Ran}(E_{\la}^{*})}^{\sigma}$, so $\mathrm{Ran}(\check{Q}_{ac})=\mathscr{C}_{\sigma}( \overline {\mathrm{Ran}(E_{\la}^{*})}^{\sigma} )=\overline {\mathrm{Ran}(e_{\la})}^{\sigma}= \overline {\mathscr{H}^{+}(H_{\sigma}) \cap \mathscr{H}^{+}(H_{\la})}^{\sigma}$. Hence, it is enough to prove that $\mathscr{H}^{+}(H_{\sigma}) \cap \mathscr{H}^{+}(H_{\la})$ is $V_\sigma$-invariant. So, let $f \in  \mathscr{H}^{+}(H_{\sigma}) \cap \mathscr{H}^{+}(H_{\la})$; hence, by \cite{JM-freeCE} we have
\begin{eqnarray*}
    (V_{\sigma, k} f)(Z)-(V_{\sigma, k} f)(0_n)&=&Z_k f(Z)\\
    &=& (V_{\la, k} f)(Z)-(V_{\la, k} f)(0_n).
\end{eqnarray*}
Thus, $(V_{\sigma, k} f)(Z)=(V_{\la, k} f)(Z)+cI_n$, where $c=(V_{\sigma, k} f)(0_n)-(V_{\la, k} f)(0_n)$ is constant. Nevertheless, $\la$ is non-Cuntz, and by  \cite[Theorem 6.4]{JM-freeCE} $\mathscr{H}^{+}(H_{\la})$ contains the constant functions. So, $(V_{\sigma, k} f)(Z) \in \mathscr{H}^{+}(H_{\la}) $, and obviously $(V_{\sigma, k} f)(Z) \in \mathscr{H}^{+}(H_{\sigma}) $. Therefore, $(V_{\sigma, k} f)(Z) \in \mathscr{H}^{+}(H_{\sigma}) \cap \mathscr{H}^{+}(H_{\la}) $.
\ep

The following theorem and its proof are also a generalization of \cite[Theorem 4]{JM-ncFatou}.
\begin{thm} \label{form-decomposition-nonCE}
    Let $\mu$ and $\la$ be two positive NC measures on $\mathcal{A}_d$, and let $\la$ be non-Cuntz. Then, $\mu$ has the  maximal Lebesgue form decomposition $\mu =\mu_{ac} + \mu_{s} $ with respect to $\lambda$, where $\mu_{ac}$ is the maximal measure which is form-AC with respect to $\la$ and $0 \leq \mu_{ac} \leq \mu$. Also, $\mu_s=\mu-\mu_{ac}$  is form-Sing with respect to $\la$ and $0 \leq \mu_{s} \leq \mu $.
\end{thm}
\bp 
Put $\sigma=\mu+\la$ and $\mathcal{D}_{\la}=\mathbb{A}_d + N_{\la}$, where $N_{\la}$ is the left ideal of zero length elements under $\la$, and $\mathcal{D}_{\la}$ is a dense linear subspace of $\mathbb{H}_{d}^{2}(\la)$ by the GNS construction. Since $\la$ is non-Cuntz, by Theorem \ref{left-ideal-non-Cuntz=0} we have $N_{\la}=0 \subseteq N_{\mu}$; hence, $q_{\mu}^{\la}$ and $q_{\sigma}^{\la}$  are well-defined. Let's drop the superscript $\la$ until the end of the proof and write $q_\mu$ instead of $q_{\mu}^{\la}$ and so on. Now, the the inner product $\langle \cdot ,\cdot \rangle_{q_{\mu} +1}=q_{\mu}(\cdot ,\cdot)+ \langle\cdot ,\cdot \rangle$ on $\mathcal{D}_{\la}$ is the form corresponding to $q_{\sigma}$ so that the completion of $\mathcal{D}_{\la}$ with respect to $\langle 
\cdot ,\cdot \rangle_{q_{\sigma} }$ is $\mathbb{H}_{d}^{2}(\sigma)$. Since $\la \leq \sigma$, or by Theorem \ref{left-ideal-non-Cuntz=0} and Corollary \ref{left-ideal-Cuntz+non-Cuntz=0},  $N_\la=N_\sigma=0$, we can define the Simon's contractive co-embedding
\begin{eqnarray*} 
    E_{\la}=E_{\la, \sigma}: \mathbb{H}_{d}^{2}(\sigma) &\longrightarrow& \mathbb{H}_{d}^{2}(\la)\\
    f+N_{\sigma} &\longrightarrow& f+N_{\la},
\end{eqnarray*}
and the projections 
$$Q_{ac}=P_{(\mathrm{ker}(E_\la))^{\perp}}, \quad Q_s=P_{\mathrm{ker}(E_\la)}.$$

Now, Simon's formulas for absolutely continuous and singular parts of $q_{\mu}$ reads as
\begin{eqnarray} \label{Simon's-formulas}
    (q_{\mu})_{ac} (a+N_{\la}, b+N_{\la})&=& \langle a+N_{\sigma}, (Q_{ac} -E_{\la}^{*}E_{\la})(b+N_{\sigma}) \rangle_{\sigma}\\
    (q_{\mu})_{s} (a+N_{\la}, b+N_{\la})&=& \langle a+N_{\sigma}, Q_s (b+N_{\sigma}) \rangle_{\sigma}. \nonumber
\end{eqnarray}
Note that the Simon derivative $D=(Q_{ac} -E_{\la}^{*}E_{\la}) $ is bounded on $\mathbb{H}_{d}^{2}(\sigma)$ since $\la \leq \sigma$, $e_{\sigma, \la}: H^{+}(H_\la) \rightarrow H^{+}(H_\sigma)$ is contraction, and $E_{\la}=E_{\la, \sigma}=\mathscr{C}_{\la}^{*} e_{\sigma, \la}^{*} \mathscr{C}_{\sigma}$.  By Simon's decomposition, $q_{\mu} = (q_{\mu})_{ac} +(q_{\mu})_{s} $ the maximal form decomposition of $q_{\mu}$  into absolutely continuous and singular parts. However, one must show that the forms $(q_{\mu})_{ac}$ and $(q_{\mu})_{s}$ are arising from  measures. Our next task is to prove that 
\begin{eqnarray}
    \hat{q}_{ac}:\mathbb{A}_{d}+\mathbb{A}_{d}^{*} &\longrightarrow& \mathbb{C} \nonumber \\
    \hat{q}_{ac}(a^{*}+b)&:=&(q_{\mu})_{ac}(a+N_{\la},I+N_{\la}) + (q_{\mu})_{ac} (I+N_{\la},b+N_{\la}) \nonumber \\
    &=&(q_{\mu})_{ac}(a,1) + (q_{\mu})_{ac} (1,b);     \quad a \in \mathbb{A}_d , \quad ,b \in \mathbb{A}_{d}^{(0)}
\end{eqnarray}
defines a positive NC measure on $\mathcal{A}_d$. Here, $\mathbb{A}_{d}^{(0)}$ denotes those members of $\mathcal{A}_d$ which vanish at zero.  Note that since $N_{\la}=0$, we we may use $a$ instead $a+N_{\la}$ or write 1 in place of the identity operator $I$, but we will write the isotropic ideals $N_\la$ and $N_\sigma$ when there is a danger of confusion.  We first show $\hat{q}_{ac}$ is a bounded linear functional on $\mathbb{A}_{d}+\mathbb{A}_{d}^{*}$. Linearity is obvious, and we need to show that it is bounded on $\mathbb{A}_{d}+\mathbb{A}_{d}^{*}$ since it can extend by continuity to a bounded linear functional on the free disk operator system $\mathcal{A}_d=\overline{\mathbb{A}_{d}+\mathbb{A}_{d}^{*}}$. Note that $\pi_{\sigma}$ is a unital *-representation of the Cuntz-Toeplitz C*-algebra $\mathcal{E}_{d}$; hence,  it  is completely contarctive. By Proposition \ref{reducing-projections}, $Q_{ac} \in vN(\Pi_{\sigma})^{'}$, so by applying formula \ref{Simon's-formulas} for $(q_{\mu})_{ac}$ we see:
\begin{eqnarray*}
    \left|\hat{q}_{ac}(a^{*}+b)\right|&=&\left| 
(q_{\mu})_{ac}(a,1) + (q_{\mu})_{ac} (1,b) \right|\\
&=&\left| \langle a+N_{\sigma}, (Q_{ac} -E_{\la}^{*}E_{\la})(I+N_{\sigma}) \rangle_{\sigma}+ \langle I+N_{\sigma}, (Q_{ac} -E_{\la}^{*}E_{\la})(b+N_{\sigma}) \rangle_{\sigma} \right|  \\
&=&\arrowvert \langle \pi_{\sigma}(a)(I+N_{\sigma}), Q_{ac}(I+N_{\sigma}) \rangle_{\sigma} -\langle a+N_{\la}, I+N_{\la} \rangle_{\la}\\
& &+\langle I+N_{\sigma}, Q_{ac}\pi_{\sigma}(b)(I+N_{\sigma}) \rangle_{\sigma} -\langle I+N_{\la}, b+N_{\la} \rangle_{\la} \arrowvert \\
&=& \arrowvert \langle Q_{ac}(I+N_{\sigma}), \pi_{\sigma}(a^{*})Q_{ac}(I+N_{\sigma}) \rangle_{\sigma} +\langle Q_{ac}(I+N_{\sigma}), \pi_{\sigma}(b) Q_{ac}(I+N_{\sigma}) \rangle_{\sigma} \\
& & -\langle I+N_{\la}, [ \pi_{\la}(a^{*})+ \pi_{\la}(b)] (I+N_{\la}) \rangle_{\la} \arrowvert \\
&=&\arrowvert \langle Q_{ac}(I+N_{\sigma}), [\pi_{\sigma}(a^{*})+\pi_{\sigma}(b)]Q_{ac}(I+N_{\sigma}) \rangle_{\sigma}  \\
& & -\langle I+N_{\la}, [ \pi_{\la}(a^{*})+ \pi_{\la}(b)] (I+N_{\la}) \rangle_{\la} \arrowvert \\
& \leq & \lVert \pi_{\sigma}(a^{*}+b) \rVert \lVert Q_{ac}(I+N_{\sigma}) \rVert^{2}+ \lVert \pi_{\la}(a^{*}+b) \rVert \lVert I+N_{\la} \rVert^{2}\\
& \leq &  \lVert {a(L)}^{*}+b(L) \rVert (\mu(I) +2\la(I) )
\end{eqnarray*}
So, $\hat{q}_{ac}$ defines a bounded linear functional on ${\mathbb{A}_{d}}^{*} +\mathbb{A}_{d}$. Note that we have used the facts that $\lVert I+N_{\sigma} \rVert^{2}=\langle I+N_{\sigma} , I+N_{\sigma} \rangle_{\sigma}=\sigma(I)=\mu(I)+\la(I)$ and $\lVert I+N_{\sigma} \rVert^{2}=\la(I)$ in the final calculations above.

We now show that $\hat{q}_{ac}$ is positive. We first show it is self-adjoint. Let $a \in \mathcal{A}_d$. Since the form $(q_{\mu})_{ac}$ is positive semi-dfinite, we have
$$\overline{\hat{q}_{ac}(a^{*})}= \overline{(q_{\mu})_{ac}(a, 1)}=(q_{\mu})_{ac}(1, a).$$
By definition of $\hat{q}_{ac}$,
$$\hat{q}_{ac}(a)=\hat{q}_{ac}(a(0)1+[a-a(0)]1)=(q_{\mu})_{ac}(\overline{a(0)}1, 1)+(q_{\mu})_{ac}(1, [a-a(0)]1)=(q_{\mu})_{ac}(1, a).$$
Upon comparing the last two equations, we see that $\overline{\hat{q}_{ac}(a^{*})}=\hat{q}_{ac}(a)$, which shows that $\hat{q}_{ac}$ is self-adjoint.

We claim that the Simon derivative $D=Q_{ac} -E_{\la}^{*}E_{\la}$ is $\Pi_{\sigma}$-Toeplitz. Since $D$ is bounded, we must show that 
\begin{eqnarray} \label{Toeplitz-property-Pi_sigma}
    \Pi_{\sigma, i}^{*} D \Pi_{\sigma,j}=\delta_{i,j} D; \quad 1 \leq i, j \leq d.
\end{eqnarray}
Recall that the contractive co-embedding $E_{\la}$ is an intertwiner between $\Pi_{\sigma}$ and $\Pi_{\la}$, and by Proposition \ref{reducing-projections},  $Q_{ac}$ commutes with components of $\Pi_{\sigma}$. So, 
\begin{eqnarray*}
    \Pi_{\sigma, i}^{*} D \Pi_{\sigma,j}&=& \Pi_{\sigma, i}^{*} (Q_{ac} -E_{\la}^{*}E_{\la}) \Pi_{\sigma,j}\\
    &=& \Pi_{\sigma, i}^{*} \Pi_{\sigma,j} Q_{ac}-E_{\la}^{*} \Pi_{\la, i}^{*} \Pi_{\la,j}E_{\la}\\
    &=&\delta_{i,j}I_{\sigma}Q_{ac}-E_{\la}^{*} \delta_{i,j}I_{\la}E_{\la}\\
    &=&\delta_{i,j}(Q_{ac}-E_{\la}^{*}E_{\la})\\
    &=&\delta_{i,j} D.
\end{eqnarray*}

We are now left with the task of proving that $\hat{q}_{ac}$ is positive. To do so, by \cite[Lemma 4.6]{JMS-ncratClark} we notice that any positive element of $\mathcal{A}_d$ is the norm limit of $\sum_{1}^{N}{p_n(L)}^{*}p_n(L)$ where each $p_n \in \mathbb{C}\{Z\}$ is a free polynomial in $d$ variables. Thus, we only need to check that $\hat{q}_{ac}({p(L)}^{*}p(L))$ is positive semi-definite for any free polynomial. Let $p$ be any free polynomial of homogenous degree $N$. Then, by using the semi-Dirichlet property \ref{semi-Dirichlet}, one can find free polynomial $u$ of degree at most $N$ such  that ${p(L)}^{*}p(L)=u(L)^{*}+u(L)$, see \cite[Theorem 4]{JM-ncFatou}.  So,
\begin{eqnarray} \label{middle-formula-Toeplitz}
    \hat{q}_{ac}({p(L)}^{*}p(L)) &=& \hat{q}_{ac}(u(L)^{*}+u(L))\nonumber \\
    &=& (q_{\mu})_{ac}(u(L),1) + (q_{\mu})_{ac} (1,u(L))\nonumber \\
    &=& \langle u+N_{\sigma}, D(I+N_{\sigma}) \rangle_{\sigma} +\langle I+N_{\sigma}, D(u+N_{\sigma}) \rangle_{\sigma} \nonumber \\
    &=& \langle u(\Pi_\sigma)(I+N_{\sigma}), D(I+N_{\sigma}) \rangle_{\sigma} +\langle I+N_{\sigma}, Du(\Pi_\sigma)(I+N_{\sigma}) \rangle_{\sigma} \nonumber \\
    &=& \langle I+N_{\sigma}, [u(\Pi_\sigma)^{*} D+ Du(\Pi_\sigma)](I+N_{\sigma}) \rangle_{\sigma}.
\end{eqnarray}
However, since $D$ is $\Pi_{\sigma}$-Toeplitz,  like the proof of \cite[Theorem 4]{JM-ncFatou} we have
$$p(\Pi_\sigma)^{*} D p(\Pi_\sigma)=u(\Pi_\sigma)^{*} D+ Du(\Pi_\sigma).$$
By combining this with equation \ref{middle-formula-Toeplitz}, we see that 
\begin{eqnarray*}
    \hat{q}_{ac}({p(L)}^{*}p(L)) &=& \langle I+N_{\sigma}, p(\Pi_\sigma)^{*} D p(\Pi_\sigma)(I+N_{\sigma}) \rangle_{\sigma}\\
    &=& \langle p(\Pi_\sigma)(I+N_{\sigma}), D p(\Pi_\sigma)(I+N_{\sigma}) \rangle_{\sigma}\\
    &=&  \langle p+N_{\sigma}, D( p+N_{\sigma}) \rangle_{\sigma}\\
    &=& (q_{\mu})_{ac}(p,p) \geq 0.
\end{eqnarray*}
Note that positivity of $\hat{q}_{ac}$ totally depends on $\Pi_\sigma$-Toeplitz property of $D$, and this property can hold without any knowledge about reducing projections.  Anyway, $\hat{q}_{ac}$ is a positive NC measure, so we define $\mu_{ac}:=\hat{q}_{ac}$ and call it the absolutely continuous part of $\mu$ with respect to $\la$. This naming is justified by the fact that $q_{{\mu}_{ac}}=(q_{\mu})_{ac}$ is closed by construction. The measure $\mu_{ac}$ is also the maximal absolutely continuous positive NC measure less than $\mu$, and this is because the Simon construction gives the maximal closeable form. Also, by a similar discussion the singular part $(q_{\mu})_{s}$ induces a positive NC measure which we call the singular part of $\mu$, and we denote it by $\mu_s$. This nomenclature is justified by that fact that $(q_{\mu})_{s}$ is singular, and it can not majorize any positive NC measure which is absolutely continuous. Also, note that by Simon formulas, we have $\mu_{s}=\mu-\mu_{ac}$.
\ep

\begin{remark} \label{Toeplitz-form-and-measure}
Notice that that positivity of $\hat{q}_{ac}$ in the end of proof of Theorem \ref{form-decomposition-nonCE}  depends merely on $\Pi_\sigma$-Toeplitz property of $D$ not on the reducing properties of Simon's projections. It might happen that $D$ is $\Pi_\sigma$-Toeplitz without Simon's projections being reducing, and $\hat{q}_{ac}$ still gives a measure. For instance by taking into account the notations of Proposition \ref{Pi-Toeplitz-ops-GNS-la}, let $\la$ be any positive NC measure, and  suppose that $T$ is a closed, positive, semi-definite and $\Pi_{\la}$-Toeplitz operator such that free polynomials are a core for $\sqrt{T}$. Define
\begin{eqnarray*}
    \mu: {\mathbb{A}_d}^{*}+\mathbb{A}_d &\longrightarrow& \mathbb{C} \\
    \mu(p^* +q)&=&\langle \sqrt{T}(p+N_{\la}), \sqrt{T}(1+N_{\la})\rangle_{\la} +\langle \sqrt{T}(1+N_{\la}), \sqrt{T}(q+N_{\la})\rangle_{\la} \\
    \mu(p^* q)&=&\langle \sqrt{T}(p+N_{\la}), \sqrt{T}(q+N_{\la}) \rangle_{\la},
\end{eqnarray*}
then by the same discussion as in the end of proof of Theorem \ref{form-decomposition-nonCE}, we see that $\mu$ is positive. Also, note that
$$\mu(1)=\|\sqrt{T}(1+N_{\la})\|^{2},$$
and by \cite[Corollary 2.8]{Paulsen-cb}, we see that $\| \mu\| \leq \mu(1)$, i.e., $\mu$ is bounded. Hence, $\mu$ can be extended to a positive bounded linear functional on the closure of ${\mathbb{A}_d}^{*}+\mathbb{A}_d$. Thus, $\mu$
defines a positive linear functional on $\mathcal{A}_d$. So, by Riesz lemma we conclude that any closed Toeplitz form induces a measure.$\blacksquare$

\end{remark}

\begin{cor} \label{ac-sing-equiv-form-sense}
   Let $\mu$ and $\la$ be two positive NC measures on $\mathcal{A}_d$, and let $\la$ be non-Cuntz. Following the notations of Theorem \ref{form-decomposition-nonCE}, we see that:
   \begin{enumerate}[(i)]
   \item $(q_{\mu}^{\la})_{ac}=q_{{\mu}_{ac}}^{\la}$ and $(q_{\mu}^{\la})_{s}=q_{{\mu}_{s}}^{\la}$,
       \item $\mu \ll_{F} \la  \Leftrightarrow q_{\mu}^{\la} ~\mathrm{ is ~closable} \Leftrightarrow q_{\mu}^{\la}=(q_{\mu}^{\la})_{ac} \Leftrightarrow Q_s =0 \Leftrightarrow Q_{ac}=I_{\mathbb{H}^{2}_{d}(\sigma)}$ ,
       \item $\mu \perp_{F} \la \Leftrightarrow q_{\mu}^{\la} ~\mathrm{ is~ singular} \Leftrightarrow q_{\mu}^{\la}=(q_{\mu}^{\la})_{s} \Leftrightarrow Q_{ac} =E_{\la}^{*}E_{\la} \Leftrightarrow Q_{s}=I_{\mathbb{H}^{2}_{d}(\sigma)}-E_{\la}^{*}E_{\la}  $.
   \end{enumerate}
\end{cor}
The following corollary is a generalization of \cite[Corollary 1 and Corollary 2]{JM-ncFatou} with basically the same proof; however,  for the sake of completeness and future reference we will give the complete proof.
\begin{cor} \label{RK-Form-AC-equiv}

Let $\mu$ and $\la$ be two positive NC measures on $\mathcal{A}_d$ with  $\la$ being  non-Cuntz. Then,
\begin{enumerate}[(i)]
    \item $\mu \ll_{F} \la$ if and only if $\mu+\la \ll_{R} \la$.
    \item $\mu \perp_{F} \la$ if and only if  $\mathbb{H}_{d}^{2}(\mu+\la)= \mathbb{H}_{d}^{2}(\mu) \oplus \mathbb{H}_{d}^{2}(\la)$.
\end{enumerate}

\end{cor} 
\bp
(i) Let $\sigma=\mu+\la$. By Lemma \ref{ac-sing-equiv-form-sense}, we know that $\mu \ll_{F} \la$ if and only if $Q_{ac}=I_{\mathbb{H}^{2}_{d}(\sigma)}$. However, $\mathrm{Ran}(Q_{ac})=\overline {\mathrm{Ran}(E_{\la}^{*})}^{\sigma}$. So, $Q_{ac}=I_{\mathbb{H}^{2}_{d}(\sigma)}$ if and only if $E_{\la}^{*}$ has  dense range. The latter happens if and only if $e_{\la}=\mathscr{C}_{\sigma} E_{\la}^{*} \mathscr{C}_{\la}^{*} $ has dense range. That is, $\mathrm{int}(\sigma, \la)$ is dense in $\mathbb{H}^{2}_{d}(\sigma)$.

(ii) Note that $ \mu , \la \leq \sigma = \mu+\la $. So, we can define the contractive co-embeddings $E_\mu$ and $E_\la$ with dense ranges, and also $E_{\mu}^{*}E_{\mu}+E_{\la}^{*}E_{\la}=I_{\mathbb{H}^{2}_{d}(\sigma)}$ . By Lemma \ref{ac-sing-equiv-form-sense},  $\mu \perp_{F} \la$ if and only if $Q_{ac}=E_{\la}^{*}E_{\la}$, or equivalently,  $Q_s=I_{\mathbb{H}^{2}_{d}(\sigma)} - Q_{ac}=I_{\mathbb{H}^{2}_{d}(\sigma)} -E_{\la}^{*}E_{\la}=E_{\mu}^{*}E_{\mu} $. That is, $E_{\la}^{*}E_{\la}$ and $E_{\mu}^{*}E_{\mu}$ are projections; thus, $E_{\la}$ and $E_{\mu}$ are partial isometries onto their ranges. This means that $\mathrm{Ran}(Q_{ac}) \cong \mathbb{H}^{2}_{d}(\la) $ and $\mathrm{Ran}(Q_{s}) \cong \mathbb{H}^{2}_{d}(\mu)$, and by Simon's construction $\mathrm{Ran}(Q_{ac})$ and $\mathrm{Ran}(Q_{s})$ are orthogonal to each other.
\ep

When $\la$ is a non-Cuntz measure, by Theorem \ref{left-ideal-non-Cuntz=0} we know that $N_\la =0$. So, by  Remark \ref{GNS-of-non-Cuntz-vs-m} we can write $\mathbb{A}_{d}$ in place of $\mathbb{A}_{d} +N_{\la}$, and also  $\mathbb{C}\{Z_1,\cdots, Z_d\}$ in place of $\mathbb{C}\{Z_1,\cdots, Z_d\}+N_\la$. Thus, the definition of $\Pi_\la$-Toeplitz is essentially the same as \cite[Definition 5.2]{JM-ncld}, and this leads us to the following proposition which is a generalization of \cite[Lemma 2]{JM-ncld}.

\begin{prop} \label{Pi-Toeplitz-ops-GNS-la}
 Let $\la$ be a positive NC measure and $T$ be a closed, positive, semi-definite operator  in $\mathbb{H}^{2}_{d}(\la)$. Suppose that $T$ is $\Pi_{\la}$-Toeplitz and $\mathbb{A}_{d} +N_{\la}$ is a core for $\sqrt{T}$. Then, $\mathrm{Dom}(\sqrt{T})$ is $\Pi_{\la}$-invariant and $\mathbb{C}\{Z_1,\cdots, Z_d\}+N_\la$ is a core for $\sqrt{T}$.
\end{prop}
\bp
Let $x \in \mathrm{Dom(\sqrt{T})}$. Since $\mathbb{A}_{d}+N_{\la}$ is a core for $\sqrt{T}$, we can find  $a_n \in \mathbb{A}_{d}$  so that $a_n+N_{\la} \longrightarrow x$ in $\mathbb{H}^{2}_{d}(\la)$ and $\sqrt{T}(a_n+N_{\la}) \longrightarrow \sqrt{T}x$. For each $1 \leq k \leq d$, note that $\Pi_{\la,k}=\Pi_{\la}(L_k)$ is a bounded operator on $\mathbb{H}^{2}_{d}(\la)$  which acts by left multiplication through $L_k$ on $\mathbb{A}_{d}+N_{\la}$. However,  $\mathbb{A}_{d}$ is invariant under left multiplications; hence, $\Pi_{\la,k}(a_n+N_{\la}) \in \mathbb{A}_{d}+N_{\la} \subseteq  \mathrm{Dom}(\sqrt{T})$ and $ \Pi_{\la,k}(a_n +N_{\la})\longrightarrow \Pi_{\la,k}( x)$. So,
\begin{eqnarray*}
    \lVert \sqrt{T}( \Pi_{\la,k}(a_n+N_{\la}) -\Pi_{\la,k}(a_m+N_{\la}))  \rVert^{2}_{\mathbb{H}^{2}_{d}(\la)} &=&q_{T}^{\la}( \Pi_{\la,k}(a_n-a_m+N_{\la}), \Pi_{\la,k}(a_n-a_m+N_{\la}) )\\
    &=&q_{T}^{\la}(a_n-a_m+N_\la, a_n-a_m+N_\la)\\
    &=& \lVert \sqrt{T}( a_n -a_m+N_{\la})  \rVert^{2}_{\mathbb{H}^{2}_{d}(\la)}.
\end{eqnarray*}
The middle equality holds since $q_{T}^{\la}$ is $\Pi_{\la}$-Toeplitz. However, $\sqrt{T}(a_n+N_{\la}) $ is a Cauchy sequence, so from the above equations $\sqrt{T} \Pi_{\la,k}(a_n+N_{\la})$ is also Cauchy and must converge to some $y \in \mathbb{H}^{2}_{d}(\la)$. Since $\sqrt{T}$ is closed, we see that $\sqrt{T}\Pi_{\la,k} x =y$. So, $\Pi_{\la,k} x \in \mathrm{Dom}(\sqrt{T})$, and it follows that $\mathrm{Dom}(\sqrt{T})$ is $\Pi_{\la}$-invariant.

To show that $\mathbb{C}\{Z\}+N_{\la}$ is a core for $\sqrt{T}$, we proceed like \cite[Lemma 2]{JM-ncFatou} and show that the set $\{(p+N_{\la},\sqrt{T}(p+N_{\la})): p \in \mathbb{C}\{Z\} \} $
is dense in $\{(a+N_{\la},\sqrt{T}(a+N_{\la})): a \in \mathbb{A}_d \}$. So, let $a(L)+N_\la \in \mathbb{A}_d+N_\la $. Since $\mathbb{A}_d=\overline{\mathbb{C}\{L\}}^{op}$, we can find a sequence of free polynomials $p_n(L)$ such that $p_n(L)+N_\la \longrightarrow a(L)+N_\la$ in the operator norm. In particular, $p_n +N_\la:=p_n(L)1 +N_\la \longrightarrow a(L)1+N_\la=a(L)1+N_{\la}=a+N_{\la}$ in $\mathbb{H}^{2}_{d}(\la)$. For each $1\leq k \leq d$, we define the linear map 
\begin{eqnarray*}
    \Pi_k : \mathrm{Ran}\sqrt{T} \longrightarrow \mathbb{H}^{2}_{d}(\la) \\
    \Pi_k \sqrt{T}x=\sqrt{T} \Pi_{\la,k}x
\end{eqnarray*}
This is well-defined since $\mathrm{Dom}(\sqrt{T})$ is $\Pi_{\la}$-invariant, and note that for any $a\in \mathbb{A}_d$ we have $\sqrt{T} \Pi_{\la,k}(a+N_{\la})=\sqrt{T} (L_k a+N_{\la})$. Also,  $\Pi=(\Pi_1, \cdots, \Pi_d)$ extends to a row isometry on $\mathcal{H}_{T}(\la):=\overline{\mathrm{Ran}(\sqrt{T})}^{\la} $ because the form $q_T$ is $\Pi_{\la}$-Toeplitz and for $x,y \in \mathrm{Dom}(\sqrt{T}) $
\begin{eqnarray*}
    \langle \Pi_k \sqrt{T}x, \Pi_j \sqrt{T}y \rangle &=&   \langle  \sqrt{T}\Pi_{\la,k}x, \sqrt{T}\Pi_{\la,j}y \rangle \\
    &=& \delta_{j,k} \langle  \sqrt{T}x, \sqrt{T}y \rangle.
\end{eqnarray*}
So, $ \Pi_{j}^{*} \Pi_k=\delta_{j,k} I_{\mathbb{H}^{2}_{d}(\la)}$.  Furthermore, by the very definition of $\Pi_k$, for any $a\in \mathbb{A}_d$ we have 
$$p(\Pi_k)\sqrt{T}(a+N_{\la})=\sqrt{T}p(\Pi_{\la,k})(a+N_{\la})=\sqrt{T}(p(L_k)a+N_{\la}).$$ 
Note that this is just true for elements of $\mathbb{A}_d +N_\la$.  Now, we see that
\begin{eqnarray*}
    \lVert \sqrt{T}(p_n - p_m+N_{\la}) \rVert_{\mathbb{H}^{2}_{d}(\la)} &=& \lVert \sqrt{T}([p_n (L) - p_m(L)]1+N_{\la}) \rVert_{\mathbb{H}^{2}_{d}(\la)}\\
    &=& \lVert (p_n (\Pi) - p_m(\Pi))\sqrt{T}(1+N_{\la}) \rVert_{\mathbb{H}^{2}_{d}(\la)}\\
    &\leq& \lVert (p_n (\Pi) - p_m(\Pi)) \rVert \lVert \sqrt{T}(1+N_{\la}) \rVert \\
    &\leq& \lVert (p_n (L) - p_m(L)) \rVert \lVert \sqrt{T}(1 +N_{\la})\rVert. \\
    \end{eqnarray*}
    The last inequality is an application of Popescu-von Neumann inequality \cite{Pop-vN-inq}.  Anyway, $p_n(L)$ is convergent to $a(L)$, so it is Cauchy; hence, by the above inequality, $(\sqrt{T}[p_n+N_{\la}])$ must converge to some $y \in \mathbb{H}^{2}_{d}(\la) $. However,  $p_n(L)+N_{\la} \longrightarrow a(L)+N_{\la}$ and  $\sqrt{T}$ is closed. It follows $\sqrt{T}(a+N_{\la})=y$. Consequently, $\mathbb{C}\{Z\}+N_{\la}$ is a core for $\sqrt{T}$.
\ep
\begin{prop} \label{transfer-of-the-domain}
    Assume that $\mu$ and $\la$ are two positive NC measures such that $\mu \ll_{F} \la$. Then there is a  closed, positive, semi-definite and $\Pi_\la$-Toeplitz operator $T$ such that  $\mathbb{A}_d +N_\la$ is a core for 
$\sqrt{T}$ and $\overline{q_{\mu}^{\la}}=q_{T}^{\la}$. Also, $\mathrm{Dom}(T)$ can be  identified  with a dense subspace $\mathcal{D}_{\mu}(T)$ of  $\mathbb{H}^{2}_{d}(\mu)$ such that $\mathscr{C}_{\mu}\mathcal{D}_{\mu}(T) \subseteq \mathscr{H}^{+}(H_{\mu}) \cap \mathscr{H}^{+}(H_{\la}) $.
\end{prop}
\bp
Note that $\overline{q_{\mu}^{\la}}$ is closed, so  a closed, positive, semi-definite operator $T$ exist by \cite[Chapter VI, Theorem 2.1, Theorem 2.23]{Kato}. Also  by the definition, the hybrid form $\overline{q_{\mu}^{\la}}$ has the domain $\mathbb{A}_d +N_\la$, and by Lemma \ref{hybrid-form-Toeplitz} it is $\Pi_\la$-Toeplitz. So, $T$ is $\Pi_\la$-Toeplitz, and  $\mathbb{C}\{Z\}+N_{\la}$ is a core for $\sqrt{T}$ by Proposition \ref{Pi-Toeplitz-ops-GNS-la}. 

Next, we will show how we can lift  $\mathrm{Dom}(T)$ onto a subset of $\mathbb{H}^{2}_{d}(\mu)$ denoted by $\mathcal{D}_{\mu}(T)$. Recall that $N_\la \subseteq N_\mu$ by assumption, hence we can define the following mapping:
\begin{eqnarray*}
    E_{\mu ,\la}: \mathbb{H}^{2}_{d}(\la) &\longrightarrow &\mathbb{H}^{2}_{d}(\mu)\\
    a+N_\la & \longrightarrow & a+N_\mu.
\end{eqnarray*}
Put $\mathcal{D}_{\mu}(T)=E_{\mu ,\la} (\mathrm{Dom}(T))$, and we will show that $\mathcal{D}_{\mu}(T)$ has the desired properties. Let's $x \in \mathrm{Dom}(T) \subseteq \mathrm{Dom}(\sqrt{T}) $. Since $\mathbb{C}\{Z\}+N_\la$ is a core for $\sqrt{T}$, there is a sequence of free polynomials $(p_n)$ such that 
$$p_n +N_\la \longrightarrow x, \quad \mathrm{and} \quad \sqrt{T} (p_n+N_\la) \longrightarrow \sqrt{T}x.$$
Note that 
\begin{eqnarray*} 
    \langle \sqrt{T} (p_n-p_m+N_\la), \sqrt{T} (p_n-p_m+N_\la) \rangle_{\mathbb{H}^{2}_{d}(\la)} &=&q_{\mu}^{\la}(p_n-p_m+N_\la, p_n-p_m+N_\la)\\
    &=& \langle p_n-p_m+N_\mu, p_n-p_m+N_\mu \rangle_{\mathbb{H}^{2}_{d}(\mu)}.
\end{eqnarray*}
In particular, 
$$ \lVert p_n-p_m+N_\mu \rVert_{\mu}=\lVert \sqrt{T} (p_n-p_m+N_\la) \rVert_{\la} \longrightarrow 0,$$
whence the sequence $(p_n+N_\mu)$ is Cauchy in $\mathbb{H}^{2}_{d}(\mu)$, so it converges to an $\hat{x} \in \mathbb{H}^{2}_{d}(\mu)$. This shows that we can identify $\mathrm{Dom}(T)$ with a subspace $\mathcal{D}_{\mu}(T)$ of $\mathbb{H}^{2}_{d}(\mu)$.

We now claim that $\hat{y} \in \mathcal{D}_{\mu}(T)$ implies that $\mathscr{C}_{\mu}(\hat{y}) \in \mathscr{H}^{+}(H_{\la}) $. When $\hat{y} \in \mathcal{D}_{\mu}(T)$, there is a $y \in \mathrm{Dom}(T)$ such that $\hat{y}=E_{\mu, \la}(y)$. Also, there is a sequence of free polynomials $(p_n)$ such that $p_n+N_\la \longrightarrow y $, $\sqrt{T}p_n+N_\la \longrightarrow \sqrt{T}y$ in $\mathbb{H}^{2}_{d}(\la)$, and $p_n+N_\mu \longrightarrow \hat{y}$ in $\mathbb{H}^{2}_{d}(\mu)$. The free Cauchy transform of $\hat{y}$ is then
\begin{eqnarray*} 
\mathscr{C}_{\mu}(\hat{y})(Z)&=&\sum_{\alpha \in \mathbb{F}^{d}} Z^{\alpha} \langle L^{\alpha}1+N_\mu,\hat{y} \rangle_{\mu}\\
    &=& \lim_{n \to \infty} \sum_{\alpha \in \mathbb{F}^{d}} Z^{\alpha} \langle L^{\alpha}1+N_\mu,p_n(L)1+N_\mu \rangle_{\mu}\\
    &=& \lim_{n \to \infty} \sum_{\alpha \in \mathbb{F}^{d}} Z^{\alpha} q_{\mu}^{\la}(L^{\alpha}1+N_\la,p_n(L)1+N_\la)\\
    &=& \lim_{n \to \infty} \sum_{\alpha \in \mathbb{F}^{d}} Z^{\alpha} \langle \sqrt{T}(L^{\alpha}1+N_\la),\sqrt{T}(p_n(L)1+N_\la) \rangle_{\la}\\
    &=&  \sum_{\alpha \in \mathbb{F}^{d}} Z^{\alpha} \langle L^{\alpha}1+N_\la,Ty \rangle_{\la}\\
    &=& \mathscr{C}_{\la}(Ty)(Z).
\end{eqnarray*}
Note that in the above, we have used the fact that the series converges uniformly in sub-balls of the unit row ball, so the limit and the sum can commute. Recall that $\mathscr{C}_{\mu}(\hat{y}) \in \mathscr{H}^{+}(H_{\mu}) $ and $\mathscr{C}_{\la}(Ty) \in \mathscr{H}^{+}(H_{\la})$. Thus, from the above calculations it follows that $\mathscr{C}_{\mu}\mathcal{D}_{\mu}(T) \subseteq \mathscr{H}^{+}(H_{\mu}) \cap \mathscr{H}^{+}(H_{\la}) $.
\ep

In he following theorems we will prove that $AC_{R}[\la]=AC_{F}[\la]$  and $SG_{R}[\la]=SG_{F}[\la]$.
\begin{thm} \label{RKHS-AC=Form-AC}
    Let $\mu$ and $\la$ be two positive NC measures, with  $\la$ being non-Cuntz. Then, $\mu \ll_{R} \la$ if and only if $\mu \ll_{F} \la$. In other words, $AC_{R}[\la]=AC_{F}[\la]$.
\end{thm}
\bp
Assume that $\mu \ll_{R} \la$. Obviously, $\la \ll_{R} \la$, so  $\mu + \la \ll_{R} \la$ by Theorem \ref{positive-cone-AC-{R}(la)}. Thus,  $\mu \ll_{F} \la$ by Corollary \ref{RK-Form-AC-equiv}. Conversely, suppose that  $\mu \ll_{F} \la$, i.e., $q_{\mu}^{\la}$ is a closeable form in $\mathbb{H}^{2}_{d}(\la)$. Now, by Proposition \ref{transfer-of-the-domain}, there is an operator $T$ such that $\mathscr{C}_{\mu}\mathcal{D}_{\mu}(T) \subseteq \mathscr{H}^{+}(H_{\mu}) \cap \mathscr{H}^{+}(H_{\la}) $. However, $\mathrm{Dom}(T)$ contains $\mathbb{A}_{d}+N_{\la}$, and it is a core for $\sqrt{T}$. So, $\mathcal{D}_{\mu}(T)$ contains $\mathbb{A}_{d}+N_{\mu}$; hence, $\mathcal{D}_{\mu}(T)$ must be dense in $\mathbb{H}^{2}_{d}(\mu)$. Since $\mathscr{C}_{\mu}$ is unitary and $\mathscr{C}_{\mu}\mathcal{D}_{\mu}(T) \subseteq \mathscr{H}^{+}(H_{\mu}) \cap \mathscr{H}^{+}(H_{\la}) $, it follows that the intersection space 
$$\mathscr{H}^{+}(H_{\mu}) \cap \mathscr{H}^{+}(H_{\la}) $$
is dense in $\mathscr{H}^{+}(H_{\mu})$, which proves that $\mu \ll_{R} \la$.
\ep 
\begin{remark}
    Note that in the above theorem, the direction $\mu \ll_{F} \la \Rightarrow  \mu \ll_{R} \la$ does not depend on $\la$ being non-Cuntz, i.e., it is true for any splitting positive NC measure.
\end{remark}
\begin{lemma} \label{Lebesgue-RKHS-decomposition-of-sum}
    Let $\mu$ and $\la$ be two positive NC measures. Put $\sigma=\mu+\la$, and suppose that $\mu=\mu_{ac}^{R}+\mu_{s}^{R}$ is the maximal Lebesgue RKHS decomposition of $\mu$ with respect to $\la$. Also, assume that $\sigma=\sigma_{ac}^{R}+\sigma_{s}^{R}$ is  the maximal Lebesgue RKHS decomposition of  $\sigma$  with respect to $\la$. Then,
    $$\sigma_{ac}^{R}=\mu_{ac}^{R} +\la, \quad 
    \sigma_{s}^{R}=\mu_{s}^{R}.$$
    That is, $(\mu+\la)_{ac}=\mu_{ac}+\la$ and  $(\mu+\la)_{s}=\mu_{s}$ provided that we have necessary Lebesgue decompositions in advance.
\end{lemma}
\bp
By $\sigma=\mu_{ac}^{R}+\mu_{s}^{R} +\la $ and \cite[Theorem 4.4]{JM-ncld}, we see
$$\mathscr{H}^{+}(H_{\sigma})=\mathscr{H}^{+}(H_{\mu_{ac}})+\mathscr{H}^{+}(H_{\mu_{s}})+\mathscr{H}^{+}(H_{\la}).$$
Since in $\sigma=\mu_{ac}^{R}+\mu_{s}^{R} +\la $ each summand   is less than $\sigma$, by \cite[Lemma 4.2]{JM-ncld} each corrosponding RKHS space is contractively contained in $\mathscr{H}^{+}(H_{\sigma})$. Meanwhile, $\mu_{ac}+\la \in AC_{R}[\la] $ by Theorem \ref{positive-cone-AC-{R}(la)}. However, by maximality of $\sigma_{ac}=(\mu+\la)_{ac}$, it follows $\mu_{ac}+\la \leq (\mu+\la)_{ac} =\sigma_{ac}$. To show the converse inequality, we will prove that $\mathscr{H}^{+}(H_{\sigma_{ac}})=\overline{\mathscr{H}^{+}(H_{\sigma_{ac}}) \cap \mathscr{H}^{+}(H_{\la})}^{\sigma} \subseteq \mathscr{H}^{+}(H_{\mu_{ac}+\la})$. To this end, we will show that $\mathscr{H}^{+}(H_{\sigma_{ac}}) \cap \mathscr{H}^{+}(H_{\la}) \subseteq \mathscr{H}^{+}(H_{\mu_{ac}+\la}) $, and  $\mathscr{H}^{+}(H_{\mu_{ac}+\la})$ is closed in $\mathscr{H}^{+}(H_{\sigma})$.

To prove $\mathscr{H}^{+}(H_{\mu_{ac}+\la})$ is closed in $\mathscr{H}^{+}(H_{\sigma})$, we  establish that $\sigma=(\mu_{ac}+\la)+\mu_s$ is a Lebesgue RKHS decomposition of $\sigma$. From the Lebesgue RKHS decomposition  $\mu=\mu_{ac}^{R}+\mu_{s}^{R}$, we have 
$$\mathscr{H}^{+}(H_{\mu_{s}}) \cap \mathscr{H}^{+}(H_{\mu_{ac}})=\{0\},$$
and since $\mu_s \perp_{R} \la$, we have
$$\mathscr{H}^{+}(H_{\mu_{s}}) \cap \mathscr{H}^{+}(H_{\la})=\{0\}.$$
Now, we claim that 
$$\mathscr{H}^{+}(H_{\mu_{s}}) \cap \mathscr{H}^{+}(H_{\mu_{ac}+\la})=\{0\}.$$
First of all $\mathscr{H}^{+}(H_{\mu_{ac}+\la})=\mathscr{H}^{+}(H_{\mu_{ac}})+\mathscr{H}^{+}(H_{\la})$, so if $f \in \mathscr{H}^{+}(H_{\mu_{s}}) \cap \mathscr{H}^{+}(H_{\mu_{ac}+\la})$, then
$$f=f^{\mu}=g^{\mu}+h^{\la}; \quad g^{\mu}\in \mathscr{H}^{+}(H_{\mu_{ac}}), ~ h^{\la} \in \mathscr{H}^{+}(H_{\la}). $$
It follows that 
$$f^{\mu}-g^{\mu}=h^{\la} \in \mathscr{H}^{+}(H_{\mu}) \cap \mathscr{H}^{+}(H_{\la}) \subseteq \mathscr{H}^{+}(H_{\mu_{ac}}).  $$
Now,  $g^{\mu} \in \mathscr{H}^{+}(H_{\mu_{ac}}) $ and $f^{\mu}-g^{\mu} \in \mathscr{H}^{+}(H_{\mu_{ac}})$, so $f=g^{\mu}+(f^{\mu}-g^{\mu}) \in \mathscr{H}^{+}(H_{\mu_{ac}})$. However, $f \in \mathscr{H}^{+}(H_{\mu_{s}})$ by assumption; hence, $f \in \mathscr{H}^{+}(H_{\mu_{s}}) \cap \mathscr{H}^{+}(H_{\mu_{ac}})=\{0\} $, whence $f=0$. So, by \cite[Theorem 4.4]{JM-ncld} we have the direct sum decomposition:
\begin{eqnarray*}
    \mathscr{H}^{+}(H_{\sigma})=\mathscr{H}^{+}(H_{\mu_{ac}+\la})\oplus \mathscr{H}^{+}(H_{\mu_{s}}),
\end{eqnarray*}
which proves that $\mathscr{H}^{+}(H_{\mu_{ac}+\la})$ is closed in $\mathscr{H}^{+}(H_{\sigma})$. 

On the other hand, $\mathscr{H}^{+}(H_{\sigma_{ac}}) \cap \mathscr{H}^{+}(H_{\la})$ is dense in $\mathscr{H}^{+}(H_{\sigma_{ac}})$, and 
\begin{eqnarray*}
    \mathscr{H}^{+}(H_{\sigma_{ac}}) \cap \mathscr{H}^{+}(H_{\la}) &\subseteq &\mathscr{H}^{+}(H_{\la}) \\
    &\subseteq &\mathscr{H}^{+}(H_{\la}) +\mathscr{H}^{+}(H_{\mu_{ac}})\\
    &=& \mathscr{H}^{+}(H_{\mu_{ac}+\la}).
\end{eqnarray*}
So, $ \mathscr{H}^{+}(H_{\sigma_{ac}}) \cap \mathscr{H}^{+}(H_{\mu_{ac}+\la})$ must be  dense in  $\mathscr{H}^{+}(H_{\sigma_{ac}})$. Hence, from the above inclusions and closedness of $\mathscr{H}^{+}(H_{\mu_{ac}+\la})$,  we deduce that
$$\mathscr{H}^{+}(H_{\sigma_{ac}}) \subseteq \mathscr{H}^{+}(H_{\mu_{ac}+\la}),$$
i.e., $\sigma_{ac} \leq \mu_{ac}+\la $ as we wanted. Therefore, $\sigma_{ac} = \mu_{ac}+\la $, and from the Lebesgue RKHS decomposition
$$\mathscr{H}^{+}(H_{\sigma})= \mathscr{H}^{+}(H_{\sigma_{ac}})\oplus \mathscr{H}^{+}(H_{\sigma_{s}})$$
it follows that $\mathscr{H}^{+}(H_{\sigma_{s}})=\mathscr{H}^{+}(H_{\mu_{s}})$, i.e., $\sigma_s=\mu_s$.
\ep
\begin{thm} \label{RKHS-Sing=Form-Sing}
    Let $\mu$ and $\la$ be two positive NC measures, and let $\la$ be non-Cuntz. Then, $\mu \perp_{R} \la$ if and only if $\mu \perp_{F} \la$. In other words, $SG_{R}[\la]=SG_{F}[\la]$.
\end{thm}
\bp
It is enough to show that $\mu_{s}^{F}=\mu_{s}^{R}$, where $\mu_{s}^{F}$ is the singular part of $\mu$ obtained in the form sense of Theorem \ref{form-decomposition-nonCE}, and $\mu_{s}^{R}$ is the singular part of $\mu$ obtained in the RKHS sense of Theorem \ref{rkhs-non-CE}. Put $\sigma=\mu+\la$, and let $E_\la: \mathbb{H}^{2}_{d}(\sigma) \longrightarrow \mathbb{H}^{2}_{d}(\la) $ and $e_\la : \mathscr{H}^{+}(H_{\la}) \longrightarrow \mathscr{H}^{+}(H_{\sigma}) $ be the contractive co-embedding / embedding of Section 2. Recall that $E_\la = \mathscr{C}_{\la}^{*} e_{\la}^{*}\mathscr{C}_{\sigma} $, and  Cauchy transformations are unitary. So,
\begin{eqnarray*}
    x\in \mathrm{Ker}(E_\la)  &\Leftrightarrow & \mathscr{C}_{\la}^{*} e_{\la}^{*}\mathscr{C}_{\sigma}(x)=0_{\mathbb{H}^{2}_{d}(\sigma)} \\
    &\Leftrightarrow &  e_{\la}^{*}\mathscr{C}_{\sigma}(x)=0_{\mathscr{H}^{+}(H_{\la})} \\
    &\Leftrightarrow &  \mathscr{C}_{\sigma}(x) \in \mathrm{Ker}(e_{\la}^{*});
\end{eqnarray*}
hence, $\mathrm{Ker}(E_\la)=\mathscr{C}_{\sigma}^{*}\left(\mathrm{Ker}(e_{\la}^{*})\right)$.

Since, $\la$ is non-Cuntz, we have the RKHS decomposition $\sigma=\sigma_{ac}+\sigma_{s}$ by Theorem \ref{rkhs-non-CE}. Hence,  $(\mu+\la)_{ac}=\mu_{ac}+\la$ and $(\mu+\la)_{s}=\mu_{s}$ by Lemma \ref{Lebesgue-RKHS-decomposition-of-sum}. By \cite[Theorem 4.4]{JM-ncld}, $\mathscr{H}^{+}(H_{\mu_{ac}+\la})=\mathscr{H}^{+}(H_{\mu_{ac}})+\mathscr{H}^{+}(H_{\la})$. In addition, $\mathscr{H}^{+}(H_{\la})$ is contractively contained in $\mathscr{H}^{+}(H_{\mu_{ac}+\la})$, so $\mathscr{H}^{+}(H_{\la})=\mathscr{H}^{+}(H_{\la}) \cap \mathscr{H}^{+}(H_{\mu_{ac}+\la})$. By Theorem \ref{positive-cone-AC-{R}(la)}, $\mu_{ac}+\la \ll_{R} \la$; hence, $\mathscr{H}^{+}(H_{\la})$ is norm dense in $\mathscr{H}^{+}(H_{\mu_{ac}+\la})$. Also by Lemma \ref{Lebesgue-RKHS-decomposition-of-sum}, $\mathscr{H}^{+}(H_{\mu_{ac}+\la})=\mathscr{H}^{+}(H_{(\mu+\la)_{ac}})=\mathscr{H}^{+}(H_{\sigma_{ac}})$. These observations imply that the range of $e_\la : \mathscr{H}^{+}(H_{\la}) \longrightarrow \mathscr{H}^{+}(H_{\sigma}) $ is contained and norm dense in $\mathscr{H}^{+}(H_{\sigma_{ac}})$, i.e., 
$$\overline{\mathrm{Ran}(e_\la)}^{\sigma}=\mathscr{H}^{+}(H_{\sigma_{ac}})$$
By Lemma \ref{Lebesgue-RKHS-decomposition-of-sum}, $\mathscr{H}^{+}(H_{\sigma_{ac}})^{\perp}=\mathscr{H}^{+}(H_{\mu_{s}^{R}})$. This combined with the above equality implies
$$\mathrm{Ker}(e_{\la}^{*})={\mathrm{Ran}(e_\la)}^{\perp}=\mathscr{H}^{+}(H_{\sigma_{ac}})^{\perp}=\mathscr{H}^{+}(H_{\mu_{s}^{R}}).$$
Since the Cauchy transform is unitary, it follows $\mathrm{Ker}(E_\la)=\mathscr{C}_{\sigma}^{*}\left(\mathrm{Ker}(e_{\la}^{*})\right)=\mathbb{H}^{2}_{d}(\mu_{s}^{R})$. However, if we look at the form decomposition formulas in the proof of Theorem \ref{form-decomposition-nonCE}, the latter equality means that $\mu_{s}^{F}=\mu_{s}^{R}$.
\ep
\begin{cor}
   Let $\mu$ and $\la$ be two positive NC measures, and let $\la$ be non-Cuntz. Then, the maximal Lebesgue form decomposition and maximal  Lebesgue RKHS decomposition of  $\mu$ with respect to $\la$ are the same.
\end{cor}
\bp
Combine Theorems \ref{RKHS-AC=Form-AC} and \ref{RKHS-Sing=Form-Sing}.
\ep

If we have had the above theorem, or equivalently Theorems  \ref{RKHS-AC=Form-AC} and \ref{RKHS-Sing=Form-Sing}, in hand at the beginning, we could obtain  the maximal Lebesgue decomposition in one of the approaches (form or RKHS), and then we could say the maximal Lebesgue decomposition in another sense is the same. While this result is obtained after a long theory, in the next section we will see that for  splitting Cuntz measures we have such an equivalence at the beginning. That is  the Lebesgue decomposition against a Cuntz measure has the same meaning in both senses at the beginning; however, obtaining such a Lebesgue decomposition in either senses is much more involved than in the non-Cuntz case.

\section{Lebesgue decomposition with respect to a Cuntz measure}

So far we have generalized Jury-Martin decomposition to the case when the splitting measure $\la$ is non-Cuntz. In this section, we will obtain  the maximal Lebesgue  decomposition when the splitting measure $\la$ is Cuntz. In search of the AC part of $\mu$ with respect to $\la$, we have to give some other results along the way, then we will return to the decomposition problem after them. We should mention that some of the one-dimensional results in this section were developed  with Prof. Martin, Mr. Bal and the current author as in \cite{BMN}, and we will cite these results when mentioned. However, we should mention that our premises are different as we had started differently, and our approaches in proofs are somehow different. 

 \begin{defn}
   Let $\la$ be a positive NC measure with GNS row isometry  $\Pi_{\la}$ living on  ${\mathbb{H}}_{d}^{2} (\lambda)$. Suppose $\Pi$ is a row isometry whose components live on the Hilbert space $\mathcal{H}$. A closed operator, $X: \mathrm{Dom}(X) \longrightarrow \mathcal{H}$, with dense domain in ${\mathbb{H}}_{d}^{2} (\lambda)$ is called an intertwiner if $\mathrm{Dom}(X)$ is $\Pi_{\la}$-invariant and 
   $$X\Pi_{\la,k} x= \Pi_{k}Xx; \quad x \in \mathrm{Dom}(X)$$
\end{defn}
 \begin{prop} \label{composition-of-closed-op} Suppose that $\mathcal{H}_i$, $i=1,2,3$, are Hilbert spaces.
 \begin{enumerate}[(i)]
     \item  Let $T:\mathcal{H}_1 \longrightarrow \mathcal{H}_2 $ be a bonded operator  and $S: \mathrm{Dom}(S) \longrightarrow \mathcal{H}_3$ be a closed, unbounded operator with dense domain $\mathrm{Dom}(S)$ in $\mathcal{H}_2$. Then $ST$ is a closed operator. 
     
     \item  Let $S: \mathrm{Dom}(S) \longrightarrow \mathcal{H}_2$ be a closed, unbounded operator with dense domain $\mathrm{Dom}(S)$ in $\mathcal{H}_1$, and $U:\mathcal{H}_2\longrightarrow \mathcal{H}_3$ be a unitary operator. Then $US$ is a closed operator.   
     \item Let $S: \mathrm{Dom}(S) \longrightarrow \mathcal{H}_2$ be a closed, unbounded operator with dense domain $\mathrm{Dom}(S)$ in $\mathcal{H}_1$, and $B:\mathrm{Dom}(B)\longrightarrow \mathcal{H}_3$ be a closed operator, where $\mathrm{Dom}(B)$  is a dense linear subspace inside $\mathcal{H}_2$. Assume  also $B$ is bounded from below. Then $BS$ is a closed operator on its domain.
 \end{enumerate}
 \end{prop}
 \bp
\begin{enumerate}[(i)]
\item  Let $x_n \longrightarrow x$ and $STx_n \longrightarrow y$ both in norm. Since $T$ is bounded, we have $Tx_n \longrightarrow Tx$. Now we have that $Tx_n \longrightarrow Tx$ and  $S(Tx_n) \longrightarrow y$.
However, $S$ is closed and by assumption, so by the definition of a closed operator we must have that $Tx \in \mathrm{Dom}(S)$ and $y=S(Tx)$. This  proves that  $ST$ is a closed operator.

\item Let $x_n \longrightarrow x$ and $US(x_n) \longrightarrow y$ both in norm. Since $U$ is unitary, we have $U^{*}(USx_n )\longrightarrow U^{*}y$ in norm, which means $Sx_n \longrightarrow U^{*}y $. However, $x_n \longrightarrow x$ by assumption and $S$ is closed, so $x\in \mathrm{Dom}(S)$ and $U^{*}y=Sx $. Since $U$ is unitary, we must have $y=USx$. Therefore, $US$ is a closed operator.

\item Let $x_n \longrightarrow x$ and $BS(x_n) \longrightarrow w$ both in norm. Since $B$ is bounded below, there is $\delta >0$ such that 
$$||B(Sx_n)-B(Sx_m)|| \geq \delta ||Sx_n-Sx_m||. $$
By the above inequality and Cauchy property of $BS(x_n)$, we see that $S(x_n)$ must also be Cauchy. Thus, $S(x_n) \longrightarrow y$ for some $y$. Nevertheless, $S$ is closed, $x_n \longrightarrow x$ and $S(x_n) \longrightarrow y$. Therefore, $x \in \mathrm{Dom}(S)$ and $y=Sx$.
Now,  $S(x_n) \longrightarrow Sx$ , and $BS(x_n) \longrightarrow w$. However, $B$ is closed, so  $Sx \in \mathrm{Dom}(B) $, and  $w=B(Sx)$; showing that $BS$ is closed.

 \end{enumerate}
 \ep
 
 \begin{prop} \label{a-special-intertwiner}
   Let $\mu$ and $\la$ be two positive NC measures with row isometries $\Pi_{\mu}$ and $\Pi_{\la}$ on ${\mathbb{H}}_{d}^{2} (\mu)$ and ${\mathbb{H}}_{d}^{2} (\lambda)$ respectively such that $N_\la \subseteq N_\mu$ . Then,    $\mu \ll_{R} \la$ if and only if the (possibly) unbounded operator
   \begin{eqnarray*}
   E_{\mu, \la}:  \mathbb{A}_d +N_{\la} \subseteq {\mathbb{H}}_{d}^{2} (\lambda) &\longrightarrow& {\mathbb{H}}_{d}^{2} (\mu) \\
      p+N_{\la}&\longrightarrow& p+N_{\mu},
   \end{eqnarray*}
   is closed and densely defined. In this situation, $ E_{\mu, \la}$ is a closed intertwiner between $\Pi_{\la}$ and $\Pi_{\mu}$; i.e., $ E_{\mu, \la} \Pi_{\la}=\Pi_{\mu} E_{\mu, \la}$.
 \end{prop}
 \bp
  Remember that $ \mathcal{D}_{\la}=\mathbb{A}_d +N_{\la}$ is dense in ${\mathbb{H}}_{d}^{2} (\lambda)$ and $\Pi_\la$ -invariant since $\Pi_\la$ acts by left multiplication on $\mathbb{A}_d +N_{\la}$. Since the Cauchy transform is a unitary operator, the linear space $\mathcal{D}_{\mu}^{'}=\mathscr{C}_{\mu} (\mathbb{A}_d +N_{\la})$ is dense in the RKHS $\mathscr{H}^{+}(H_{\mu})$. Let $\mu \ll_{R} \la$; hence, $e_{\la, \mu}:\mathrm{int}(\mu,\la) \longrightarrow \mathscr{H}^{+}(H_{\la})$ is densely defined. Since  $e_{\la, \mu}$ is left multiplication by the identity, it is closed; thus, its adjoint is well-defined. 
  Now,  $E_{\mu, \la} = \mathscr{C}_{\mu}^{*} e_{\la, \mu}^{*}\mathscr{C}_{\la}$ is well-defined, and
  \begin{eqnarray*}
      E_{\mu, \la}(L^{\alpha}+N_{\la})&=&\mathscr{C}_{\mu}^{*} e_{\la, \mu}^{*}\mathscr{C}_{\la} (L^{\alpha}+N_{\la})\\
      &=&\mathscr{C}_{\mu}^{*} e_{\la, \mu}^{*} (\prescript{\la}{}{K}_{{\alpha}^{t}})\\
      &=&\mathscr{C}_{\mu}^{*}  (\prescript{\mu}{}{K}_{{\alpha}^{t}}), \quad \text{since}~ \mu \ll_{R} \la\\
      &=& L^{\alpha}+N_{\mu}.
  \end{eqnarray*}
  So, by linearity, $E_{\mu, \la}(p+N_{\la})=p+N_{\mu}$. Nevertheless, Cauchy transforms are uintary operators and  $E_{\mu, \la} = \mathscr{C}_{\mu}^{*} e_{\la, \mu}^{*}\mathscr{C}_{\la}$ with  $e_{\la, \mu}^{*}$ being closed. Hence, by Proposition \ref{composition-of-closed-op}, $E_{\mu, \la}$ is a closed operator.

  Conversely, let $E_{\mu, \la}$ be closed and densely defined operator. By the formula   $e_{\la, \mu} = \mathscr{C}_{\la} E_{ \mu, \la}^{*}\mathscr{C}_{\mu}^{*}$ and Proposition \ref{composition-of-closed-op}, we see that $e_{\la, \mu}$ is densely defined closed operator. In particular, $\mathrm{int}(\mu,\la)$ must be dense in $\mathscr{H}^{+}(H_{\mu})$, i.e., $\mu \ll_{R} \la$.
  
  Now, by the way the row isometries $\Pi_{\la}$ and $\Pi_{\mu}$ act on $\mathbb{A}_d +N_{\la}$ and $\mathbb{A}_d +N_{\mu}$, and the definition of the mapping $E_{\mu, \la}$, we see that $E_{\mu, \la}\Pi_{\la}x=\Pi_{\mu} E_{\mu, \la}x $ for any $x\in \mathrm{Dom}(E_{\mu, \la})=\mathcal{D}_{\la}$; i.e., $X=E_{\mu, \la}$  is an unbounded closed intertwiner between $\Pi_{\la}$ and $\Pi_{\mu}$. 
\ep
\begin{remark} \label{bounded-intertwiner-bounded-above-measure}
    If in  Proposition \ref{a-special-intertwiner} instead of $\mu \ll_{R} \la$ we assume that $\mu \leq t^2 \la$ for some positive $t$, then $e_{\la, \mu}$ is bounded and have norm at most $t$, see \cite[Lemma 4.2]{JM-ncld}. Hence, $E_{\mu, \la}=\mathscr{C}_{\mu}^{*} e_{\la, \mu}^{*}\mathscr{C}_{\la}$ is a bounded intertwiner with norm at most $t$. 
\end{remark}

\begin{remark} 
    Let $\mu$ and $\la$ be two positive NC measures on $\mathcal{A}_d$, and $\mu \ll_{F} \la$. Therefore, $q_{\mu}^{\la}$ with form domain $\mathbb{A}_d +N_\la$ is closeable and densely defined in $\mathbb{H}^{2}_{d}(\la)$. So, by Proposition \ref{transfer-of-the-domain} there is a closed, densely defined, and positive operator $D$ such that $\mathbb{A}_d +N_\la \subseteq \mathrm{Dom}(\sqrt{D})$ is a core for $\sqrt{D}$, and 
$$q_{\mu}^{\la}(a+N_\la, b+N_\la)=\langle \sqrt{D} (a+N_\la), \sqrt{D}(b+N_\la) \rangle_{\la}.$$
 Furthermore, recall that $q_{\mu}^{\la}$ is $\Pi_\la$-Toeplitz by Lemma \ref{hybrid-form-Toeplitz}.
\end{remark}

\begin{lemma}\label{a-row-ism-on-Ran-sqrtD}
    Let $\mu$ and $\la$ be two positive NC measures on $\mathcal{A}_d$, $\mu \ll_{F} \la$, and the operator $D$ be as in the above. Define    
\begin{eqnarray*}
    \Pi_k : \mathrm{Ran}\sqrt{D} \longrightarrow \mathbb{H}^{2}_{d}(\la) \\
    \Pi_k \sqrt{D}x=\sqrt{D} \Pi_{\la,k}x
\end{eqnarray*}
    Then, $\Pi=(\Pi_1, \cdots, \Pi_d)$ is a row isometry on $\overline{\mathrm{Ran}(\sqrt{D})}^{\la}$.
\end{lemma}
\bp
Recall that in the proof of Proposition \ref{Pi-Toeplitz-ops-GNS-la} we showed that this is well-defined since $\mathrm{Dom}(\sqrt{D})$ is $\Pi_\la$-invariant. We just need to show that $\Pi_{j}^{*}\Pi_k=\delta_{j,k}I$. To this end note that 
\begin{eqnarray*}
    \langle \Pi_j \sqrt{D}(a+N_\la), \Pi_k \sqrt{D}(b+N_\la) \rangle_{\la} 
    &=& \langle  \sqrt{D} \Pi_{\la,j}(a+N_\la),  \sqrt{D}\Pi_{\la,k}(b+N_\la) \rangle_{\la}\\
    &=& q_{\mu}^{\la}(\Pi_{\la, j}(a+N_\la), \Pi_{\la, k}(b+N_\la))\\
    &=& \delta_{j,k} q_{\mu}^{\la}(a+N_\la, b+N_\la)  \\
    &=& \delta_{j,k} \langle \sqrt{D} (a+N_\la), \sqrt{D}(b+N_\la) \rangle_{\la}
\end{eqnarray*}
\ep

Let $T$ be an unbounded operator on a Hilbert space $\mathscr{H}$, and $\mathscr{M}$ be a von Neumann algebra inside $\mathscr{B}(\mathscr{H})$. We say that $T$ is affiliated with $\mathscr{M}$, and we write $T \sim \mathscr{M} $ if for any $B \in \mathscr{M}^{\prime}$, we have $TB \supseteq BT$; i.e., $\mathrm{Dom}(T)$ is $B$-invariant, and $TBh=BTh$ for any $h \in \mathrm{Dom}(T)$. What is important here is that $T$ commutes with elements of $\mathscr{M}^{\prime}$ over the domain of $T$. When, $T$ is bounded, the condition $T \sim \mathscr{M} $ means $T \in \mathscr{M}$; however, we need a bit more general condition for possibly an unbounded operator $T$.
\begin{cor} \label{affiliation}
    Assuming the notations of Lemma \ref{a-row-ism-on-Ran-sqrtD}, and assuming that $\la$ is Cuntz, then we have that $D \sim vN(\Pi_{\la})^{\prime}$; i.e., $D$ commutes with $\Pi_\la$ on the domain of $D$.
\end{cor}
\bp
By Lemma \ref{a-row-ism-on-Ran-sqrtD}, we see that $\sqrt{D}$ is a closed operator that intertwines the row isometries $\Pi$ and $\Pi_{\la}$. Since $\Pi_\la$ is Cuntz, by  \cite[Lemma 8.9]{JM-ncld} the result follows for $D=(\sqrt{D})^{*} \sqrt{D}$. 
\ep

\begin{thm} \label{R-AC-Cuntz=F-Ac-Cuntz}
    Let $\mu$ and $\la$ be two positive NC measures on $\mathcal{A}_d$, and let $\la$ be Cuntz. Then,  $\mu \ll_{R} \la$ if and only if $\mu \ll_{F} \la$.
\end{thm}
\bp 
(i) Suppose that $\mu \ll_{R} \la$; we must show that $q_{\mu}=q_{\mu}^{\la}$ is well-defined and closable in ${\mathbb{H}}_{d}^{2}(\lambda)$. Since $\mathrm{int}(\mu, \la)$ is dense in ${\mathscr{H}}^{+}(H_{\mu})$, we can define the densely defined and closed operator
$e: \mathrm{int}(\mu, \la) \subseteq  {\mathscr{H}}^{+}(H_{\mu})\hookrightarrow {\mathscr{H}}^{+}(H_{\la}) $, which is defined by identity ( or left multiplication by the identity). So, $ee^*$ is densely defined, closed, and positive operator on ${\mathscr{H}}^{+}(H_{\la})$.  The operator $E=\mathscr{C}_{\mu}^{*} e^* \mathscr{C}_{\la}$ is densely defined and closed. To see why it is closed note that by \cite{JM-freeCE} we have $\mathscr{C}_{\la} (L^{{\omega}^{t}}+N_{\lambda})=\prescript{\la}{}{K_{\omega}}$,  where $t$ means the transpose of words, and $\prescript{\la}{}{K_{\omega}}$ denotes the coefficient kernel vectors, which are dense in ${\mathscr{H}}^{+}(H_{\la})$ \cite{JM-freeCE}, and hence they give rise to a core for $e^*$. Also, note that
\begin{eqnarray*}
    E: \mathscr{C}_{\la}^{*} Dom(e^*) \subseteq {\mathbb{H}}^{2}_{d}(\la) &\longrightarrow& {\mathbb{H}}^{2}_{d}(\mu)\\
    a+N_{\la} &\longrightarrow& a+N_{\mu}
\end{eqnarray*}
Now, consider the closed form $q_{E^*E}$ with form domain $\mathrm{Dom}(q_{E^*E})=\mathrm{Dom}(E)\subseteq {\mathbb{H}}^{2}_{d}(\la)$. Then,
\begin{eqnarray*}
  q_{E^*E} (a+N_{\la}, b+N_{\la})&=&\langle  a+N_{\la}, E^*E( b+N_{\la} )\rangle_{\la}\\
  &=&\langle E( a+N_{\la}), E(b+N_{\la} )\rangle_{\mu}\\
  &=&\langle  a+N_{\mu}, b+N_{\mu} \rangle_{\mu}
\end{eqnarray*}
Obviously, the above equation implies that $N_\mu \supseteq N_\la$, so $q_{\mu}^{\la}(a+N_{\la}, b+N_{\la})=\langle  a+N_{\mu}, b+N_{\mu}\rangle_{\mu}$ is well-defined. Thus, we see that $q_{\mu}^{\la}=q_{E^*E}$ on $\mathrm{Dom}(q_{\mu}^{\la})=\mathbb{A}_{d}+N_\la \subseteq \mathrm{Dom}(E)=\mathrm{Dom}(q_{E^*E})$, and $q_{E^*E}$ is closed in ${\mathbb{H}}_{d}^{2}(\lambda)$. Hence, $q_{\mu}^{\la}$ is closable in ${\mathbb{H}}_{d}^{2}(\lambda)$.
   
   Conversely, Suppose $\mu \ll_{F} \la$ holds, so by definition $N_{\mu} \supseteq N_{\la}$, and hence  $q_{\mu}=q_{\mu}^{\la}$ is a well-defined and  closable form on $\mathbb{H}_{d}^{2}(\lambda)$. Hence, we can find a positive closed unbounded operator $D$ on $\mathbb{H}_{d}^{2}(\lambda)$ such that 
   \begin{equation*}
       q_{\mu}(a+N_{\lambda}, b+N_{\lambda})=\langle \sqrt{D}(a+N_{\lambda}), \sqrt{D}(a+N_{\lambda}) \rangle.
   \end{equation*}
   By Corollary \ref{affiliation}, $D \sim vN(\Pi_{\lambda})^{'}$ so that $D$ commutes with $\Pi_{\lambda}$ on the domain of $D$. Let $\psi_n$ be the characteristic function of $[0,n] \cap \sigma(D)$. By the Borel functional calculus discussion, the positive closed operators $D_n=\psi_n(D) D$ belong to $vN(\Pi_{\lambda})^{'}$ and the sequence $D_n$ converges in strong operator topology (SOT) to $D$. Since $D_n$ commutes with $\Pi_{\lambda}$, the positive quadratic form
\begin{equation*}
    q_n (a+N_{\lambda}, b+N_{\lambda})=\langle \sqrt{D_n}(a+N_{\lambda}), \sqrt{D_n}(a+N_{\lambda}) \rangle
\end{equation*}
   stems from an  NC measure $\mu_n$. Also, the sequence $(q_n)$  increases to $q_{\mu}$. Therefor, for any positive element $a \in \mathcal{A}_d$ the sequence $\mu_n(a)$ increase to  $\mu(a)$. Furthermore, $D_n \leq n I_{\mathbb{H}_{d}^{2}(\lambda)}$; hence, $\mu_{n} \leq n \lambda$. By  \ref{mu-kernel}, we see that $\prescript{\mu_{n}}{}{K} \leq \prescript{\mu}{}{K}$; hence, by the NC Aronszajn's theorem \cite[Theorem 5.1]{BMV},  $\mathcal{H}_{nc}(\prescript{\mu_{n}}{}{K}) \subseteq \mathcal{H}_{nc}(\prescript{\mu}{}{K})$. From $\mu_n \uparrow \mu$, we have
\begin{equation*} 
    \bigvee \mathscr{H}^{+}(H_{\mu_{n}}) =\mathscr{H}^{+}(H_{\mu}).
\end{equation*}

\noindent Similarly, for $\mu_{n} \leq n \lambda$ we see that  $\mathcal{H}_{nc} (\prescript{\mu_{n}}{}{K}) \subseteq \mathcal{H}_{nc} (\prescript{\lambda}{}{K})$. Thus, the intersection space 
$$ \mathrm{int}(\mu , \lambda)={\mathscr{H}}^{+}(H_{\mu}) \cap {\mathscr{H}}^{+}(H_{\lambda}) $$ 
\noindent is non-empty and dense in ${\mathscr{H}}^{+}(H_{\mu})$. Now, the definition \ref{RK-AC-Sing} applies.
\ep

In section 3, when we introduced the notion of RK-AC, we showed how  $\mu \ll \la$ led to $\mu \ll_{R} \la$. The converse is also true, and  we explain it later. For the Cuntz case, we give the proof here. For an altenative and somehow similar proof see \cite[Theorem 6]{BMN}.

\begin{thm} \label{Cuntz-RK-AC-implies-AC}
  Let $\mu$ and $\la$ be two positive classical measures on $\mathbb{T}$, and $\la$ be Cuntz. If  ~$\mu \ll_{R} \la$, then $\mu \ll \la$ in the classical setting.
\end{thm}
{\bf Proof.} Note that in this classical setting the operator disk system $\mathcal{A}_1$ is just the C*-algebra $C(\mathbb{T})$. Thus, our measures live on a C*-algebra, so we can use the theory of Gheaondea-Kavruk in \cite{GK-ncld} for some parts of the proof. When $\la$ is Cuntz: By using Proposition \ref{a-special-intertwiner} to the classical setting, so we see that $X=E_{\mu, \la}$ is a closed intertwiner between $\Pi_{\la}$ and $\Pi_{\mu}$; i.e., $E_{\mu,\la}\Pi_{\la}=\Pi_{\mu}E_{\mu,\la}$, and  $X$ sends polynomials to polynomials. Obviously, $D=X^*X=\mathscr{C}_{\la}^{*}e_{\la, \mu}e_{\la, \mu}^{*}\mathscr{C}_{\la}$ is positive. Since $\la$ is Cuntz, we see that $H^2(\la)=L^2(\la)$, see \cite[Proposition 5.11]{JMT-NCFM} or use Szego theorem \cite[p. 49]{Hoff}, so both $D$ and $M_z$ live on $H^2(\la)=L^2(\la)$. We claim that $D$ is $\Pi_{\la}$-Toeplitz. Consider the quadratic form 
$${\bf q}_{D}(p,q)=\langle \sqrt{D}p, \sqrt{D}q \rangle_{\la} =\int_{\mathbb{T}} \overline{\sqrt{D}p(z)}\sqrt{D}q(z) d\la(z), $$
where $p$ and $q$ are two polynomials. Recall that polynomials are a core for the unbounded operator $D$, and also that $\Pi_{\la}$ acts as $M_z$ on polynomials. By linearity, we only need to work with $p(z)=z^n$ and $q(z)=z^m$. Furthermore, we use the sloppy language that $p=z^n$ and $M_zp=z^{n+1}$ instead of using function notation. Now, 
\begin{eqnarray*}
    {\bf q}_{D}(\Pi_\la p, \Pi_\la q)&=& \langle \sqrt{D} \Pi_\la p, \sqrt{D} \Pi_\la q \rangle_{\la}\\
    &=& \langle \sqrt{D} M_z p, \sqrt{D} M_z q \rangle_{\la}\\
    &=& \langle \sqrt{D} z^{n+1}, \sqrt{D} z^{m+1} \rangle_{\la}\\
    &=& \langle \sqrt{D} z^{n+1}, \sqrt{D} z^{m+1} \rangle_{\la}, \quad \text{by polar decomposition}~ \sqrt{D}=Ue^{*}\mathscr{C}_\la\\
    &=& \langle e^{*}\mathscr{C}_\la  z^{n+1}, e^{*}\mathscr{C}_\la  z^{m+1} \rangle_{H^{+}(H_\mu)}, \quad U ~\text{is unitary}\\
     &=& \langle e^{*} K_{n+1}^{\la}, e^{*} K_{m+1}^{\la} \rangle_{H^{+}(H_\mu)}, \quad \text{effect of Cauchy transform} \\
     &=& \langle  K_{n+1}^{\mu}, K_{m+1}^{\mu} \rangle_{H^{+}(H_\mu)}, \quad \text{effect of}~ e^*~ \text{and}~\mu \ll_R \la \\
     &=& \langle z^{n+1}, z^{m+1} \rangle_{H^{2}(\mu)}, \quad \text{using Cauchy transform} \\
     &=&\mu ( z^{-n-1}z^{m+1})\\
     &=&\mu( z^{-n}z^{m} )\\
    &=& \langle z^{n}, z^{m} \rangle_{H^{2}(\mu)}\\
     &=& \int \overline{p}q d\mu
\end{eqnarray*}
On the other hand, 
\begin{eqnarray*}
    {\bf q}_{D}( p,  q)&=& \langle \sqrt{D}  p, \sqrt{D}  q \rangle_{\la}\\
    &=& \langle \sqrt{D} z^{n}, \sqrt{D} z^{m} \rangle_{\la}, \quad \text{by polar decomposition}~ \sqrt{D}=Ue^{*}\mathscr{C}_\la\\
    &=& \langle e^{*}\mathscr{C}_\la  z^{n}, e^{*}\mathscr{C}_\la  z^{m} \rangle_{H^{+}(H_\mu)}, \quad U ~\text{is unitary}\\
     &=& \langle e^{*} K_{n}^{\la}, e^{*} K_{m}^{\la} \rangle_{H^{+}(H_\mu)}, \quad \text{effect of Cauchy transform} \\
     &=& \langle  K_{n}^{\mu}, K_{m}^{\mu} \rangle_{H^{+}(H_\mu)}, \quad \text{effect of}~ e^*~ \text{and}~\mu \ll_R \la \\
     &=& \langle z^{n}, z^{m} \rangle_{H^{2}(\mu)}, \quad \text{using Cauchy transform}\\
     &=& \int \overline{p}q d\mu
\end{eqnarray*}
The above two set formulas prove that 
$$ {\bf q}_{D}(\Pi_\la p, \Pi_\la q)= {\bf q}_{D}( p,  q),$$
or equivalently,
$$\langle \sqrt{D} M_Z (p) , \sqrt{D} M_Z (q) \rangle_{H^{2}(\la)}=\langle \sqrt{D} p , \sqrt{D} q \rangle_{H^{2}(\la)},  $$
that is $D$ is $\Pi_\la$-Toeplitz. Also,
\begin{equation} \label{equality-of-integrals}
    \int \overline{p}q d\mu =\int_{\mathbb{T}} \overline{\sqrt{D}p(z)}\sqrt{D}q(z) d\la(z).
\end{equation}
On the other hand, 
\begin{eqnarray*}
    \langle M_Z \sqrt{D} p , M_Z \sqrt{D} q \rangle_{H^{2}(\la)} &=& \int_{\mathbb{T}} \overline{z\sqrt{D}p(z)} z\sqrt{D}q(z)   ~d\la(z) \\
    &=& \int_{\mathbb{T}} \overline{\sqrt{D}p(z)} \bar{z}z\sqrt{D}q(z)   ~d\la(z)\\
    &=&\int_{\mathbb{T}} \overline{\sqrt{D}p(z)} \sqrt{D}q(z)   ~d\la(z) ,\quad \text{since}~  \bar{z}z=|z|^2=1\\
     &=&\langle \sqrt{D} p , \sqrt{D} q \rangle_{H^{2}(\la)}.
\end{eqnarray*}
Thus,
$$\sqrt{D}M_Z=M_Z\sqrt{D}, \quad \text{on} ~ \mathscr{P}_0; $$
which means that $\sqrt{D} $ is affiliated with  the commutant of the range of the following representation of $L^{\infty}(\la)$:
\begin{eqnarray*}
    \phi: L^{\infty}(\la) &\longrightarrow &\mathscr{B}(L^2(\la)) \\
    \phi(f)&=&M_f.
\end{eqnarray*}
 However, $\left(\phi(L^{\infty}(\la)\right)'\cong L^{\infty}(\la)$ and all affiliated unbounded operators with $L^{\infty}(\la)$ arise from multiplication by $L^{2}(\la)$ functions. So, $\sqrt{D}=M_{g}$ for some positive function $g \in L^{2}(\la)$, and consequently $D=M_f$ where $f=g^2 \in L^1(\la)$. Thus, from Equation \ref{equality-of-integrals} we see that  $d\mu=f d\la$, and consequently $\mu \ll \la$.
$\blacksquare$

When $\la$ is Cuntz, we should not expect that  $\mathrm{Int}(\mu, \la)$ be invariant or has a non-epmty invariant subspace as the following example shows.

\begin{eg} \label{non-example}
Let $m$ be the Lebesgue measure on the unit circle $\mathbb{T}$, $\Pi=M^{L}_{Z}$ be the operator of multiplication, shift, on $\mathbb{H}^{2}(m)=\mathbb{H}^{2}$ and $V$ be the unitarily equivalent operator, via Cauchy transform, on $\mathscr{H}^{+}(H_{m})=\mathbb{H}^{2}$. Suppose that $m_1$  equals the Lebesgue measures on the upper half of $\mathbb{T}$  and equals zero on the lower half of $\mathbb{T}$. Similarly, let $m_2$  equal the Lebesgue measures on the lower half of $\mathbb{T}$  and equals zero on the upper half of $\mathbb{T}$. It is well-known that these measures are Cuntz, see for example \cite[Example 3.20]{JMT-NCFM} and \cite[Example 2.2]{MK-rowiso}, but we will give a different proof based on our results here. Assume that $\Pi_1$ and $\Pi_2$ denote the operator of multiplication by the independent variable on $\mathbb{H}^{2}({m_1})$ and $\mathbb{H}^{2}({m_2})$. Also, let $V_1$ and $V_2$ be the corresponding operators, via Cauchy transform, on the $\mathscr{H}^{+}(H_{m_1})$ and $\mathscr{H}^{+}(H_{m_2})$. For the analytic polynomial $f(z)=1+z$, we see that the function
\begin{equation}
f_{1}(e^{i\theta}) =
    \begin{cases}
        0, & \text{if } 0 \leq \theta < \pi \\
        1+e^{i \theta}, & \text{if }  \pi \leq \theta < 2 \pi
    \end{cases}
\end{equation}
belongs to $\mathscr{H}^{+}(H_{m_1})$. Now, $\int_{\mathbb{T}} |f_1(e^{i\theta})|^{2} dm_1(\theta)=0$, so $f_1 \in N_{m_1}$. However, $f_1 \neq 0$; hence, by Theorem \ref{left-ideal-non-Cuntz=0}, $m_1$ cannot be a non-Cuntz measure; i.e., $m_1$ is a Cuntz measure. Similarly, $m_2$ is a Cuntz measure.  Since $m_1$ and $m_2$ are absolutely continuous with respect to $m$, by \cite[Definition 6.8]{JM-ncld} $m_1$ and $m_2$ are Cuntz type L measures. Obviously, $m=m_1+m_2$; hence, $m_1 \leq m$, so by Aronszajn's theorem \cite[Theorem 5.1]{Paulsen-rkhs},  $\mathscr{H}^{+}(H_{m_1}) \subseteq \mathscr{H}^{+}(H_{m})$; thus, $\mathrm{int}(m, m_1)=\mathscr{H}^{+}(H_{m_1})$.   Also, note that $\mathrm{Int}(m, m_1) \neq \mathscr{H}^{+}(H_{m})$; otherwise, by Definition \ref{RK-AC-Sing} $m\ll_{R} m_1$. However, in Theorem \ref{Cuntz-RK-AC-implies-AC}, we showed that RK-AC implies classical-AC; thus, $m$ is absolutely continuous with respect to $m_1$ in the classical sense.  Nevertheless, this is false; or see for example \cite[Example 2.2]{MK-rowiso}. Now, we claim that $\mathrm{Int}(m, m_1)$ cannot be $V_\mu=V_m \cong \Pi=M_{Z}^{L}$ invariant. By \cite[Theorem 4.5]{JM-ncld}, $\mathrm{Int}(m, m_1)$ is always $V_m$-co-invariant. So, if it were $V_m$-invariant as well, then it means that the shift operator $V_m$ would have a non-trivial reducing  subspace, which is impossible by \cite[Theorem 3.5.5]{Murphy}.
\end{eg}

\begin{remark}
    Interestingly, the previous example reveals three phenomena. Note that the two measures $m_1$ and $m_2$ are Cuntz type L or absolutely continuous Cuntz (ACC) measures bounded above by the type L measure $m$ which  is the Lebesgue measure. Firstly, the previous example shows that the sum of two 
    two ACC measures is not ACC again. Thus, the set of ACC measures is not a cone. The second observation is that though $m_i \leq m$ and $m$ is type L, the measures $m_i$ are not type L. Hence, the set of type L measures are not hereditary. Thirdly, absolutely continuity of a measure with respect to another one is not a symmetric relation; this is unlike singularity. To be more specific, obviously $dm_1 = f dm$ for a suitable characteristic function corresponding to the upper half circle (or half interval); thus, $m_1 \ll m$. However, as we discussed in the previous example,  $m \not \ll m_1$.  
\end{remark}

\begin{eg} \label{a-reducing-cuntz-eg}
    Let's give an example for splitting Cuntz measures in which the intersection space is $V_\mu$-reducing. This example is interesting since we know for non-Cuntz measures the intersection space is always $V_\mu$-reducing, but as this example illustrates it might happen for Cuntz measures despite the fact that their RKHS spaces do not contain constant functions.
    
    Suppose that the points on the unit circle are represented by their polar angles, so we can look at the unit circle as the interval $[0, 2\pi)$ on the real line. Consider the Dirac measures $\mu=\delta_{\{0,\pi \}}$ and $\la=\delta_{0}$. Obviously, $\delta_{\{0,\pi\}}=\delta_{0}+\delta_{\pi}$ defines a Lebesgue decomposition of $\mu$ with respect to $\la$ with $\mu_{ac}=\delta_0=\la$ and $\mu_s=\delta_{\pi}$. Since both measures are singular, we deduce that both $\Pi_\mu$ and $\Pi_\la$ are unitary operators, which act like bilateral shift, or invertible multiplication by $z$, on their domains. Consequently, $V_\mu=\mathscr{C}_{\mu}\Pi_\mu\mathscr{C}^{*}_{\mu}$ and $V_\la=\mathscr{C}_{\la}\Pi_\la\mathscr{C}^{*}_{\la}$ are unitaries by their definitions.  Let $\prescript{\mu}{}{K}_{n}$ be the coefficient  evaluation vector in the RKHS of $\mu$; i.e.,
    $$\langle \prescript{\mu}{}{K}_{n} , f \rangle=\hat{f}(n); $$
    also
    $$\langle \prescript{\la}{}{K}_{n} , f \rangle=\hat{f}(n); $$
    where $\hat{f}(n)$ is the Taylor's $n$-th coefficient of $f$. Note that
    $$\prescript{\mu}{}{K}_{n}(Z)=\sum_{m=0}^{\infty} \prescript{\mu}{}{K}_{n,m}Z^m,$$
    where
    $$\prescript{\mu}{}{K}_{n,m}=\langle Z^n, Z^m\rangle_\mu = 1+e^{i(m-n)\pi}.$$
    Now, since $\Pi_\mu$ and $\Pi_\la$ are Cuntz unitaries, they act act through left multiplication by $z$, and we see that
    \begin{eqnarray*}
V_\mu(\prescript{\mu}{}{K}_{n})&=&\mathscr{C}_{\mu}\Pi_\mu\mathscr{C}^{*}_{\mu}(\prescript{\mu}{}{K}_{n})\\
&=& \mathscr{C}_{\mu}M_Z(z^n+N_{\mu})\\
&=& \mathscr{C}_{\mu}(z^{n+1}+N_\mu)\\
&=&\prescript{\mu}{}{K}_{n+1},
    \end{eqnarray*}
    and similarly
    $$V_\la(\prescript{\la}{}{K}_{n})=\prescript{\la}{}{K}_{n+1}$$
However, some calculations reveal that
\begin{eqnarray*}
   \prescript{\mu}{}{K}_{n}(z) &=&  \sum_{k=-\lfloor 
\frac{n}{2} \rfloor}^{\infty} 2 z^{n+2k}\\
&=& 2 z^{n} \sum_{k=0}^{\infty}  z^{2k}  + 2 z^{n}\sum_{k=-\lfloor 
\frac{n}{2} \rfloor}^{-1}  z^{2k} \\
&=& \frac{2z^n}{1-z^2}+ 2 z^{n}\sum_{1}^{k=\lfloor 
\frac{n}{2} \rfloor}  (z^{-2})^{k} \\
&=& \frac{2z^n}{1-z^2} + 2 z^{n}\frac{1-z^{-2\lfloor 
\frac{n}{2} \rfloor}}{z^2-1}\\
&=&\frac{2z^{n-2\lfloor 
\frac{n}{2} \rfloor}}{1-z^2}.
\end{eqnarray*}
Thus,
$$\prescript{\mu}{}{K}_{n}(z)=\begin{cases}
        \frac{2}{1-z^2}, & \text{if } n ~\text{is even}\\
        \frac{2z}{1-z^2}, & \text{if }  n ~\text{is odd}
    \end{cases},$$
and also $\prescript{\la}{}{K}_{n,m}=1$, so
$$ \prescript{\la}{}{K}_{n}(z)=\frac{1}{1-z^2}.$$
Hence,
$$\begin{cases}
        V_{\mu}(\prescript{\mu}{}{K}_{n})=M_z (\prescript{\mu}{}{K}_{n}), & \text{if } n ~\text{is even}\\
        V_{\mu}(\prescript{\mu}{}{K}_{n})=M_{z}^{*} (\prescript{\mu}{}{K}_{n}), & \text{if }  n ~\text{is odd}
    \end{cases},$$
while
$$ V_\la (\prescript{\mu}{}{K}_{n})=I(\prescript{\mu}{}{K}_{n})$$ 
Furthermore, recall that the coefficient kernel vectors are dense in the corresponding RKHS spaces so, the above formulas reveal some properties of $V_\mu$ and $V_\la$. In particular, $V_\la=I$, but the situation for $V_\mu$ is abit vague as it changes behavior from multiplication by $z$ to division by $z$. However, by \cite[Theorem 4.5]{JM-ncld}, we see that $V_{\mu}^*=V_{\la}^{*}$ on the intersection space. So,  we can say for sure that $V_\mu$ acts like identity on the intersection space, and hence the intersection space is $V_\mu$-reducing.

Also, note that the Herglotz function of $\delta_0$ is $H_{\delta_{0}}=\frac{1+z}{1-z}$, and the RKHS $\mathscr{H}^{+}(H_{\delta_{0}})$ has the kernel $K(z,w)=\frac{1}{(1-z)(1-\overline{w})}$; hence, $\mathscr{H}^{+}(H_{\delta_{0}})$ is the one dimensional space generated by $\frac{1}{(1-z)}$. While, for $\delta_\pi$, the Herglotz function is $H_{\delta_{\pi}}=\frac{1-z}{1+z}$, and the RKHS $\mathscr{H}^{+}(H_{\delta_{\pi}})$ has the kernel $K^{'}(z,w)=\frac{1}{(1+z)(1+\overline{w})}$; therefore, $\mathscr{H}^{+}(H_{\delta_{\pi}})$ is the one dimensional space generated by $\frac{1}{(1+z)}$.

By the Aronsjan sum of RKHS', we see that $\mathscr{H}^{+}(H_{\mu})=\mathscr{H}^{+}(H_{\delta_{0}})+\mathscr{H}^{+}(H_{\delta_{1}}) $. So,
    $$\mathrm{int}(\mu , \la)=\mathscr{H}^{+}(H_{\delta_{0}}).  $$
    Furthermore, $ \mathscr{H}^{+}(H_{\delta_{0}})\cap \mathscr{H}^{+}(H_{\delta_{1}})=0$, so by \cite[Theorem 4.4]{JM-ncld}
    $$\mathscr{H}^{+}(H_{\mu})=\mathscr{H}^{+}(H_{\delta_{0}})\oplus \mathscr{H}^{+}(H_{\delta_{1}}). ~\blacksquare$$
\end{eg}

In fact a more general version of Theorem \ref{Cuntz-RK-AC-implies-AC} is true which is proved in \cite{BMN}.
\begin{thm} \label{RK-AC-implies-AC}
  Let $\mu$ and $\la$ be two positive classical measures on $\mathbb{T}$. If  ~$\mu \ll_{R} \la$, then $\mu \ll \la$ in the classical setting.
\end{thm}

Combining this theorem with the discussion at the beginning of section 3 reveals that: 

\begin{cor} \label{RK-AC-equiv-AC}
  Let $\mu$ and $\la$ be two positive classical measures on $\mathbb{T}$. Then,  ~$\mu \ll_{R} \la$ if and only if $\mu \ll \la$ in the classical setting.
\end{cor}
Let $q$ be a bounded and positive semi-definite quadratic form on $\mathbb{H}^{2}_{d} (\mu )$ for some $\mu \in \left( \scr{A} _d   \right) ^\dag _{+}$. We will call such a form \emph{sub-Toeplitz with respect to $\mu$}. If $q$ is sub-Toeplitz with respect to $\mu$, then by the Riesz lemma, there is a bounded, positive semi-definite $D_{ac} \in \scr{B} (\mathbb{H}^{2}_{d} (\mu ))$ so that 
$$q (p_1 + N_\mu , p_2 + N_\mu ) = \ip{ p_1 + N_\mu}{D_{ac} p_2 + N_\mu } _\mu. $$ 
Note that a more exact notation for the above $q$ is $q^{\mu}$. We can define the space of all $q-$Cauchy transforms by:
\ba \scr{C} _q (p) (Z) & := & \sum _\om Z^\om q(L^\om + N_\mu , p + N_\mu )  \\
& = & \sum _\om Z^\om \ip{L^\om + N_\mu}{D_{ac} p + N_\mu}_\mu = \scr{C} _\mu (D_{ac} p ) (Z). \ea
That is, the set of all $q-$Cauchy transforms is the range of $\scr{C} _\mu D \scr{C} _\mu ^*$. 

Let $\mu , \la \in \left( \scr{A} _d   \right) ^\dag _{+}$, and $N_\la \subseteq N_\mu$. Consider the densely--defined and positive semi-definite quadratic form, $q_{\mu}^{\la}$, in $\mathbb{H}^{2}_{d} (\la )$ with form domain $ \mathbb{A}-d + N_\la, $ and form core $\langle \mathfrak{z} _1 , ..., \mathfrak{z} _d \rangle + N_\la$. Recall that $q_{\mu}^{\la}$  is $\Pi _\la-$Toeplitz in the sense that 
$$ q_{\mu}^{\la} ( \Pi _{\la ; j} p_1 + N_\la , \Pi _{\la ; k} p_2 + N_\la ) = \delta _{j,k} q_{\mu}^{\la} (p_1 + N_\la , p_2 + N_\la ).$$
Recall that by B. Simon, $q_{\mu}^{\la}$ has the Lebesgue form decompostion:
$$ q_{\mu}^{\la} = q_{ac}^{\la} + q_{s}^{\la}, $$ where $ 0 \leq q_{ac}^{\la} , q_{s}^{\la} \leq q_{\mu}^{\la}$, $q_{ac}^{\la}$ is the maximal closeable quadratic form in $\mathbb{H}^{2}_{d} (\la )$ bounded above by $q_{\mu}^{\la}$ and $q_{s}^{\la}$ is singular. These forms are given by the formulas:
\be q_{s}^{\la} (p_1 + N_{\la} , p_2 + N_\la ) = \ip{ p_1 + N _{\mu +\la}}{ Q_s p_2 + N_{\mu +\la}}_{\mu +\la}, \ee and 
\be q_{ac}^{\la} (p_1 + N_{\la} , p_2 + N_\la ) = \ip{ p_1 + N _{\mu +\la}}{ (Q_{ac} - E^* E) p_2 + N_{\mu +\la}}_{\mu +\la}, \ee 
where $E : \mathbb{H}^{2}_{d} (\mu + \la ) \hookrightarrow \mathbb{H}^{2}_{d} (\la )$ is the contractive co-embedding, $Q_s = P _{\nbker E}$ and $Q_{ac} = I - Q_s$. Note that  $q_{\mu}^{\la} = q_{ac}^{\la} + q_{s}^{\la}$ where $q_{\mu}^{\la}, q_{ac}^{\la}, q_{s}^{\la}$ have the common form core $\mathbb{A}_d + N_\la$ which is norm--dense in $\mathbb{H}^{2}_{d} (\mu )$. With these observation, we now see that:
\begin{lemma}
 One can view $q_{ac}^{\la}, q_{s}^{\la}$ as contractive sub-Toeplitz forms, e.g.,
$$q_{ac}^{\mu}(a+N_\mu, b+N_\mu):=q_{ac}^{\la}(a+N_\la, b+N_\la),$$
in $\mathbb{H}^{2}_{d} (\mu )$ so that $q_{ac}^{\la} + q_{s}^{\la} = \ip{ \cdot }{\cdot }_\mu$. In particular, since
$$ q_{ac}^{\mu} (p_1 + N_\mu , p_2 + N_\mu ) = \ip{p_1 + N_\mu}{D_{ac} p_2 + N_\mu} _\mu,$$ and similarly for $q_{s}^{\mu}$, we have that 
$0 \leq D_{ac} , D_s \leq I$ and that $D_{ac} + D_s = I$.    
\end{lemma}
\bp
By the discussion before the lemma, we are only left with the task of provig the well-defindness.  So, let $a+N_\mu$=0. Then, $0=\mu(a^*a)=q_{\mu}^{\la}(a+N_\la, a+N_\la)$. However, $q_{ac}^{\la} \leq q_{\mu}^{\la}$; thus, $q_{ac}^{\la}(a+N_\la, a+N_\la)=0$.
\ep
\begin{remark} \label{types-of-derivatives}
    Now let's mentioned that we have three understanding of $q_{ac}^{\la}$ ( the absolutely continuous part of $q_{\mu}^{\la}$)  in mind, and all coincide:
    \begin{itemize}
        \item {\bf Simon:} as $q_{ac}^{\la}$ on $\mathbb{H}^{2}_{d}(\la)$ given by Simon's formula
        $$q_{ac}^{\la}(a+N_\la, b+N_\la)=\langle a+N_{\mu+\la} ,(Q_{ac}-E^*E) b+ N_{\mu+\la} \rangle_{\mu+\la},$$
        Here, we can call the bonded operator $Q_{ac}-E^*E$ on $\mathbb{H}^{2}_{d}(\sigma)$  the Simon derivative, and we denote it by either $D_\sigma$ or $D_S$.
        \item {\bf Riesz:} as the closeable form  $q_{ac}^{\la}$ on $\mathbb{H}^{2}_{d}(\la)$ given by an unbounded, closed, operator $D$ such that $\mathbb{A}_{d}+N_\la$ is a form core  for $\sqrt{D}$ and 
        $$q_{ac}^{\la}(a+N_\la, b+N_\la)=\langle \sqrt{D} (a+N_{\la}) , \sqrt{D}( b+ N_{\la}) \rangle_{\la},$$
        Here, we can call the unbounded operator $D$ on $\mathbb{H}^{2}_{d}(\la)$ the Riesz derivative, and we denote it by either $D_\la$ or $D_R$.
        \item {\bf Toeplitz:} as the sub-Toeplitz form $q_{ac}^{\mu}$  on $\mathbb{H}^{2}_{d}(\mu)$ defined  by 
        $$q_{ac}^{\mu}(a+N_\mu, b+N_\mu):=q_{ac}^{\la}(a+N_\la, b+N_\la).$$
        In this case, since $q_{ac}^{\mu} \leq q_{\mu}^{\mu}$ we see that $q_{ac}^{\mu}$ is bounded, so it is given by 
        $$q_{ac}^{\mu}(a+N_\mu, b+N_\mu)=\langle a+N_{\mu} ,D_{ac}( b+ N_{\mu} )\rangle_{\mu}.$$
        
        Here, we can call the bonded operator $D_{ac}$ on $\mathbb{H}^{2}_{d}(\mu)$ the Toeplitz derivative, and we denote it by either $D_\mu$ or $D_T$.        
    \end{itemize}
    
        In the sequel, we refrain from using the superscripts $\la$ and $\mu$, and we write a unified symbol $q_{ac}$ since the superscripts can be read by paying attention to the domain of the forms. Also, we implicitly assume that $N_\la \subseteq N_\mu$.
    
\end{remark}

\begin{thm} \label{q-Cauchy-dense}
The set of all $q_{ac}$ Cauchy transforms is contained and dense in $\mr{Int} (\mu , \la )$. 
\end{thm}

\begin{proof}
First, since $q_{ac} \leq q_{\mu}$ as quadratic forms with common form domains $\mathbb{A}_{d} + N_{\lambda}$, and $\mathbb{A}_d + N_{\mu}$ is dense in $\mathbb{H}^{2}_{d} (\mu )$, it follows that $q_{ac}$ can be identified with a contractive sub-Toeplitz form on $\mathbb{H}^{2}_{d} (\mu )$ so that the set of all $q_{ac}-$Cauchy transforms is contractively contained in $\scr{H} ^+ (H _{\mu} )$.

Since $q_{ac}$ is a closeable form in $\mathbb{H}^{2}_{d} (\la )$, we have that for any $a_1, a_2 \in \mathbb{A}_{d}$
$$ q_{ac} (a_1 + N _{\la } , a_2 + N_\la ) = \ip{ \sqrt{D} a_1 + N_\la }{\sqrt{D} a_2 + N_\la } _\la,$$ for some densely--defined, closed, self-adjoint and positive semi-definite operator $D \geq 0$ with $\mathbb{A}_{d} + N_\la$ as a core for $\sqrt{D}$. Now we claim that the set of all $q_{ac}-$Cauchy transforms of the dense linear subspace $\nbdom D$ belongs to $\scr{H} ^+ (H _\la )$. Indeed, for any $y \in \nbdom D$,
\ba (\scr{C} _{q_{ac}} y) (Z) & = & \sum  Z^\om q_{ac} \left( L^\om + N_\la , y \right) \\
& = & \sum Z^\om \ip{\sqrt{D} L^\om + N_\la}{\sqrt{D} y }_\la \\
& = & \sum Z^\om \ip{L^\om + N_\la}{Dy} _\la \\
& = & (\scr{C} _\la Dy ) (Z). \ea 
On the other hand if $y \in \mathrm{Dom}(D ) \subseteq \mathrm{Dom}(\sqrt{D})$, then since $\mathbb{A}_d + N_\la$ is a core for $\sqrt{D}$, there is a sequence $a_n + N_\la$, so that $a_n + N_\la  \rightarrow y$ and $\sqrt{D} a_n + N_\la \rightarrow \sqrt{D} y$. Hence,
$$ \scr{C} _{q_{ac}} a_n + N_\la \rightarrow \scr{C} _{q_{ac}} y \in \scr{H} ^+ ( H _\la ). $$ However, since $q_{ac} \leq q_\mu$, we have that $\scr{H} ^+ (q_{ac} ) \subseteq \scr{H} ^+ ( H_\mu )$ and the embedding $e : \scr{H} ^+ (q_{ac} ) \hookrightarrow \scr{H} ^+ (H _\mu )$ is a contraction. Hence, $\mr{e} \circ \scr{C} _{q_{ac}} a_n + N_\la \in \scr{H} ^+ (H _\mu )$ is a Cauchy sequence in $\scr{H} ^+ ( H_\mu )$ and  we conclude that 
$$ \scr{C} _{q_{ac}} \nbdom D \subseteq \scr{H} ^+ ( H _\la ) \cap \scr{H} ^+ (H _\mu ) = \mr{int} (\mu , \la ). $$  Finally, since $\nbdom D$ is a core for $\sqrt{D}$, it follows that $\scr{C} _{q_{ac}} \nbdom D$ is dense in the set of all $q_{ac}-$ Cauchy transforms so that the set of all $q_{ac}-$Cauchy transforms is contained in $\mr{Int} (\mu, \la ) = \mr{int} (\mu, \la ) ^{-\| \cdot \| _{H_\mu}}$.  

Conversely, consider $\mr{int} (\mu , \la ) \subseteq \mr{Int} (\mu , \la ) \subseteq \scr{H} ^+ ( H _\mu )$. We define the embedding $\mr{e} : \mr{int} (\mu , \la ) \hookrightarrow \scr{H} ^+ (H _\la )$. Then it is easy to check that $\mr{e}$ is closeable and that it's adjoint obeys 
$$ \mathrm{e} ^* K ^\la \{ Z , y , v \}  = K ^{\mu \cap \la } \{ Z , y ,v \} = P _{\mu \cap \la} K ^\mu \{ Z , y ,v \}. $$ 
Here,
$K ^{\mu \cap \la }$ and $P _{\mu \cap \la}$ are the reproducing kernel of $\mathrm{Int} (\mu , \la )$ and the orthogonal projection of $\scr{H} ^+ (H_\mu )$ onto the $V_\mu-$co-invariant subspace $\mathrm{Int} (\mu, \la )$, respectively. Consider the closeable quadratic form, $q:= q_{T}$, ~$T=\scr{C} _\la ^* \mathrm{e} \mathrm{e} ^* \scr{C} _\la$,  with form domain $\mathbb{A}_d + N _\la$. By construction, since $T := \scr{C} _\la ^* \mathrm{e} \mathrm{e} ^* \scr{C} _\la$ is a closeable, densely--defined and positive semi-definite quadratic form in $\mathbb{H}^{2}_{d} (\la )$ with form domain $\mathbb{A}_d + N_\la$, we have that $q$ is an absolutely continuous form in $\mathbb{H}^{2}_{d} (\la )$, \emph{i.e.} $q$ is absolutely continuous with respect to $\la$. Consider the $q-$Cauchy transform of any monomial $L^{\gamma} + N_\la$,
\ba  \scr{C} _q \left(L^\ga + N _\la \right)  (Z) & = & \sum _\om Z^\om \ip{\mr{e} ^* K ^\la _\om}{\mr{e} ^* K ^\la _\ga}_{H_\mu} \\ 
& = & \sum Z^\om \ip{K ^{\mu \cap \la} _\om}{K^{\mu \cap \la} _\ga} \\
& = & \sum Z^\om K ^{\mu \cap \la} _{\om , \ga}. \ea 
On the other hand,
\ba y ^* K ^{\mu \cap \la} _\ga (Z) v & = & \ip{ K ^{\mu \cap \la} \{ Z , y , v \} }{K ^{\mu \cap \la} _\ga}_{H_\mu} \\
& = & \sum _\om y^* Z^\om v \ip{ K ^{\mu \cap \la} _\om }{K ^{\mu \cap \la} _\ga}_{H_\mu} \\
& = & y^* \left( \sum Z^\om K^{\mu \cap \la} _{\om, \ga} \right) v, \ea 
and so we conclude that  $\scr{C} _q ( L^\ga + N_\la) = K ^{\mu \cap \la} _\ga$. It follows that the set of all $q-$Cauchy transforms is equal to $\mr{Int} (\mu , \la )$ and that this in turn implies that $q_{ac} \leq q \leq q_\mu$. However, $q_{ac}$ is, by definition the maximal closeable positive semi-definite quadratic form in $\mathbb{H}^{2}_{d} (\la )$ which is bounded above by $q_\mu$. By maximality, $q = q_{ac}$ and the proof is complete. 
\end{proof}

\begin{remark}
Though in \cite[Theorem 4.7]{JM-ncld} the absolutely continuous (AC) part was extracted from a reducing subspace, but not all AC measures are arising from a reducing subspace. The reducing subspace theory works well for splitting non-Cuntz measures. For Cuntz measure and any measure in general, we have observed that the method of the proof of Proposition \ref{q-Cauchy-dense} somehow tells us which subspaces of $\mathscr{H}^{+}(H_{\mu})$ are a better candidate to define the AC part of the measure $\mu$. This subspaces somehow are related to Riesz derivative which are $\Pi_\la$-Toeplitz. We later give a modified definition which extends Jury-Martin definition, and then we prove  equivalent theorems for it.  $\blacksquare$
\end{remark}

To motivate our definition, we will give a discussion.

{\bf Discussion:} Suppose that $\gamma$ and $\lambda$ are two positive clasical measures on $\mathbb{T}$ with $\la$ being Cuntz, and $\gamma \ll \la$ in the classical sense. Since $\la$ is Cuntz, by 
Corollary \ref{H2(la)-Cuntz-nonCuntz}, we see that ${H}^{2}(\la)=L^{2}(\la)$. Let $f \geq 0$ be he Radon-Nikodym derivative of $\ga$ with respect to $\la$; i.e., $d\ga =f d\la$. Then, for any polynomials $p$ and $q$ in $\mathbb{C}[Z]$, we define the positive linear functional $\check{\ga}$ by $\check{\ga}(p^{*}q):=\langle p , q\rangle_{H^{2}(\ga)}$. So,
\begin{eqnarray*}
    \check{\ga}(p^{*}q)&:=&\langle p , q\rangle_{H^{2}(\ga)}\\
    &=& \int_{\mathbb{T}} \overline{p(z)} q(z) ~d\ga(Z) \\
    &=& \int_{\mathbb{T}} \overline{p(z)} q(z) f(z) ~ d \la(z) \\
    &=& \langle \sqrt{D} p , \sqrt{D} q\rangle_{H^{2}(\la)},
\end{eqnarray*}
where $\sqrt{D}$ is the compression of the multiplication by $\sqrt{f}$ on $L^{2}(\la)$ to $H^{2}(\la)$, i.e., $\sqrt{D}=P_{H^{2}(\la)} M_{\sqrt{f}}|_{H^{2}(\la)}$. Also, note that $\mathbb{C}[Z] \subseteq \mathrm{Dom}(\sqrt{D})$, and $D=P_{H^{2}(\la)} M_{f}|_{H^{2}(\la)}$ is the Radon-Nikodym derivative of $\ga$ with respect to $\la$. $D$ is also $\Pi_{\la}$-Toeplitz, where $\Pi_{\la}=M_{z} |_{H^{2}(\la)}$. To see this, let
\begin{eqnarray*}
    \langle \sqrt{D} \Pi_{\la} p , \sqrt{D} \Pi_{\la} q\rangle_{H^{2}(\la)}&=& \int_{\mathbb{T}} \overline{zp(z)} zq(z) f(z) ~d\la(Z) \\
    &=& \int_{\mathbb{T}} \overline{p(z)} q(z) f(z) ~ d \la(z) \\
    &=& \langle \sqrt{D} p , \sqrt{D} q\rangle_{H^{2}(\la)}.
\end{eqnarray*}
Furthermore, recall that when $\ga \ll \la$, then $\ga \ll_{R} \la$ so that $\mathrm{Int}(\ga, \la)=\overline{\mathrm{int}(\ga , \la)}^{\ga}=\mathscr{H}^{+}(H_\ga)$. In this situation, we see that  $\mathrm{int}(\ga , \la)$ is the largest subspace such that $D$, the Radon-Nikodym derivative of $\ga$ with respect to $\la$, is $\Pi_{\la}$-Toeplitz. If we look at  Proposition \ref{composition-of-closed-op}, we see that $\mathrm{int}(\ga, \la)=\mathrm{Dom}(e_{\la,\ga})$ and $D=E_{\ga,\la}^{*} E_{\ga, \la}$.

Conversely, suppose that $\mu$ and $\la$ are two positive clasical measures on $\mathbb{T}$ with $\la$ being Cuntz. Recall that $H^2(\la)=L^2 (\la)$. Let $\check{M} \subseteq \mathrm{int}(\mu ,\la) \subseteq \mathscr{H}^{+}(H_{\mu})$. Now, the embedding
$$e=e_{\la, \mu}: \check{M} \subseteq \mathscr{H}^{+}(H_{\mu}) \hookrightarrow \mathscr{H}^{+}(H_{\la}) $$
is  closed and densely defined over $\overline{\check{M}}^{\mu}$. So, the operator, $ee^{*}: \mathrm{Dom(ee^{*})} \subseteq \mathscr{H}^{+}(H_{\la}) \longrightarrow \mathscr{H}^{+}(H_{\la}) $ is densely defined and closed. Thus, by Proposition \ref{composition-of-closed-op}, the operator $D:={\mathscr{C}_{\la}}^{*} ee^{*} \mathscr{C}_{\la} $ is  closed, densely defined operator in $H^{2}(\la)$. Since $\la$ is Cuntz, constant functions can be approximated by monomials $z^n$ for all $n \in \mathbb{N}$. We have that  $\mathscr{P}_0 \subseteq \mathrm{Dom}(\sqrt{D})$ is a core for $\sqrt{D}$ , where $\mathscr{P}_0$ is the dense linear subspace generated by monomials $z^n, ~n \in \mathbb{N}$. Note that if we look at the proof of Proposition \ref{a-special-intertwiner}, then we see that $D=E_{\mu, \la}^{*} E_{\mu, \la}$; however, unlike Proposition \ref{a-special-intertwiner}  we do not assumed  $\mu \ll \la$ here.  By observing the above explanations, \emph{let $\check{M}$ be the largest subspace of $\mathrm{int}(\mu ,\la) $ such that  $D$ is $\Pi_{\la}$-Toeplitz}. Then, we can define
$$\check{\ga}(p^{*}q):=\langle \sqrt{D} p , \sqrt{D} q \rangle_{H^{2}(\la)} =\int_{\mathbb{T}} \overline{\sqrt{D}p(z)} \sqrt{D}q(z)   ~d\la(z).$$

Obviously, $\check{\ga}(p^{*})=\overline{\check{\ga}(p)}$ and polynomials are dense in the space of continuous functions $C(\mathbb{T})$; hence, $\check{\ga} $ has a unique  extension to a positive linear functional on $C(\mathbb{T})$. Thus, by Riesz-Markov theorem it corresponds to a positive measure $\ga$ by
$$\check{\ga}(p^{*}q)=\int_{\mathbb{T}} \overline{p(z)} q(z) ~d\ga(z). $$
By the latter two formulas, we see that
\begin{equation} \label{two-integrals-ac}
  \int_{\mathbb{T}} \overline{p(z)} q(z) ~d\ga(z)=\int_{\mathbb{T}} \overline{\sqrt{D}p(z)} \sqrt{D}q(z)   ~d\la(z).  
\end{equation}
We claim that $D=M_f$, where $M_f$ is the multiplication operator by an $f \in L^1(\la)$. Since $\la$ is Cuntz, we have $H^2(\la)=L^2(\la)$, so one can use the spectral theorem for unbounded operators \cite[Theorem VIII.4]{RnS1} to guess that $D=M_f$ for some positive measurable function $f$. We will give an alternative proof which is more streamlined with our notations. Let $M_z$ be the multiplication operator by the independent variable. Since $H^2(\la)=L^2(\la)$, both $D$ and $M_z$ live on $L^2(\la)$. In addition, on polynomials $\Pi_{\la}$ acts by $M_z$, so $\mathrm{Dom}(\sqrt{D})$ is $\Pi_\la$-invariant. Since $D$ is $\Pi_{\la}$-Toeplitz, we have that
$$\langle \sqrt{D} M_Z (p) , \sqrt{D} M_Z (q) \rangle_{H^{2}(\la)}=\langle \sqrt{D} p , \sqrt{D} q \rangle_{H^{2}(\la)}.  $$
On the other hand, 
\begin{eqnarray*}
    \langle M_Z \sqrt{D} p , M_Z \sqrt{D} q \rangle_{H^{2}(\la)} &=& \int_{\mathbb{T}} \overline{z\sqrt{D}p(z)} z\sqrt{D}q(z)   ~d\la(z) \\
    &=& \int_{\mathbb{T}} \overline{\sqrt{D}p(z)} \bar{z}z\sqrt{D}q(z)   ~d\la(z)\\
    &=&\int_{\mathbb{T}} \overline{\sqrt{D}p(z)} \sqrt{D}q(z)   ~d\la(z) ,\quad \text{since}~  \bar{z}z=|z|^2=1\\
    &=& \ga(p^*q)\\
    &=&\langle \sqrt{D} p , \sqrt{D} q \rangle_{H^{2}(\la)}.
\end{eqnarray*}
Thus,
$$\sqrt{D}M_Z=M_Z\sqrt{D}, \quad \text{on} ~ \mathscr{P}_0; $$
which means that $\sqrt{D} $ is affiliated with  the commutant of the range of the following representation of $L^{\infty}(\la)$:
\begin{eqnarray*}
    \phi: L^{\infty}(\la) &\longrightarrow &\mathscr{B}(L^2(\la)) \\
    \phi(f)&=&M_f.
\end{eqnarray*}
 However, $\left(\phi(L^{\infty}(\la)\right)'\cong L^{\infty}(\la)$ and all affiliated unbounded operators with $L^{\infty}(\la)$ arise from multiplication by $L^{2}(\la)$ functions. So, $\sqrt{D}=M_{g}$ for some positive function $g \in L^{2}(\la)$, and consequently $D=M_f$ where $f=g^2 \in L^1(\la)$. Thus, from Equation \ref{two-integrals-ac} we see that $d\ga=f d\la$, and consequently $\ga \ll \la$. We remark here that in the non-commutative case, we still have $\sqrt{D}=M_{g}$ for some positive function $g \in L^{2}(\la)$, and $D=M_f$ where $f=g^2$; however, we do not require $f \in L^1(\la)$ as we do not define $L^1$ in the NC case.

We claim that $M=\overline{\check{M}}^{\mu}={\mathscr{H}}^{+}(H_{\ga})$. To do so, it is enough to show that $M$ is the range of the $q$-Cauchy transform of $D$; i.e., $M=\mathrm{Ran}(\mathscr{C}_{q})$, where $q(p,q)=\langle \sqrt{D}p, \sqrt{D}q \rangle_{\la}= \ga(p^*q)$. Actually, this would mean that $M=\mathrm{Ran}(\mathscr{C}_{\ga})$; however, the proof shows why we need to work with form instead of measure. Because $D=\mathscr{C}_{\la}^{*}ee^* \mathscr{C}_{\la}$, by polar decomposition there is a unitary operator $U$ such that $\sqrt{D}=Ue^* \mathscr{C}_{\la}$. So, for the inner product we see that
$$\langle \sqrt{D} (L^\omega+N_\la), \sqrt{D}(L^\alpha+N_\la) \rangle_{\la}=\langle e^* \mathscr{C}_{\la} (L^\omega+N_\la), e^* \mathscr{C}_{\la} (L^\alpha+N_\la) \rangle_{\la}$$
Now, consider the $q$-Cauchy transform of the monomial $L^\alpha+N_\la$,
\begin{eqnarray*}
\mathscr{C}_\ga(L^\alpha+N_\la)&=&\mathscr{C}_{q} (L^\alpha+N_\la)(Z)\\& = & \sum  Z^\om q \left( L^\om + N_\la , L^\alpha +N_\la \right) \\
& = & \sum Z^\om \ip{\sqrt{D} L^\om + N_\la}{\sqrt{D} L^\alpha +N_\la }_\la \\
& = & \sum Z^\om \ip{e^* \mathscr{C}_{\la} L^\om + N_\la}{e^* \mathscr{C}_{\la} L^\alpha +N_\la }_\la \\
& = & \sum Z^\om \ip{e^* K^{\la}_{\om}}{e^* K^{\la}_{\alpha} }_\la \\
& = & \sum Z^\om \ip{e^* K^{\la}_{\om}}{ K^{M}_{\alpha} }_M \\
&=& K_{\alpha}^{M}.
\end{eqnarray*}
Hence, we see that $M=\mathrm{Ran}(\mathscr{C}_{\ga})=H^{+}(H_\ga)$. $\blacksquare$

The above discussion is the main motivation for the modified definition of absolute continuity given below. 

\begin{defn} \label{RK-AC-part-general}
    Let $\mu$ and $\la$ be two positive NC measures. We say that:
    \begin{itemize}
    \item $\mu$ is \emph{absolutely continuous with respect to $\la$ in the reproducing kernel-Toeplitz sense}, in notations $\mu \ll_{RT} \lambda$, if for the embedding $e: \mathrm{int}(\mu , \la) \hookrightarrow \mathscr{H}^{+}(H_{\la}) $, the closed operator $D:={\mathscr{C}_{\la}}^{*} ee^{*} \mathscr{C}_{\la}=E_{\mu, \la}^{*}E_{\mu,\la} $ is  non-zero and $\Pi_\la$-Toeplitz. The set of all positive NC measures that are absolutely continuous with respect to $\la$ in the repreducing kernel Hilbert space-Toeplitz sense is denoted by $AC_{RT}[\la]$.
\item $\mu$ is \emph{singular with respect to $\la$ in the reproducing kernel-Toeplitz sense}, in notations $\mu \perp_{RT} \lambda$, if for the embedding $e: \mathrm{int}(\mu , \la)  \hookrightarrow \mathscr{H}^{+}(H_{\la}) $, either $e=0$, or $e$ does not have any restriction such that the closed operator $D:={\mathscr{C}_{\la}}^{*} ee^{*} \mathscr{C}_{\la}=E_{\mu, \la}^{*}E_{\mu,\la} $ is  $\Pi_\la$-Toeplitz (on any restricted subspace).  The set of all positive NC measures that are singular with respect to $\la$ in the repreducing kernel Hilbert space sense is denoted by $SG_{RT}[\la]$.
\item the decomposition $\mu=\mu_{1}+\mu_{2}$ with respect to $\la$ is called a Lebesgue RKHS-Toelitz  decomposition if $\mu_{1} \ll_{RT} \la$ and $\mu_{2} \perp_{RT} \la$. We may write this decomposition by $\mu=\mu^{RT}_{1}+\mu^{RT}_{2}$, where the superscript $RT$ pints to the RKHS-Toeplitz decomposition sense.
\end{itemize}
\end{defn}

\begin{thm} \label{LD-RKHS-Cuntz}
Let $\mu$ and $\la$ be two positive NC measure with $\la$ being Cuntz. Let $\mu_{ac}$ be given by Definition \ref{RK-AC-part-general}, and $\mu_s$ be defined by 
$\mu_{s}=\mu-\mu_{ac}$
Then, $\mu=\mu_{ac}+\mu_{s}$ is the Lebesgue RKHS-T decomposition of $\mu$ with respect to $\lambda$. 
\end{thm}
\bp
 Suppose that  $\check{M} \subseteq \mathrm{int}(\mu ,\la)$  is the largest subspace so that for the embedding 
$e: \check{M} \subseteq \mathscr{H}^{+}(H_{\mu}) \hookrightarrow \mathscr{H}^{+}(H_{\la}) $, the operator $D:={\mathscr{C}_{\la}}^{*} ee^{*} \mathscr{C}_{\la} $ is $\Pi_\la$-Toeplitz. In this situation, we call $D$ the Radon-Nikodym derivative of $\mu$ with respect to $\la$. We claim there is a measure $\ga$ such that $M=\overline{\check{M}}^{\mu}={\mathscr{H}}^{+}(H_{\ga})$ and we call it the RKHS-T absolutely continuous part of $\mu$ with respect to $\la$ and is denoted by $\mu_{ac}$. Note that $D$ is densly defined and closed on the Hilbert space ${\mathscr{H}}^{+}(H_{\ga})$ and is $\Pi_\la$-Toeplitz. We show that $M$ is the range of the $q$-Cauchy transform of $D$; i.e., $M=\mathrm{Ran}(\mathscr{C}_{q})$, where $q(p,q)=\langle \sqrt{D}p, \sqrt{D}q \rangle_{\la}= \ga(p^*q)$ is our $\Pi_\la$ Toeplitz form which induces the measure $\ga$ by Remark \ref{Toeplitz-form-and-measure}. Actually, this would mean that $M=\mathrm{Ran}(\mathscr{C}_{\ga})$; however, the proof shows why we need to work with form instead of measure. Because $D=\mathscr{C}_{\la}^{*}ee^* \mathscr{C}_{\la}$, by polar decomposition there is a unitary operator $U$ such that $\sqrt{D}=Ue^* \mathscr{C}_{\la}$. So, for the inner product we see that
$$\langle \sqrt{D} (L^\omega+N_\la), \sqrt{D}(L^\alpha+N_\la) \rangle_{\la}=\langle e^* \mathscr{C}_{\la} (L^\omega+N_\la), e^* \mathscr{C}_{\la} (L^\alpha+N_\la) \rangle_{\la}$$
Now, consider the $q$-Cauchy transform of the monomial $L^\alpha+N_\la$,
\begin{eqnarray*}
   \mathscr{C}_\ga(L^\alpha+N_\la)&=&\mathscr{C}_{q} (L^\alpha+N_\la)(Z)\\& = & \sum  Z^\om q \left( L^\om + N_\la , L^\alpha +N_\la \right) \\
& = & \sum Z^\om \ip{\sqrt{D} L^\om + N_\la}{\sqrt{D} L^\alpha +N_\la }_\la \\
& = & \sum Z^\om \ip{e^* \mathscr{C}_{\la} L^\om + N_\la}{e^* \mathscr{C}_{\la} L^\alpha +N_\la }_\la \\
& = & \sum Z^\om \ip{e^* K^{\la}_{\om}}{e^* K^{\la}_{\alpha} }_\la \\
& = & \sum Z^\om \ip{e^* K^{\la}_{\om}}{ K^{M}_{\alpha} }_M \\
&=& K_{\alpha}^{M}.
\end{eqnarray*}
Hence, we see that $M=\mathrm{Ran}(\mathscr{C}_{\ga})=H^{+}(H_\ga)$.

The remainder $\mu_{s}=\mu-\mu_{ac}$ is the RKHS-T singular part, and the corresponding singular part of $e$ for the derivative purpose lives on $\check{N}= \mathrm{int}(\mu ,\la) \setminus \check{M} $. Also, our decomposition is the maximal one since $\check{M}$ is the largest one.
\ep

\begin{remark}\label{Recovering-old-RK-AC-defn}
In the above theorems,  we should not expect   that $\mathrm{int}(\mu_{s}, \la)=0$. These phenomenon happens when $\la$ is non-Cuntz.
  In fact, our new definition of absolute continuity is more general than the previous definition for absolute continuous part in the non-Cuntz case. That is,  if $\la$ is non-Cuntz and $\mu \ll_{R} \la$, then $\mu \ll_{RT} \la$. When $\la$ is non-Cuntz and $\mu \ll_{R} \la$, we will show that $\mathrm{int}(\mu , \la)$ is indeed the largest subspace such that the Riesz derivative  $D_\la:={\mathscr{C}_{\la}}^{*} ee^{*} \mathscr{C}_{\la} $ is $\Pi_\la$-Toeplitz. For  non-Cuntz $\la$, by Proposition \ref{reducing-intersection} the intersection space $\mathrm{Int}(\mu , \la)$ is $V_\mu$ reducing. In addition, $\mu \ll_{R} \la$ implies that the intersection space is dense in ${\mathscr{H}}^{+}(H_{\mu})$. Also, from Theorem \ref{form-decomposition-nonCE} we see that the Simon derivative $D_\sigma=Q_{ac}-E_{\la,\sigma}^{*}E_{\la, \sigma}$ is $\Pi_\sigma$-Toeplitz. Recall that the left regular representations act on their natural domains by left multiplication. So, from Remark \ref{types-of-derivatives} and coincidence of various forms, we see that the form $q_{\mu}^{\la}$ must be  $\Pi_\la$-Toeplitz. This alongside with densely defindness of $D_\la$, implies that the Riesz derivative $D_\la$ must be a closed $\Pi_\la$-Toeplitz operator. \\
 Conversely, if $\la$ is non-Cuntz and $\mu \ll_{RT} \la$, we show that $\mu \ll_{R} \la$. If the latter does not happen, then by the method of decomposition of Theorem \ref{rkhs-non-CE}, we see that $\mu = \mu_{ac}^{R} + \mu_{s}^{R}$ with $\mu_s \neq 0$,
 $${\mathscr{H}}^{+}(H_{\mu_{ac}}):= {\overline{{\mathscr{H}}^{+}(H_{\mu}) \cap {\mathscr{H}}^{+}(H_{\lambda}) }}^{H^{+}(H_{\mu})}=\mathrm{Int}(\mu, \la)$$
and 
$${\mathscr{H}}^{+}(H_{\mu_{s}}):={\mathscr{H}}^{+}(H_{\mu}) \ominus {\mathscr{H}}^{+}(H_{\mu_{ac}}).$$
 
Thus, $\mu_{ac}^{R} < \mu$. By NC-Aronszajn theorem \cite[Theorem 4.1]{JM-ncld}, we see that $ {\mathscr{H}}^{+}(H_{\mu_{ac}}) \subsetneqq  {\mathscr{H}}^{+}(H_{\mu})$ and $ \mathrm{int}(\mu_{ac} , \la) \subsetneqq \mathrm{int}(\mu , \la)  $, so $M=\mathrm{int}(\mu , \la)\setminus \mathrm{int}(\mu_{ac} , \la) \neq \varnothing$. Also, $\mathrm{int}(\mu_s , \la)=0$. However, the remainder $\mu_s$ shows that $e$ has a restriction to  subspace $M$ of $\mathrm{int}(\mu , \la)$ such that $D_\la$ is not $\Pi_\la$-Toeplitz; contradicting $\mu \ll_{RT} \la$. Thus as we can see, for non-Cuntz measures, the AC-part in both RT and RKHS senses is the same. We will show that it is also true for AC-parts of Cuntz measures; however, the singular parts might not be the same, see Example \ref{classical-sing-not-RKHS-Sing}. That is the definition of singularity based on RKHS has limitations as previously explained. $\blacksquare$ 
 \end{remark}

\begin{thm} \label{form-rkhs-T-equiv-CE}
    Let $\mu$ and $\la$ be two positive NC measures on $\mathcal{A}_d$, and let $\la$ be Cuntz. Then,
\begin{enumerate}[(i)]
\item  $\mu \ll_{RT} \la $ if and only if $\mu \ll_{F} \la$. 
\item  $\mu \perp_{RT} \la$ if and only if $\mu \perp_{F} \la$. 
\end{enumerate}
\end{thm}
\begin{proof}
   (i) Suppose that $\mu \ll_{RT} \la$; we must show that $q_{\mu}=q_{\mu}^{\la}$ is well-defined and closable in ${\mathbb{H}}_{d}^{2}(\lambda)$. However, this claim is because of the following identity and closedness of $D$
   \begin{eqnarray*}
  q_{D} (a+N_{\la}, b+N_{\la})&=&q_{E^*E} (a+N_{\la}, b+N_{\la})\\
  &=&\langle  a+N_{\la}, E^*E( b+N_{\la} )\rangle_{\la}\\
  &=&\langle E( a+N_{\la}), E(b+N_{\la} )\rangle_{\mu}\\
  &=&\langle  a+N_{\mu}, b+N_{\mu} \rangle_{\mu}\\
  &=& q_{\mu}^{\la} (a+N_{\la}, b+N_{\la})
\end{eqnarray*}
      
   Conversely, Suppose $\mu \ll_{F} \la$ holds, so by definition $N_{\mu} \supseteq N_{\la}$, and hence  $q_{\mu}=q_{\mu}^{\la}$ is a well-defined and  closable form on $\mathbb{H}_{d}^{2}(\lambda)$. Hence, we can find a positive closed unbounded operator $D$ on $\mathbb{H}_{d}^{2}(\lambda)$ such that 
   \begin{equation*}
       q_{\mu}(a+N_{\lambda}, b+N_{\lambda})=\langle \sqrt{D}(a+N_{\lambda}), \sqrt{D}(a+N_{\lambda}) \rangle.
   \end{equation*}
   By Corollary \ref{affiliation}, $D \sim vN(\Pi_{\lambda})^{'}$ so that $D$ commutes with $\Pi_{\lambda}$ on the domain of $D$, hence $D$ is $\Pi_{\lambda}$-Toeplitz, i.e., $\mu \ll_{RT} \la$.

(ii) The forward direction is like the previous part. If $D=0$, then $q_{\mu}^{\la}=0$ which is obviously singular. Also, if $D$  does not have a closed and $\Pi_{\la}$-Toeplitz restriction to any subspace, then by the equality $q_D = q_{\mu}^{\la}$, the form $q_{\mu}^{\la}$ cannot be closeable.

For the converse, let $\mu \perp_{F} \la$ and $\la$ be Cuntz . Note that $N_\mu \supseteq N_\la$, and $q_{\mu}=q_{\mu}^{\la} $ is well-defined and is singular by assumption, i.e., $q_{\mu}=(q_{\mu})_{s}$. Define $e$ and let  $D \neq 0$. Obviously, $q_\mu = q_D$. Now, suppose to the contrary that $D$ has a restriction to a subspace on which it is closed and $\Pi_{\la}$-Toeplitz. Call this restricted operator $G$. This Toeplitz operator gives rise to a measure $\ga$ like Remark \ref{Toeplitz-form-and-measure}. Now, by the method of the previous part, we can define a closeable form $q=q_{G}=q_{\ga}^{\la}$. Since $G$ is the restriction of $D$, this means that $q=q_{G} \leq q_\mu$ contradicting the singularity of $q_\mu$.

\end{proof}

\begin{cor} \label{R-AC=RT-AC}
    Let $\mu$ and $\la$ be any two NC measures. Then, $\mu \ll_{R} \la$ if and only if $\mu \ll_{RT} \la$
\end{cor}
\bp
This is because of Theorem \ref{R-AC-Cuntz=F-Ac-Cuntz}, Theorem \ref{form-rkhs-T-equiv-CE}, and Remark \ref{Recovering-old-RK-AC-defn}.
\ep

\begin{remark} \label{GK-MN-relation}
Suppose we want to decompose the NC measures $\mu$ against the NC measure $\la$. One can extend $\mu$ and $\la$ to the Cuntz-Toeplitz  C*-algebra $\mathcal{E}_d$ by Krein's or Arveson's extension theorem, see \cite[Chapters 2 and 7]{Paulsen-cb}.  When the splitting measure $\la$ is Cuntz, it has a unique extension to  $\mathcal{E}_d$ \cite[Proposition 5]{JMT-NCFM}. Now the Gheondea-Kavruk method gives a Lebesgue decomposition which can be restricted back to the operator system $\mathcal{A}_d$, . However, at this point it is not obvious to us that this decomposition coincides with ours because we are not sure that the form domains in Gheondea-Kavruk theory coincides with ours, i.e., form domains might be different. Besides, for splitting non-Cuntz measures the problem of non-uniquness of $\la$ prevents us from using Gheondea-Kavruk decomposition. This is while the RKHS method still works. So, we prefer to implement the decomposition using our RKHS language, which works for all NC measures on the disc operator system. $\blacksquare$
\end{remark}

\begin{thm}
    Let $\la$ be a Cuntz measure. Then the set $SG_{RT}[\la]$ has  hereditary property. 
\end{thm}
\bp
 suppose that $\nu \leq \mu$ and $\mu \perp_{RT} \la$. We will show that the form $q_{\nu}^{\la}$ cannot majorize any closeable form except the zero one. Suppose to the contrary that it majorizes the non-zero closeable form $p$. Obviously, $p \leq q_{\nu}^{\la} \leq q_{\mu}^{\la}$, so $q_{\mu}^{\la}$ majorizes a non-zero closable form. However, this means that $\mu$ cannot be singular with respect to $ \la$ in the form sense, so by the equivalence of form and RT-sense, we see that $\mu$ cannot be singular with respect to $ \la$ in the RT- sense, which is false.

\ep

    The content of Corollary  \ref{R-AC=RT-AC} is that the AC part in RKHS-T decomposition is the same as the AC part in  RKHS decomposition.  However, the same symmetry cannot exist for the singular parts as the following theorem and the next example show.

\begin{thm} \label{RK-Sing-implies-Sing}
    Let $\mu$ and $\la$ be two positive classical measures on the unit circle $\mathbb{T}$.  If $\mu \perp_{R} \la$, then $\mu \perp \la$. 
\end{thm}
\bp
By the definition of RK-Sing, we notice that the only measure $\nu$ which is less than both $\mu$ and $\la$ is the zero measure. Since if $\nu$ is not zero, then by $\nu \leq \mu, ~\la$, we see that $\nu$ is Classically absolutely continuous with respect to $\mu$ and $\la$. However, we have shown that classical-AC and RK-AC are the same. This, means that $\mathrm{int}(\mu, \la)$ is non-empty, contradicting the assumpton that $\mu \perp_R \la$. Thus, the joint minimum of $\mu$ and $\la$ is zero. Hence, by \cite[Theorem 37.5]{Aliprantis}, $\mu \perp \la$ in the classical sense.
\ep
\begin{eg} \label{classical-sing-not-RKHS-Sing}
    The converse of the above theorem is wrong. That is, there are two classical measures $\mu$ and $\la$  on the unit circle $\mathbb{T}$ such that $\mu \perp \la$ but $\mu \not\perp_{R} \la$, and it should be deduced from the counter  Example  \ref{non-example} by $\mu=m_1$ and $\la=m_2$. Here we outline a  direct way to show that $\mathrm{int}(m_1, m_2) \neq 0$. In fact, one can single out an element in this intersection space. For another example see \cite{BMN}. Note $V_{m_i}$'s are unitary operator (since $m_i$'s are Cuntz measures), their ranges are subspaces of ${\mathscr{H}}^{+}(H_{m_i})$, so  $\mathrm{Ran}(V_{m_1}) \cap \mathrm{Ran}(V_{m_2}) \subseteq \mathrm{int}(m_1, m_2)$. We try to find an element in  $\mathrm{Ran}(V_{m_1}) \cap \mathrm{Ran}(V_{m_2})$. One  can see that 
\begin{eqnarray*}
    H_{m_i}(z)&=&m_i (1) + 2 \sum_{n=1}^{\infty} \overline{m_i(\zeta ^n) }z^n, \qquad i=1,2 \\
    m_i(\zeta ^n)&=&\int \zeta^n d~m_i(\zeta)= \begin{cases} 
          (-1)^{i} \frac{2}{n+1}, & n=even \\
          
          0, &  n=odd
       \end{cases}, \qquad i=1,2 
\end{eqnarray*}
So,
\begin{eqnarray*}
    H_{m_i}(z)&=&m_i (1) + 2 \sum_{n=1}^{\infty} \overline{m_i(\zeta ^n) }z^n=\pi + (-1)^{i} 4 \sum_{k=1}^{\infty} \frac{z^{2k}}{2k+1}.
\end{eqnarray*}
Now, if $K_w(z)=K(z,w)=\frac{1}{1-z\overline{w}}$ is the Szego kernel, then
\begin{equation*}
\prescript{m_i}{}{K}_w(z)=\frac{1}{2}[H_{m_i}(z)+ \overline{H_{m_i}(w)}]K_{w}(z),
\end{equation*}
and for $i=1, 2$ at $w=0$, we have 
\begin{equation*}
\prescript{m_i}{}{K}_0(z)=\frac{1}{2}[H_{m_i}(z)+ \overline{H_{m_i}(0)}] 
= \pi + (-1)^i2\sum_{k=1}^{\infty} \frac{z^{2k}}{2k+1} 
= \prescript{m_i}{}{K}_0(0) + (-1)^i2\sum_{k=1}^{\infty} \frac{z^{2k}}{2k+1}.
\end{equation*}
Thus,
\begin{equation*}
  \prescript{m_i}{}{K}_0(z) -  \prescript{m_i}{}{K}_0(0) =   (-1)^i2\sum_{k=1}^{\infty} \frac{z^{2k}}{2k+1} =\pm f(z),
\end{equation*}
where $f(z)=2\sum_{k=1}^{\infty} \frac{z^{2k}}{2k+1}$. However, by \cite[Section 3.4, p.22]{JM-ncFatou}, we know that 
$$ \prescript{m_i}{}{K}_0(z) -  \prescript{m_i}{}{K}_0(0) \in \mathrm{Ran}(V_{m_i}); $$
also $\mathrm{Ran}(V_{m_i})$ is a linear subspace so $\pm f \in \mathrm{Ran}(V_{m_i})$. This means that 
$$ f \in \mathrm{Ran}(V_{m_1}) \cap \mathrm{Ran}(V_{m_2}) \subseteq \mathrm{int}(m_1, m_2),$$
and since $f \neq 0$, we see that $\mathrm{int}(m_1, m_2) \neq 0$. $\blacksquare$
\end{eg}
\begin{remark}
The counter Example \ref{non-example} does not fall within the scope of Jury-Martin initial decomposition since it has something to do with their construction of singular part through the direct sum decomposition. This is because they use the direct sum decomposition for the Aronsjazn sum of the RKHS corresponding to the AC and Sing parts.  We should mention that singularity with respect to the splitting Lebesgue measure $m$ is equivalent to having zero derivative, see \cite[Proposition 3.30]{Folland}. So, this condition, translated into the language of RKHS, is equivalent to zero intersection, and hence providing direct sums. However, for general splitting measure $\la$ such a theory of differentiation like \cite[Section 3.4]{Folland} does not work, so we do not know that the Radon-Nikodym derivatives is zero in this case. As a result, the intersection space might not be zero. This is why we cannot directly generalize the initial definition of Jury-Martin \cite{JM-ncld}. We think that singularity is an antithesis to absolute continuity in every subspace, so Definition \ref{RK-AC-part-general} is plausible. Finally, we should also comment that our operator theoretic and operator algebraic language allow the existence of the non-zero derivative  $D_{\la}(\mu)=E_{\mu, \la}^{*}E_{\mu, \la}$  of singular measure $\mu$ against Cuntz measure $\la$  so that $D_{\la}(\mu)$ is  not $\Pi_\la$-Toeplitz. The existence of such non-zero derivative through the language of operator algebra is a bit counter intuitive to the classical measure theory. $\blacksquare$
\end{remark}

By \cite[Theorem 4.5]{JM-ncld}, $\mathrm{Int}(\mu, \la)$ is always $V_\mu$-co-invariant. The following proposition reveals when it is $V_\mu$-invariant.
\begin{prop} \label{invariance-of-intersection-space}
    Let $\mu$ and $\la$ be two positive NC measures on $\mathcal{A}_d$. Then, $ \mathrm{Int}(\mu, \la)$ is $V_\mu$-invariant if and only if either $\la$ is non-Cunts, or 
    $$(V_{\mu, j}f)(0)=(V_{\la, j}f)(0), \quad \forall f \in  \mathrm{int}(\mu, \la),\quad j=1, \cdots, d  $$
\end{prop} 
\bp
The proof of the reverse direction is obvious from \cite[Proposition 4.8]{JM-ncld} and \ref{reducing-intersection}. For the forward implication suppose that $ \mathrm{Int}(\mu, \la)$ is $V_\mu$-invariant  and $\la$ is Cuntz. For any $\forall f \in  \mathrm{int}(\mu, \la)$ and any $Z\in \mathbb{B}_{n}^{d}$, $n \in \mathbb{N}$, we have
\begin{eqnarray*}
    (V_{\mu}f)(Z)=(V_{\la}f)(Z)+I_n [(V_{\mu}f)(0_n)-(V_{\mu}f)(0_n)]
\end{eqnarray*}
Now, $V_{\mu}f \in \mathrm{int}(\mu, \la) \subseteq \mathscr{H}^{+}(H_\la)$ and $V_{\la}f \in \mathscr{H}^{+}(H_\la) $. So, by the above equation we must have that 
$$I_n [(V_{\mu}f)(0_n)-(V_{\mu}f)(0_n)] \in \mathscr{H}^{+}(H_\la).$$
However, $\la$ is Cuntz and by \cite[Theorem 6.4]{JM-freeCE}, $\mathscr{H}^{+}(H_\la)$ cannot contain non-zero constant functions. Hence, $(V_{\mu}f)(0_n)=(V_{\mu}f)(0_n)$, and this is obviously true on each coordinates of the row isometries.
\ep

Again, this Proposition somehow indicates that we should not expect in general that the intersection space be reducing. This condition is a bit difficult to satisfy for two different Cuntz measures. Just like example \ref{non-example}, one might find another counter example by taking $\la$ as Cuntz, and finding another Cuntz measure $\mu$ such that  by basic properties of $\mu$ the intersection space cannot be reducing.

\begin{remark}
    Let $\mu$ and $\la$ be two positive NC measures. We have seen that if $\la$ is non-Cuntz, then Propositin \ref{reducing-intersection} shows that $\mathrm{Int}(\mu, \la)$ is $V_\mu$ reducing. However, if $\la$ is non-Cuntz, Example \ref{non-example} indicates that there is no guarantee that this intersection be $V_\mu$-reducing. Nevertheless, if $\la$ is Cuntz, and 
    $$(V_{\mu, j}f)(0)=(V_{\la, j}f)(0), \quad \forall f \in  \mathrm{int}(\mu, \la),\quad j=1, \cdots, d  $$
    then by Proposition \ref{invariance-of-intersection-space}, then  $ \mathrm{Int}(\mu, \la)$ is $V_\mu$-reducing. In this case, let's say something about the RK-Lebesgue decomposition of $\mu $ with respect $\la$. Obviously, for any Cuntz measure $\la$, the intersection space $\mathrm{Int}(\la, \la)$ is $V_\la$-reducing.
\end{remark}
\begin{prop}
   Let $\la$ be a Cuntz measure, $\mu$ be any positive NC measure, and  $ \mathrm{Int}(\mu, \la)$ be $V_\mu$-reducing. Then the RK-AC part of $\mu$ with respect to $\la$ can be computed like Theorem \ref{rkhs-non-CE}, and in this case $\mu_{ac}$ is pure Cuntz.
\end{prop}
\bp
 Note that $D={\mathscr{C}_{\la}}^{*} ee^{*} \mathscr{C}_{\la}=E_{\mu,\la}^{*}E_{\mu, \la}$ defines the Radon-Nikodym derivative of $\mu$ with respect to $\la$, and by Corollary \ref{affiliation}, $D$ is affiliated with the von-Neumann algebra $vN(\Pi_\la)^{'}$. That is $D$ commutes with $\Pi_\la$ on the domain of $D$. However, $\Pi_\la$ is unitary, thus we deduce that $D$ is $\Pi_\la$-Topelitz. However, recall that that by definition, this intersection is the largest subspace such that $D={\mathscr{C}_{\la}}^{*} ee^{*} \mathscr{C}_{\la} $ is $\Pi_\la$-Toeplitz. On the other hand, the method of proof of Theorem \ref{rkhs-non-CE} for RK-Lebesgue decomposition is just using the reducing subspace assumption, so we can use it to detect the AC measure like \cite[Theorem 4.7]{JM-ncld}. By the maximal assumption on $M$ in the Definition \ref{RK-AC-part-general} and the fact that $\overline{M}^{\mu} \subseteq \mathrm{Int}(\mu, \la)  $, we see that these two RK-AC measures coincide. So, we have just one $\mu_{ac}$. Note that $\mu_{ac}$ is Cuntz measure since by our construction $\mu_{ac} \ll_{R} \la$, so 
 \begin{eqnarray*}
    E_{\mu_{ac}, \la}:  \mathbb{A}_d +N_{\la} \subseteq{\mathbb{H}}^{2}_{d}(\la) &\longrightarrow& {\mathbb{H}}^{2}_{d}(\mu_{ac})\\
    a+N_{\la} &\longrightarrow& a+N_{\mu_{ac}}
\end{eqnarray*}
 must be a well-defined closed intertwiner between $\Pi_\la$ and $\Pi_{\mu_{ac}}$. For this to make sense, we must have that $N_{\mu_{ac}} \supseteq N_\la \neq 0$; hence,
  $N_{\mu_{ac}} \neq 0$. Therefore, by Theorem \ref{left-ideal-non-Cuntz=0}, $\mu_{ac}$ cannot have any type $L$ part, so it is purely Cuntz.  Note that in the classical case this is obvious since $D$ is $\Pi_\la$-Toeplitz, we can see that $D=M_f$; thus, $d\mu_{ac}=fd\la$, which proves that $\mu_{ac}$ must be Cuntz. 
\ep
\large { {\bf Acknowledgments:}} The author would like to thank Professor Robert Martin for his support and encouragement.
\bibliographystyle{unsrtnat}

\Addresses

\end{document}